\documentclass[12pt,reqno]{amsart}
\input{preamble}
\usepackage[margin=1in,marginpar=.8in]{geometry}
 
\DeclareMathOperator*{\ucolim}{colim}
\newtheorem{specialcase}[thm]{Special case}
\usepackage{caption}
\usepackage{subcaption}
\newcommand{\cp}[2]{\left( \begin{matrix} #1 \\ \odot \\ #2 \end{matrix} \right)}

\newcommand*{\bigboxtimes}{\mathchoice{\mathop{\raisebox{-0.25em}{\scalebox{1.35}{$\boxtimes$}}}}{\boxtimes}{{\textstyle \boxtimes}}{\scriptstyle \boxtimes}}
\DeclareMathOperator*{\dbigboxtimes}{{\raisebox{-0.25em}{\scalebox{1.35}{$\boxtimes$}}}}

\newcommand{\Glue}{\mathrm{Glue}}
\DeclareMathOperator{\Aut}{Aut}
\DeclareMathOperator{\mVert}{Vert}
\DeclareMathOperator{\Inv}{Inv}
\DeclareMathOperator{\Pair}{Pair}

\newcommand{\rop}{\mathrm{op}}
\newcommand{\rCospan}{\mathrm{Cospan}}
\renewcommand*{\to}{\mathchoice{\longrightarrow}{\rightarrow}{\rightarrow}{\rightarrow}}
\newcommand*{\To}{\mathchoice{\Longrightarrow}{\Rightarrow}{\Rightarrow}{\Rightarrow}}
\let\oldmapsto\mapsto
\renewcommand*{\mapsto}{\mathchoice{\longmapsto}{\oldmapsto}{\oldmapsto}{\oldmapsto}}
\newcommand{\simTo}{\overset{\smash{\raisebox{-0.25em}{$\sim$}}\,}{\To}}

\title{Plethysm Products, Element-- and  Plus Constructions}
\date{}

\author{Ralph M. Kaufmann}
\address{Purdue University Department of Mathematics, and Department of Physics \& Astronomy, West Lafayette, IN 47907, USA}
\email{rkaufman@purdue.edu}

\author{Michael Monaco}
\address{Purdue University Department of Mathematics, West Lafayette, IN 47907, USA}
\email{monacom@purdue.edu}

\begin{document}

\maketitle

\begin{abstract}
Motivated by viewing categories as bimodule monoids over their isomorphism groupoids,
we construct monoidal structures called plethysm products on three levels: that is for bimodules, relative bimodules and factorizable bimodules. For the bimodules we work in the general setting of actions by categories. We give a comprehensive theory linking these levels to each other as well as to Grothendieck element constructions, indexed enrichments, decorations and algebras.

Specializing to groupoid actions leads to applications including the plus construction. In this setting, the third level encompasses the known constructions of Baez-Dolan and its generalizations, as we prove. One new result is that the plus construction can also be realized as an element construction compatible with monoidal structures that we define.
This allows us to prove a commutativity between element and plus constructions, a special case of which was announced earlier. 
Specializing the results on the third level yields a criterion for when a  
definition of operad--like structure as a plethysm monoid ---as exemplified by operads--- is possible.
\end{abstract}


\section{Introduction and Overview}

\subsection{Introduction}
Joyal introduced a monoidal product for species in \cite{species}. This type of product for $\Sigma-$modules  was used by Smirnov in  \cite{Smirnov} to give a  definition of operads as a monoid as an alternative to that of May \cite{Mayoperad}.
Under the name plethysm product, it was used in \cite{GKmodular} to study characteristics for cyclic and modular operads. The neologism plethysm goes back to the 1950 edition of \cite{Littlewood} where it was used for symmetric functions.
A plethysm definition is also available for props and properads through the work of Vallette \cite{Vallette} who used it in the theory of Koszul duality. 
The framework we consider is that of bi--modules for categorical actions as a specialization of the profunctors of B\'enabou \cite{Benabou67} where the plethysm product is the relative product defined by a coend.

This approach furthermore allows us to link our construction to other categorical constructions in particular to the plus construction and Grothendieck element constructions.
The element construction is part of Grothendieck's construction described in \cite{groupe-fondamental} which establishes an equivalence between fibrations $\mathcal E \to \arbcat$ over a category $\arbcat$ and pseudo-functors $\arbcat \to \mathcal Cat$. This construction can be seen as part of a comprehension scheme providing  covering morphisms \cite{BergerKaufmann}.
The plus construction under the name of opetopes goes back to Baez and Dolan \cite{BaezDolan} and was designed to make definitions of algebras over operads and higher analogs. The lower analogue being modules over algebras. The twistings of \cite{GKmodular} that are needed to define bar/cobar and dual, aka.\ Feynman transforms, implicitly utilize the plus-construction. It is thus essential to the theory of operads and their generalizations.

The motivation for this work was fourfold ---see 
\S\ref{par:applications} for more technical details of this summary. 
The first aim was to answer the question of when a plethysm-monoid definition for operad--like theories is possible. The answer to this is that if the operad--like object comes from a Feynman category (FC) \cite{feynman} then it has such a definition if and only if the category in question is the plus construction of a Unique Factorization Category (UFC). 
The second motivation is the observation that the two constructions of decoration \cite{decorated,BergerKaufmann} and enrichment \cite{feynman,feynmanrep} seem to be of a similar nature. If the Feynman categories are graphical, then one decorates the vertices and the other decorates the edges of the graphs underlying the morphisms. The fact that we show that both of the constructions can be viewed as element constructions gives a rigorous framework for this observation.
The third motivation is to consider the most general framework by starting just with categories and then adding natural extra structures in parallel of the treatment of plus constructions in \cite{KMoManin}. Our new perspective followed in the current article is to treat categories as bimodules over their underlying groupoids and consider the monoidal categories of such bimodules and their element constructions.
This approach also makes the treatment of units, which are more subtle that one would think at first glance, natural and explains the different versions of plus constructions of \cite{KMoManin,feynmanrep}.
This ties to the forth and original motivation by Baez and Dolan, namely to define algebras. There is a straightforward notion for algebras for bimodules, and we relate this to more general construction that appear in operad theory notably cyclic and modular operads. We emphasize that all structures now are derived from monoidal categories with special monoidal categories, such as FCs or UFCs, entering the theory later. Formulated in this generality, our results should translate readily to related theories, see \S\ref{par:other}.

We treat three levels of bimodules: the bimodules, relative bimodules and factorizable bimodules.  Here the bimodules are taken in the general setting of actions by categories. Monoidal structures on each of these levels and link them by monoidal functors and monoidal equivalences. Using the element construction, we define element representations as functors from the category of elements of a given bimodule and show that these again form a monoidal category. We explain the interrelations between these categories by providing strong monoidal functors and equivalences between them.
This establishes basic facts and constructions for bimodules. Table \ref{table:reference} contains a list of the objects and monoidal products.
\begin{table}
  \begin{tabular}{llc}
    Name & Objects & Product \\
    \hline
    Bimodules & Functors $\rho: \act^{op} \times \act \to \C$ & $\pl$ \\
    Relative bimodules & Bimodules $\xi$ over $\rho$ & $\pl_{\rho}$ \\
    Element representation & Functors $D: \el(\rho) \to \C$ & $\epl$ \\
    Decorated element rep. & Functor $E: \el(D) \to \C$ & $\epl_D$ \\
    Basic element representation & Functors $D: \el(\nu) \to \C$ & $\bepl$ \\
    Basic relative bimodules & Bimodules $\xi$ over $\nu$ & $\pl_{(\nu)}$ \\
    Basic dec. element rep. & Functor $E: \el(D) \to \C$ & $\bepl_D$
  \end{tabular}
  \caption{\label{table:reference} Reference table for  standard notations and
  monoidal products.}
\end{table}

In the special case of a groupoid bimodule,
we establish a new connection between the element construction and the plus construction for a category as defined in \cite{KMoManin}. Unital bimodule monoids can be viewed as categories with additional structure. This yields an interpretation of the constructions for unital bimodule monoids as constructions for categories. 
 The plus construction of a category corepresents relative bimodules. 
 As is shown in the current article, these relative bimodules can  be identified with  element representations, that is functors from a category of elements. 
 There is also a plus construction for the unital version that corepresents indexed enrichments for categories. These are shown to correspond to the unital monoidal bimodules. There are extra technical complications in the case of bimodules taking values in a non--Cartesian monoidal category. In this case, there is extra data that is needed for an element representation to yield a relative monoid. This is identified 
as a system of weights which translates to a counit for a respective monoidal functor.

Considering these as special cases, bimodules yield a broader purely categorical and algebraic framework for these constructions using only functors, monoidal structures, element constructions, coends and Kan extensions.
 The connection to classical theory is via hereditary unique factorization categories \cite{KMoManin} and
 Feynman categories \cite{feynman}.
 The construction for bimodules is more general, as  neither a groupoid, nor a  factorization, nor a monoid structure, nor a unit is presupposed. The compatibles arising when  adding such structures are proven at each step. Under the interpretation of a unitary bimodule monoid as a category,  an element representation is a decoration or a covering. As an application, using the decorated element representations which can be seen as a double application of the element construction, we obtain a commutativity result: the plus construction of an indexed enrichment is equivalent to the decoration of the plus construction. As an added benefit, we obtain new natural definitions for being hereditary and having mergers.

\subsection{Results and constructions for bimodules} 
Motivated by the applications that appear on the third level, the article provides the following standalone theory of general basic constructions and results for bimodules.

\subsubsection{First level: bimodules}

Bimodules are the  fundamental structure and all of the other structures in this paper depend on it. For a fixed category $\act$, a bimodule is a functor $\rho: \act^{op} \times \act \to \Set$ or more generally a functor $\rho: \act^{op} \times \act \to \C$ to a target category $\C$ which is usually closed monoidal, so that it can serve as an enrichment category.
These and the  structure of the plethysm product $\pl$, are introduced in paragraph~\ref{sec:bimodules}. The fact that $\pl$ is a monoidal structure is established in the current setting in Proposition~\ref{square-mon-cat}. The unit is the bimodule $\tilde\rho_\act = \Hom_\act(-,-)$. 

A bimodule monoid is given by an associative natural transformation $\gamma: \rho \pl \rho \To \rho$. The connection to categories is summarized by Corollary~\ref{cor:bimodcat} which establishes a bijection between unital bimodule monoids and categories $\arbcat$ equipped with a functor of the action category into the category $\arbcat$.

\subsubsection{Second level: relative bimodules and element representations}

Element representations for a fixed bimodule $\rho$ are functors $D:\el(\rho)\to \C$, where $\el(\rho)$ is the category of elements, cf.\ Section~\ref{sec:element-rep}. If $\rho$ is a unital monoid for $\pl$ then the element representations are a monoidal category whose monoidal structure $\epl$ is called the element plethysm product by Proposition~\ref{prop:diamondproduct}. There is a strong monoidal functor $\chi$ from element representations with $\epl$ as a monoidal product back to bimodules with $\pl$ as a monoidal product, see Theorem~\ref{thm:strongmonoidal}. 

We also consider relative bimodules over a fixed bimodule $\rho$, which are bimodules $\xi$ equipped with a natural transformation to $\rho$. If $\rho$ is a unital monoid, there is a monoidal structure
  $\pl_\rho$ of relative plethysm for relative bimodules, Proposition \ref{prop:relative-plethysm}.
In the Cartesian case, the relative bimodules are exactly the bimodules in the image of $\chi$, cf.\ Theorem~\ref{thm:elementrelative}.
In the non--Cartesian case, the situation is more complicated as monoidal units need not be terminal objects and additional data is needed to glean a relative bimodule from an element representation cf.\ Corollaries~\ref{cor:weighttorelative}.
 A relative bimodule gives rise to a functor between the two categories defined by the base and the over-object. This is the definition of indexed enrichment, cf.\ Proposition \ref{prop:indexingfunctor}.

 \subsubsection{Third level: factorizable bimodules}

If the category $\act$ has a monoidal structure, there is a natural notion of essentially unique factorization $\rho \overset{\sim}{\Rightarrow} \nu^{\otimes}$ for the base bimodule $\rho$. This is   formalized by a horizontal extension construction using a Kan extension, Definition \ref{eq:hor-ext-def}. We drop the adjective essential in the following.
The correct notion of a morphism between such uniquely factorizable bimodules is a hereditary morphism, Definition~\ref{def:hereditary}.  
This identifies and characterizes the hereditary condition in the theory of bimodules.
The factorizable bimodules with hereditary morphisms form a monoidal subcategory, cf.\ Proposition \ref{prop:monoialsubcat}.
This naturally leads to the notions of basic element representations and basic relative bimodules based on $\nu$. 
These again form monoidal categories which are compatible with the structures for $\rho$; see Propositions \ref{elem-element-rep-mon-cat} and \ref{prop:basic-elt-mon}. 

\subsubsection{Decorations and Algebras}
The element construction can be repeated for element representations,  both basic and regular.
These again form monoidal categories.
The main result is Theorem~\ref{thm:pluselement}, which establishes an equivalence between the monoidal category of decorated element representations of an element representation $D$ and the monoidal category of element representations of $\chi_D$, and Theorem~\ref{thm:pluseltbasic} which is the analogous result for decorated basic element representations.
This is the previously mentioned ``commutativity'' of element constructions.

Representations of monoidal bimodules are most naturally viewed as natural transformations to a reference functor. This specializes to the fact that algebras of general operad--like structures are  defined relative to some other structure, namely a Hom functor, see Example~\ref{ex:referenceindex}.
For element representations there are two possible versions of bialgebras, which we prove to be equivalent
in Theorem~\ref{thm:algebras-el-rep}. The two constructions correspond to either 
lifting the reference functor to the category of elements or pushing the element representation $D$ to $\chi_D$. The former corresponds to algebras on the first level and the second interpreted in the relative situation can be seen in the situation of a Hom functor a functor out of the indexed enrichment.
The monoidal situation is analogous, Theorem~\ref{thm:algebras-monoids}. Here in the aforementioned special case, if the target category is monoidal, this notion is equivalent to the algebraic notion of a module for the plethysm product for the monoid $\rho$; cf.\ \ref{ex:algindexenrich}.
In the factorizable case, one has an additional monoidal product and the behaviour of the algebras can be lax or strong with respect to this product. This is formalized in the definition of lax and strong algebras
Definitions \ref{def:laxalgebra} and \ref{def:strongalgebra} which each have two equivalent definitions, see Propositions \ref{prop:laxequiv} and \ref{prop:strongequiv}. There is one more possible approach in this case, which just utilizes basic element representations. This presupposes the existence of an additional structure that we call a merger.
In the presence of these, one again has two definitions for algebras and their equivalence, cf.\ Proposition \ref{prop:laxequiv}. 
Finally, we can describe mergers naturally as an algebra over a particular monad, Theorem~\ref{thm:mergermonad}.
This is a novel general description of mergers which appear in graphs and PROPs, cf.\ \cite[Section 2, 3.2.1.]{feynman}.

\subsection{Application to classical theory}\label{par:applications}
The current paper generalizes the classical constructions of \cite{BaezDolan} in the following sense:
Factorizable groupoid bimodules correspond
to  special monoidal categories called hereditary unique factorization categories (UFCs) \cite{KMoManin}, see Proposition~\ref{prop:UFC}.  Here the monoidal structure is an additional structure provided by the factorization ---not that of the bimodules. Basic examples are spans and cospans.
It was shown in {\it loc.\ cit.} that the plus construction for such a hereditary UFC,  Definition~\ref{def:UFC}  yields a Feynman category (FC), Definition \ref{def:FC},  and that this is a necessary condition. 
FCs \cite{feynman} corepresent operad--like theories via strong monoidal functors $\F \mddash \ops_\C=[\F,\C]_\ot$, among them are operads, algebras, props, properads and the whole list of known operad--like theories. 
The correspondences between constructions in bimodules and categories is summarized in Table~\ref{table:reference2}. With the view toward the application to operad--like theories, typical examples are listed as well.  The  plus construction for FCs was given in \cite{feynman} and generalized in \cite{KMoManin} to categories and monoidal categories.
The opetopes of \cite{BaezDolan} which are a plus construction, were announced for operads and worked out for non--Sigma operads.

\begin{table}
  \scalebox{0.9}{
  \begin{tabular}{lll}
    $\act$--bimodule structure
    & Categorical structure
    & Typical factorizable example \\
    \hline
    unital bimodule
    & Category $\arbcat(\rho,\gamma)$\\
    \quad monoid$(\rho,\gamma)$&\\
    Element representations $D$
    & Functor $D\in[\catplus{\arbcat},\C]$
    &\\
    Factorizable  unital
    & UFC $\M$
    &$\Cospan$, any FC \\
    \quad bimodule monoid $\rho_\M\simeq \nu^\ot$\\

    Basic element representation
    & $D\in \M^+\mdash\ops_\C$
    & Operad, properad, hyper--   \\
                              &&\quad operad as functors\\
    Chi construction for a unital
    &Indexed enrichment
    &Operad as plethysm monoid\\
    \quad relative bimodule monoid
    &$\arbcat_D=\arbcat(\chi_D,\gamma)$
    & \\
    Decorated element representation $E$
    & Functor $E\in[\catplus{\arbcat_D}, \C]$
    &Non--Sigma operad as functor\\
    Algebra as element representation &$N:D \Rightarrow E_{\alpha_\O}$
    & Algebra over operad as  \\
    \quad $N:D \Rightarrow E_{\alpha_\O}$&&\quad morphism of operads\\
    Algebra as bimodule
    & Functor $\catplus{\arbcat} \to \alpha_\O$
    & Algebra over operad \\
    \quad $\chi_D\rightarrow \alpha_\O$&& \quad as plethysm module
  \end{tabular}}
  \caption{\label{table:reference2} Basic reference table for interpretation of bimodules in categorical terms with further technical conditions omitted. The plus constructions $\catplus{\arbcat}$ and $\M^+$ are defined in \protect{\cite{KMoManin}}}

\end{table}

The  main new result  from the current framework is Theorem~\ref{main-thm} which  states that if a Feynman category $\F$ can be written as a plus construction $\F \simeq \M^+$ for some UFC $\M$ then the monoidal category of strong monoidal functors $[\F,\C]_\ot$ is equivalent to the category of monoids in the monoidal category of  basic element representations described in Section~\ref{sec:basic-el-rep}, see Remark \ref{rmk:reduced} for technical details.
This establishes that the plus construction corepresents a type of element construction. This is also true for the categorical plus construction $\catplus{\arbcat}$ defined in \cite{KMoManin} which concerns level 2.
The bimodule approach also naturally produces the condition of being hereditary as well as the existence of mergers as natural compatibilities.

Specializing to FCs gives a criterion for when a monoid definition of the corepresented structures is possible. In particular, there is a plethysm definition of the corepresented structures precisely when  the Feynman category is the plus construction of some  hereditary UFC $\M$. In the situation of operads, this corresponds to the fourth definition of an operad as a plethysm monoid.  The three other definitions are
(1) being able to compose along rooted trees, (2) being an algebra over the monad of trees, and (3)  being defined via May's  $\gamma$-operations   \cite{Mayoperad} or the $\circ_i$ of Markl \cite{Markl}, cf.\ \cite{MSS}. The first three are also available for modular and cyclic operads, but not the fourth.
Taking the point of view of  FCs, more generally, the first is the definition of operad--like objects, the second is a consequence of a monadicity theorem and the third is available if there is a generator and relation description of the morphisms of the FC. 
 The list of examples for the plethysm construction for operads and properads stemming from the FC of finite sets and the hereditary UFC of reduced cospans.  
 New consequences are that hyper (modular) operads do have a plethysm definition with respect to an underlying groupoid, see Section~\ref{sec:hyper-modular-special-case} while for instance modular or cyclic operads now necessarily do not.

The element construction for a functor $D:\F\to \C$ out of an FC is again an FC and is called a decoration \cite{decorated}, and denoted by $\F_{\dec D}$.  The forgetful functor from the decorations to the original category furnish coverings which are part of a factorization system \cite{BergerKaufmann}. If $\F\simeq \M^+$, then 
the ``commutativity''  yields 
Theorem \ref{thm:feynmandecplus} which was conjectured in \cite{feynmanrep}. It states that the plus construction of the indexed enriched category is the decoration of the plus construction of the decoration $(\F_D)^+\simeq (\F_{dec D})^+$.
A particular example of this  commutativity is that the category that corepresents non--Sigma operads of the plus construction on planar corollas as was noticed in \cite{KMoManin}. This applies now to any FC which is a plus construction as conjectured in \cite{KMoManin}, see Theorem \ref{thm:feynmandecplus}.

In FCs, a morphism in $\F\mdash\ops$ is a natural transformation of functors.  
If an operad is thought of as a functor, then  an algebra over a functor is a natural transformation to a reference functor.
In the classical case this is a morphism of operads to an endomorphism operad.
The plus construction is related as follows, algebras over a functor out of an FC $\F$ can be corepresented by another FC, if the Feynman category $\F$ is the plus construction of another Feynman category $\M$. The corepresenting FC $M_\O$ is an indexed enrichment \cite[Section 4]{feynman}. 
The ``commutativity'' in this situation yields 
Theorem~\ref{thm:feynmandecplus} which was conjectured in \cite{feynmanrep}. It states that the plus construction of the indexed enriched category is the decoration of the plus construction of the decoration $(\F_D)^+\simeq (\F_{\dec D})^+$.

\subsubsection{Connection to other points of view on operad--like objects.}
 \label{par:other}
The connection to other presentations of operad--like theories is as follows.
There are equivalent ways to think about  FCs in terms of the  groupoid colored operads of \cite{PetersenOperad} or as  regular patterns of \cite{Getzler}, see \cite[Section 1.11]{feynman}. The property of a pattern is exactly what is needed for a monadicity theory. The details of these connections in a two--categorical setting were later worked out in \cite{BKW}. The idea of using corepresenting monoidal categories goes back to \cite{BoVo,BV,Costelloenvelope,Getzler}. The hereditary condition in the graphical context appears in different guises in \cite{Markl} and \cite{BorMan}. These connections allow us to export the results of this paper to these theories.
  In other contexts, the plus construction can be found in \cite{opetope}, \cite{BMplus}, and \cite{Bergermoment}. The exact relationships between operadic categories and FCs is again given through a plus construction: there are functors from FCs to operadic categories and back whose composition is a plus construction. The plus construction of a moment category is an FC. 
 Independent approaches to decorations of cospans are in \cite{Steinebrunner,PhilCospan}. 

\section*{Acknowledgements}
We wish to thank Yuri I.\ Manin, Clemens Berger, Alexander Voronov, Jim Stasheff, and Phil Hackney for discussions, and in particular Jan Steinebrunner  for pointing out technicalities about cospans and UFCs.
RK thanks the Simons foundation for their support.

\section{Bimodules}\label{sec:bimodules}

\subsection{Bimodules and Representations}\label{sec:bimod-repr}

Roughly speaking, a bimodule is something with a two-sided action. We will devote this section to making this statement precise.
As motivation, we consider group actions in both a classical and categorical formalism. We then show that these ideas transfer to a more general setting which will ultimately lead us to the definition of a bimodule.

\subsubsection{Group actions}\label{sec:group-actions}

Recall that a group action for a group $G$ is classically defined as a set $X$ equipped with a map $(- \cdot -): G \times X \to X$ such that $gh \cdot x = g \cdot (h \cdot x)$.
Taking the adjunction, we get maps of the form $G \to X^X$. This means that instead of thinking of $G$ as ``acting'' on elements of a set $X$, we can think of ``representing'' elements $g \in G$ as endofunctions on $X$.
This point of view allows us to define actions in a more categorical manner.

First recall that any group $G$ determines a category $\underline{G}$ as follows:
\begin{enumerate}
  \item There is a single object called $\bullet$.
  \item There is morphism $\bullet \overset{g}{\to} \bullet$ for each element $g \in G$.
  \item Composition is defined by group multiplication: $g \circ h = gh$.
\end{enumerate}

\begin{definition}
  We will call a functor of the form $\rho: \underline{G} \to \Set$ a \defn{representation} of $G$.
  Such a functor constitutes the following data:
  \begin{enumerate}
  \item The functor determines an object $X := \rho(\bullet)$.
  \item Each element $g \in G$ of the group determines an endofunction $\rho(g): X \to X$.
  \end{enumerate}
\end{definition}

\begin{prop}
  There is a one-to-one correspondence between $G$-actions and functors $\rho: \underline{G} \to \Set$.
\end{prop}

\begin{proof}
  Let $X := \rho(\bullet)$ denote the image of a functor $\rho: \underline{G} \to \Set$ and observe that $\rho$ is determined by a map $G \to X^X$.
  By adjunction, there is a one-to-one correspondence between maps of the form $G \times X \to X$ and maps of the form $G \to X^X$.
  Moreover, a map of the form $G \times X \to X$ satisfies the compatibility condition for a group action if and only if its adjoint $G \to X^X$ satisfies the composition condition for a functor.
  Likewise, a map of the form $G \times X \to X$ satisfies the identity condition for a group action if and only if its adjoint $G \to X^X$ satisfies the identity condition for a functor.
\end{proof}

\subsubsection{General actions}\label{sec:general-actions}

A natural next step would be to replace the group $G$ with a groupoid $\Gpd$.
In fact, it is just as easy to generalize it to a category $\act$, which we will do.
We can also replace the target category $\Set$ with any category $\C$ copowered over $\Set$.

\begin{definition}
  \label{copower}
    The \defn{copower} of a set $I$ with an object $X \in \Obj(\C)$ is an object $I \odot X \in \Obj(\C)$ with the following property:
    \begin{equation}
        \Hom_{\C}(I \odot X, Y) \cong \Hom_{\Set}(I, \Hom_{\C}(X,Y))
    \end{equation}
    In Kelly's book \cite{kellybook}, this is called the \emph{tensor} in the enriched case.
    However, we prefer to use the term ``copower'' in both the enriched and unenriched settings.
\end{definition}

\begin{example}
    The category $\Vect$ of vector spaces over a fixed field has a copower $I \odot V = \bigoplus_{i \in I} V$.
    By the universal property of the coproduct, a linear map $\bigoplus_{i \in I} V \to U$ is the same thing as an $I$-indexed collection of linear maps $V \to U$.
\end{example}

\begin{definition}
  Define a \defn{left $\act$-module} in $\C$ to be a class of objects $\M = \{ M_X \}$ in $\C$ indexed by the objects of $\act$ and a family of morphisms
  \begin{equation}
    l_{X,Y}: \Hom_{\act}(X,Y) \odot M_X \to M_Y
  \end{equation}
  with the following properties:
  \begin{enumerate}
    \item The action is compatible with composition:
    \begin{equation}
    \begin{tikzcd}
    {\Hom_{\act}(Y,Z) \odot (\Hom_{\act}(X,Y) \odot M_X)}
    \arrow[r, "{\id \odot l_{X,Y}}"]
    \arrow[d, "\sim"]
    & {\Hom_{\act}(Y,Z) \odot M_Y}
    \arrow[dd, "{l_{Y,Z}}"] \\
    {(\Hom_{\act}(Y,Z) \times \Hom_{\act}(X,Y)) \odot M_X}
    \arrow[d, "\circ \odot \id"]
    & \\
    {\Hom_{\act}(X,Z) \odot M_X}
    \arrow[r, "{l_{X,Z}}"]
    & M_Z                                               
    \end{tikzcd}
    \end{equation}
    \item Let $\eta_X: pt \to \Hom_{\act}(X,X)$ be the map that selects the identity. We then require the action to be compatible with identities:
    \begin{equation}
        \begin{tikzcd}
        pt \odot M_X
        \arrow[r, "\eta_X \odot \id"]
        \arrow[rd, "\sim"'] & {\Hom_{\act}(X,X) \odot M_X}
        \arrow[d, "{l_{X,X}}"] \\
        & M_X 
        \end{tikzcd}
    \end{equation}
  \end{enumerate}
\end{definition}

Like before, the ``action point of view'' where we consider a compatible family of maps $\Hom_{\act}(X,Y) \odot M_X \to M_Y$ can be complemented by a ``representation point of view'' where we consider functors $\rho: \act \to \C$.

\begin{prop}
  There is a one-to-one correspondence between $\act$-modules in $\C$ and functors $\rho: \act \to \C$.
\end{prop}

\begin{proof}
  Suppose we have the module $\M = \{M_X\}$ with an action $l_{X,Y}: \Hom_{\act}(X,Y) \odot M_X \to M_Y$.
  Set $\rho(X) = M_X$, then define the correspondence $\rho_{X,Y}: \Hom_{\act}(X,Y) \to \Hom_{\C}(M_X, M_Y)$ from $l_{X,Y}$ using the property of copowers described in Definition~\ref{copower}.
  Functorality of $\rho$ follows from conditions (1) and (2).
  The converse is similar.
\end{proof}

\begin{definition}
  With group actions, there is a distinction between left and right action.
  This carries over to groupoid and category actions.
  To be consistent with traditional notation for right actions, let $\boxdot$ denote the ``flipped'' version of the copower.
  
  Define a \defn{right $\act$-action} in $\C$ to be a class of objects $\M = \{ M_X \}$ in $\C$ indexed by $\act$ and a family of morphisms $r_{Y,Z}: M_Y \boxdot \Hom(X, Y) \to M_X$ such that:
  \begin{enumerate}
    \item The action is compatible with composition:
    \begin{equation}
    \begin{tikzcd}
    {(M_Z \boxdot \Hom_{\act}(Y,Z)) \boxdot \Hom_{\act}(X,Y)}
    \arrow[r, "{r_{Y,Z} \boxdot \id}"]
    \arrow[d, "\sim"]
    & {M_Y \boxdot \Hom_{\act}(X,Y)}
    \arrow[dd, "{r_{X,Y}}"] \\
    {M_Z \boxdot (\Hom_{\act}(Y,Z) \times \Hom_{\act}(X,Y))}
    \arrow[d, "\id \boxdot \circ"]
    & \\
    {M_Z \boxdot \Hom_{\act}(X,Z)}
    \arrow[r, "{r_{X,Z}}"]
    & M_X
    \end{tikzcd}
    \end{equation}
    \item The action is compatible with identities:
    \begin{equation}
        \begin{tikzcd}
        M_X \boxdot pt
        \arrow[r, "\id \boxdot \eta_X"]
        \arrow[rd, "\sim"']
        & {M_X \boxdot \Hom_{\act}(X,X)}
        \arrow[d, "{r_{X,X}}"] \\
        & M_X 
        \end{tikzcd}
    \end{equation}
  \end{enumerate}
\end{definition}

\begin{prop}
  There is a one-to-one correspondence between right $\act$-actions in $\C$ and functors $\rho: \act^{op} \to \C$.
\end{prop}
\begin{proof}
  Similar.
\end{proof}

\subsubsection{Bimodules}

We will be interested in things that simultaneously have a left and a right action.
Unfortunately, the terminology ``left'' and ``right'' for modules clashes with the tendency to read from left to right when composing arrows in a category.
This will become more confusing later on, so we will drop these terms.
Instead, we will say that the module $\rho: \act \to \C$ has a \defn{covariant action} and the module $\rho: \act^{\rop} \to \C$ has a \defn{contravariant action}.

\begin{definition}
  An \defn{$\act \mddash \act$ bimodule} is a functor of the following form:
  \begin{equation}
    \rho: \act^{\rop} \times \act \to \C
  \end{equation}
  We will call $\act$ the \defn{action category}.
  In many of our examples, $\act$ is a groupoid, but we generally do not need that assumption in our proofs.

  For notation, we will write the morphisms of $\act^{\rop} \times \act$ as $(\sigma \biact \tau)$ with the contravariant part on the left and the covariant part on the right.
  In situations where the action point of view is appropriate, we will write $(\sigma \biact \tau) \cdot x$ so that the element $x$ is receiving the contravariant action of $\sigma$ and the covariant action of $\tau$.
\end{definition}

\begin{example}
  Let $\act = \SS$ denote the groupoid whose objects are the natural numbers and whose morphisms are given by $\Hom(n,n) = \SS_n$ and $\Hom(n,m) = \varnothing$ if $n \neq m$.
  The composition in $\SS$ is given by group multiplication.
  From the action point of view, an $\SS \mddash \SS$ bimodule is a sequence of objects $\{ \P(n,m) \}_{n,m \in \N}$ with a contravariant $\SS_n$-action and a covariant $\SS_m$-action.
\end{example}

\begin{example}
  \label{ex:cospan}
  Define a bimodule $\rho_{\rCospan}: \Iso(\FinSet)^{\rop} \times \Iso(\FinSet) \to \Set$ so that $\rho_{\rCospan}(S,T)$ is the set of diagrams $S \rightarrow V \leftarrow T$ called \defn{cospans} modulo isomorphisms.
  The bimodule action $\rho_{\rCospan}(\sigma, \tau): \rho_{\rCospan}(S, T) \to \rho_{\rCospan}(S', T')$ is defined by sending the cospan $S \overset{l}{\rightarrow} V \overset{r}{\leftarrow} T$ to the cospan $S' \overset{l \circ \sigma}{\rightarrow} V \overset{r \circ \tau^{-1}}{\leftarrow} T$.

  Graphical, we can think of $\rho_{\rCospan}(S, T)$ as the set of graphs with $S \amalg T$ as the flags, unlabeled vertices, and no edges.
  We will think of $S$ as the set of in-tails and $T$ as the set of out-tails.
  For example, an element $x \in \rho_{\rCospan}(\{s_1, s_2, s_3, s_4\}, \{t_1, t_2, t_3, t_4\})$ might look like this:
  \begin{equation}
    \begin{tikzpicture}[baseline={(current bounding box.center)}]
      \node [style=none] (0) at (-1, 0.75) {$s_1$};
      \node [style=none] (1) at (-0.5, 0.75) {$s_2$};
      \node [style=none] (2) at (0, 0.75) {$s_3$};
      \node [style=none] (3) at (-0.75, -0.75) {$t_1$};
      \node [style=none] (4) at (-0.25, -0.75) {$t_2$};
      \node [style=White] (5) at (-0.5, 0) {};
      \node [style=White] (6) at (0.75, 0) {};
      \node [style=none] (7) at (0.75, 0.75) {$s_4$};
      \node [style=none] (8) at (0.5, -0.75) {$t_3$};
      \node [style=none] (9) at (1, -0.75) {$t_4$};
      \draw (0) to (5);
      \draw (1) to (5);
      \draw (2) to (5);
      \draw (5) to (3);
      \draw (5) to (4);
      \draw (6) to (8);
      \draw (6) to (9);
      \draw (7) to (6);
    \end{tikzpicture}
  \end{equation}
  Under this interpretation, the groupoid $\Iso(\FinSet)$ acts by relabeling the tails.
\end{example}

\begin{example}
  \label{can-bimod}
  Any category $\arbcat$ induces a canonical $\Iso(\arbcat) \mddash \Iso(\arbcat)$ bimodule where $\Iso(\arbcat)$ is the subcategory of $\arbcat$ consisting of all objects and all isomorphisms.
  
  \begin{description}
  \item[Module point of view] Let $\M = \Mor(\arbcat)$ be the collection of morphisms.
    Then the groupoid $\Iso(\arbcat)$ acts both contravariantly and covariantly by composition:
    \begin{equation}
      (\sigma \biact \tau) \cdot \phi
      = \tau \circ \phi \circ \sigma
    \end{equation}
  \item [Representation point of view] Define the functor $\rho_\arbcat:\Iso(\arbcat)^{\rop}\times \Iso(\arbcat)\to \Set$ on objects by $\rho_\arbcat(X,Y) = \Hom(X,Y)$.
    For a pair of morphisms $\sigma: X' \to X$ and $\tau: Y \to Y'$ in $\Iso(\arbcat)$, define $\rho_\arbcat(\sigma \biact \tau): \Hom(X,Y) \to \Hom(X',Y')$ by sending $\phi$ to $\tau \circ \phi \circ \sigma$.
  \end{description}
\end{example}

\subsection{Plethysm}\label{sec:plethysm}

The $\act \mddash \act'$ bimodules naturally constitute a functor category.
We will see that when $\act = \act'$, this category can be equipped with a monoidal product which we will call the \emph{plethysm product}.

\subsubsection{Categories as a motivation}\label{sec:cat-motivation}

To motivate the definition of the plethysm product, we will take a close look at composition in a category.
Recall from Example~\ref{can-bimod}, that for any category $\arbcat$ we can consider $\M = \Mor(\arbcat)$ as a bimodule with an $\Iso(\arbcat) \mddash \Iso(\arbcat)$ action.
Let $\Mor(\arbcat) \leftsub{t}{\times_s} \Mor(\arbcat)$ denote the collection of ordered pairs $(\phi, \psi)$ of morphisms such that the target of $\phi$ agrees with the source of $\psi$:
\begin{equation}
  \begin{tikzcd}
    A
    \arrow[r, "\phi"]
    & B=B
    \arrow[r, "\psi"]
    & C
  \end{tikzcd}
\end{equation}
Note that $\Mor(\arbcat) \leftsub{t}{\times_s} \Mor(\arbcat)$ can be given either an $(\Iso(\arbcat)^{\rop} \times \Iso(\arbcat)) \times (\Iso(\arbcat)^{\rop} \times \Iso(\arbcat))$--action or we can ignore the inner variables and give it an $\Iso(\arbcat)^{\rop} \times \Iso(\arbcat)$--action on the ``outer'' variables.
If we give $\Mor(\arbcat) \leftsub{t}{\times_s} \Mor(\arbcat)$ an $\Iso(\arbcat)^{\rop} \times \Iso(\arbcat)$--action on the ``outer'' variables, then we can define composition in $\arbcat$ as an $\Iso(\arbcat) \mddash \Iso(\arbcat)$ bimodule map:
\begin{align*}
  \bar \gamma: \Mor(\arbcat) \leftsub{t}{\times_s} \Mor(\arbcat) &\to \Mor(\arbcat) \\
  (\phi, \psi) &\mapsto \psi \circ \phi
\end{align*}
Observe that the bimodule map $\bar \gamma$ is then equivariant with respect to the ``inner'' variable since the following compositions are always equal:
\begin{equation}
  \begin{tikzcd}
    X
    \arrow[rr, "\phi"]
    & & Y
    \arrow[rr, "\psi"]
    & & Z \\
    X
    \arrow[rd, "\phi"']
    \arrow[rr, "(\id \biact \sigma) \cdot \phi"]
    & & Y'
    \arrow[rd, "\sigma^{-1}"']
    \arrow[rr, "(\sigma^{-1} \biact \id) \cdot \psi"]
    & & Z \\
    & Y
    \arrow[ru, "\sigma"']
    & & Y
    \arrow[ru, "\psi"'] &
  \end{tikzcd}
\end{equation}
Hence the composition map factors through the coinvariants of the diagonal action:
\begin{equation}
  \label{eq:factor-coinvar}
  \begin{tikzcd}[column sep = tiny] \Mor(\arbcat)\leftsub{t}{\times_s}\Mor(\arbcat)
    \arrow[rr, "\bar \gamma"]
    \arrow[rd, "\pi"']
    & & \Mor(\arbcat) \\
    & \Mor(\arbcat) (\leftsub{t}{\times_s})_{\Iso(\arbcat)} \Mor(\arbcat)
    \arrow[ru, "\gamma"'] &
  \end{tikzcd}
\end{equation}

\subsubsection{Formal definition of plethysm}
\label{sec:formal-def-plethysm}

When we view a category $\arbcat$ as an $\Iso(\arbcat) \mdash \Iso(\arbcat)$ bimodule, we relied on the monoidal structure $(\Set, \times)$ to define composition.
More generally, in order for the category of bimodules $\Fun(\act^{op} \times \act, \C)$ to have a monoidal structure, we need the target category $\C$ to also have a monoidal structure.
We also claim that any reasonable notion of composition should satisfy an equivariance property like the one seen in the last section.
This will lead us to consider a ``tensored'' version of the monoidal product.
We will formalize this as a coend.

\begin{definition}
  Let $\rho: \act \to \C$ and $\rho': \act' \to \C$ be two modules such that the target category $\C$ is monoidal.
  Define the \defn{exterior product} $\rho \boxtimes \rho: \act \times \act' \to \C$ on objects by $(\rho\boxtimes \rho')(X,Y) = \rho(X)\otimes_\C\rho(Y)$ and similarly on morphisms.
\end{definition}

\begin{example}
  Recall from Section~\ref{sec:bimod-repr} that an action of a group $G$ on a set $X$ can be represented as a functor $\rho: \underline{G} \to \Set$ such that $X = \rho(\bullet)$.
  Let $\rho': \underline{G}' \to \Set$ represent an action of the group $G'$ on the set $X' = \rho'(\bullet)$.
  Then $\rho \boxtimes \rho'$ represents the action of the group $G \times G'$ on the set $X \times X'$ defined by $(g,g') \cdot (x,x') := (g \cdot x, g' \cdot x')$.
\end{example}

\begin{definition}
  Given two $\act \mddash \act$ bimodules $\rho_1$ and $\rho_2$, define the \defn{plethysm product} $\rho_1 \pl \rho_2: \act^{op} \times \act \to \C$ as follows:
  \begin{equation}
    [\rho_1 \pl \rho_2](A,C)
    := \int^{B : \act} [\rho_1 \boxtimes \rho_2]((A,B), (B,C))
  \end{equation}
\end{definition}

\begin{example}
  In the case of $\act=\underline{G}$, this is just the relative product for two $G\mddash G$ bimodules $A$ and $B$.
  This is the object $A \otimes_G B$ considered as a $G\mddash G$ bimodule with the left action being that on $A$ and the right action being that on $B$.
\end{example}

\begin{rmk}\label{cocone-remark}
To see the connection with the coinvariants discussed in Section~\ref{sec:cat-motivation}, consider the co--wedge for the morphism $\sigma: B \to B'$ in the plethysm product:
\begin{equation}
  \begin{tikzcd}
    {\rho_1(A,B) \otimes \rho_2(B', C)}
    \arrow[d, "{\rho_1(\id,\id) \otimes \rho_2(\sigma,\id)}"']
    \arrow[rrr, "{\rho_1(\id,\sigma) \otimes \rho_2(\id,\id)}"]
    & & & {\rho_1(A,B') \otimes \rho_2(B',C)}
    \arrow[d, dashed] \\
    {\rho_1(A,B) \otimes \rho_2(B,C)}
    \arrow[rrr, dashed]
    && & {(\rho_1 \pl \rho_2)(A,C)}
  \end{tikzcd}
\end{equation}
In the case where $\act = \Gpd$ is a groupoid, we can use the invertibility of the morphisms to replace the co--wedge with a cocone:
\begin{equation}
  \begin{tikzcd}
    {\rho_1(A,B) \otimes \rho_2(B, C)}
    \arrow[rr, "{\rho_1(\id,\sigma) \otimes \rho_2(\sigma^{-1},\id)}"]
    \arrow[rd, dashed]
    & & {\rho_1(A,B') \otimes \rho_2(B',C)}
    \arrow[ld, dashed] \\
    & {(\rho_1 \pl \rho_2)(A,C)} &
  \end{tikzcd}
\end{equation}
Hence a bimodule map $\alpha: \rho_1 \pl \rho_2 \To \xi$ is precisely a family of maps which are invariant under the diagonal action on the middle variables similar to \eqref{eq:factor-coinvar}.
\begin{equation}
  \rho_1(A,B) \otimes \rho_2(B,C) \to \xi(A,C)
\end{equation}
\end{rmk}

\begin{definition}
  Recall that any category $\arbcat$ has an associated $\Iso(\arbcat) \mddash \Iso(\arbcat)$ bimodule in $\Set$ through $\rho_\act$ as in Example~\ref{can-bimod}.
  Similarly, define the \defn{action bimodule} $\tilde\rho_\act: \act^{\rop} \times \act \to \Set$ by
  \begin{equation}
    \tilde\rho_{\act}(A,B) = \Hom_{\act}(A,B)
  \end{equation}
  Like the canonical bimodule, define the action by precomposition and postcomposition.
  Moreover, assuming $\C$ has a free functor $F: \Set \to \C$, define $\tilde\rho_{\act}^{\C} := F \circ \tilde\rho_\act$.
\end{definition}

\begin{example}
  If $\act = \underline{G}$ and $\C = \Vect_k$, then the $\act \mddash \act$ bimodule $\tilde\rho_{\act}^{\Vect} = F \circ \tilde\rho_\act$ is $k[G]$ with the left and right action of multiplication.
\end{example}

Recall, that in the case $\C = \Set$, by what is known as the co--Yoneda lemma or density formula, for any functor $X: \C^{op} \to \Set$ there is an isomorphism $X \simeq \int^{c:\C}  \Hom(-, c) \times X(c)$, with an explicit isomorphism coming from the Yoneda lemma, see e.g.\ \cite{MacLane}. Generalizing this one obtains:

\begin{prop}
  \label{square-mon-cat}
  $\act\mddash\act$ bimodules with values in a monoidal category $\C$ with enough colimits form a monoidal category with monoidal product $\pl$.
  The unit is given by the functor $\tilde\rho_\act^{\C}$ and the unit constraints and associativity constraints are those induced by the respective constraints on $\C$.
\end{prop}
\begin{proof}
  The left unit constraint directly follows from the density formula.
  For the right unit constraint, the density formula also provides a natural isomorphism $Y \simeq \int^{c:C} Y(c) \times \Hom(c, -)$ for any functor $Y: \C \to \Set$.
  The associativity constraints follow from the so--called Fubini Theorem for co--ends, see \cite{MacLane}.

  For the other cases of $\C$ as above, the functor is given by post-composing with $F$, which is a left adjoint, to the functor $U: \C \to \Set$ and hence the isomorphisms transfer.
\end{proof}

\subsection{Bimodule Monoids}\label{sec:bimodule-monoids}

In the last section, we showed that modules form a monoidal category with plethysm as the monoidal product. In this section, we will discuss what monoids in this category look like.

\subsubsection{Definition and examples}\label{sec:Bimod-monoid-def-examples}

\begin{definition}
  A \defn{bimodule monoid} is a monoid in the monoidal category $(\act\mddash\act\text{-modules}, \pl, \tilde\rho_\act)$ meaning it is a $\act\mddash\act$-bimodule $\rho$ together with a natural transformation $\gamma: \rho \pl \rho \Rightarrow \rho$.
  It is \defn{unital} if there is a natural transformation $\eta: \tilde\rho_\act \Rightarrow \rho$, making the following diagram commute:
  \begin{equation}
    \begin{tikzcd}
      \rho \pl \tilde\rho_{\act}
      \arrow[r, "\id \pl \eta"]
      \arrow[rd, "\sim"']
      & \rho \pl \rho
      \arrow[d, "\gamma"']
      & \tilde\rho_{\act} \pl \rho
      \arrow[l, "\eta \pl \id"']
      \arrow[ld, "\sim"] \\
      & \rho &
    \end{tikzcd}
  \end{equation}
\end{definition}

Both unital and non-unital monoids occur naturally.
We will give examples of both.

\begin{example}[non-unital]
  \label{ex:glue}
  Let $\xi_{Glue}: \Iso(\FinSet)^{op} \times \Iso(\FinSet) \to \Set$ be the bimodule where $\xi_{Glue}(S,T)$ is the set of graphs with $S$ as the in-tails, $T$ as the out-tails, and unlabeled vertices.
  A typical element might look something like this:
  
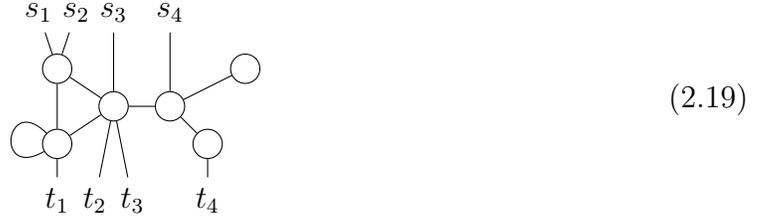
\begin{equation}
    \begin{tikzpicture}[baseline={(current bounding box.center)}]
      \node [style=none] (0) at (-1.25, 0.5) {$s_1$};
      \node [style=none] (1) at (-0.75, 0.5) {$s_2$};
      \node [style=none] (2) at (-0.25, 0.5) {$s_3$};
      \node [style=none] (3) at (-1, -2) {$t_1$};
      \node [style=none] (4) at (-0.5, -2) {$t_2$};
      \node [style=White] (5) at (-1, -0.25) {};
      \node [style=White] (6) at (-0.25, -0.75) {};
      \node [style=none] (7) at (0.5, 0.5) {$s_4$};
      \node [style=none] (8) at (0, -2) {$t_3$};
      \node [style=none] (9) at (1, -2) {$t_4$};
      \node [style=White] (10) at (-1, -1.25) {};
      \node [style=White] (12) at (0.5, -0.75) {};
      \node [style=White] (13) at (1.5, -0.25) {};
      \node [style=White] (14) at (1, -1.25) {};
      \draw (5) to (10);
      \draw [in=-150, out=135, loop] (10) to ();
      \draw (5) to (6);
      \draw (10) to (3);
      \draw (10) to (6);
      \draw (0) to (5);
      \draw (1) to (5);
      \draw (2) to (6);
      \draw (12) to (6);
      \draw (14) to (9);
      \draw (12) to (14);
      \draw (12) to (13);
      \draw (12) to (7);
      \draw (6) to (4);
      \draw (6) to (8);
    \end{tikzpicture}
  \end{equation}
  
  There is a natural monoid structure $\gamma_{Glue}: \xi_{Glue} \pl \xi_{Glue} \Rightarrow \xi_{Glue}$ given by gluing the out-tails of one graph to the in-tails of the other graph.
  Note that gluing can only lengthen the graph, so it is not possible to define a unit.
\end{example}

\begin{example}[unital]
\label{ex:cospan-monoid}
  Let $\rho_{Cospan}: \Iso(\FinSet)^{op} \times \Iso(\FinSet) \to \Set$ be the bimodule of cospans defined in Example~\ref{ex:cospan}.
  Define multiplication $\gamma: \rho_{Cospan} \pl \rho_{Cospan} \Rightarrow \rho_{Cospan}$ by composing two cospans by taking a pushout.
  Graphically, this corresponds to gluing two graphs together and then contracting edges.
  
  Define a unit $\eta: \tilde\rho_{\Iso(\FinSet)} \Rightarrow \rho_{Cospan}$ by sending an isomorphism $A \to A'$ to the equivalence class containing the cospan $A \rightarrow A' \stackrel{=}{\leftarrow} A'$.
  Graphically, this corresponds to depicting a bijection $b: X \to Y$ as a graph $\del: F \to V$ with flags $F = X \amalg Y$, vertices $V = Y$, and $\del = b \amalg \id_Y$:
  \begin{equation}
    \begin{tikzpicture}[baseline={(current bounding box.center)}]
      \node [fill=white, draw=black, shape=circle]
      (0) at (3, 0) {};
      \node [fill=white, draw=black, shape=circle]
      (1) at (4, 0) {};
      \node [fill=white, draw=black, shape=circle]
      (2) at (5, 0) {};
      \node [] (3) at (3, 1) {$a$};
      \node [] (4) at (4, 1) {$b$};
      \node [] (5) at (5, 1) {$c$};
      \node [] (6) at (3, -1) {$2$};
      \node [] (7) at (4, -1) {$3$};
      \node [] (8) at (5, -1) {$1$};
      \node [] (9) at (1, 0) {};
      \node [] (10) at (2, 0) {};
      \node [] (11) at (-0.5, 1) {$a \mapsto 2$};
      \node [] (12) at (-0.5, 0) {$b \mapsto 3$};
      \node [] (13) at (-0.5, -1) {$c \mapsto 1$};
      \draw (3) to (0);
      \draw (0) to (6);
      \draw (4) to (1);
      \draw (1) to (7);
      \draw (5) to (2);
      \draw (2) to (8);
      \draw [{|->}] (9.center) to (10.center);
    \end{tikzpicture}
  \end{equation}
\end{example}

\subsubsection{Categories and unital monoids}
\label{sec:cat-unital-monoids}

In Section~\ref{sec:cat-motivation}, we used categories as one of our motivating examples.
There, the groupoid $\Iso(\arbcat)$ ``acted'' on the morphisms of $\arbcat$ by precomposition and postcomposition.
In other words, the action came from the inclusion $\imath: \Iso(\arbcat) \to \arbcat$.
With a proper definition of plethysm at hand, we will write all of this down formally and generalize it by allowing $\imath$ to be arbitrary.
We will then show that these ``category-like'' monoids are precisely unital monoids.

\begin{prop}
  \label{cat-to-mon}
  Let $\arbcat$ be a category equipped with an identity on objects functor $\imath: \act \to \arbcat$.
  Then there is a canonical unital monoid structure for the bimodule $\rho_{\arbcat}: \act^{\rop} \times \act \to \Set$ defined by:
  \begin{equation}
    \rho_{\arbcat}(X,Y) := \Hom_\arbcat(X,Y)
  \end{equation}
  with the action given by $(\sigma \biact \tau) \cdot \phi = \imath(\tau) \circ \phi \circ \imath(\sigma)$.
\end{prop}
\begin{proof}
  The functorality of $\rho_{\arbcat}$ is clear.
  For the monoid structure, we first notice that the composition map induces a family of morphisms
  \begin{equation}
    \circ_{XYZ}: \rho(X,Y) \times \rho(Y,Z) \to \rho(X,Z)
  \end{equation}
  These maps induce a morphism $\gamma_{XZ}$ from the coinvariants with respect to the $\act^{\rop} \times \act$ action acting on the ``$Y$-slot'' to $\rho(X,Z)$.
  Hence the totality of $\{\gamma_{XZ}\}$ constitutes a natural transformation $\gamma: \rho \pl \rho \To \rho$ giving us our multiplication.

  Recall that the unit for this monoidal category of bimodules is $\tilde\rho_{\act}$.
  Unpacking the definition of the functor $\imath: \act \to \arbcat$, we have a family of maps $\tilde\rho_{\act}(A,B) \to \rho_{\arbcat}(A,B)$ which forms a natural transformation $\eta: \tilde\rho_{\act} \To \rho_{\arbcat}$.
  The left and right unitality diagrams will commute because we defined the action as $(\sigma \biact \tau) \cdot \phi = \imath(\tau) \circ \phi \circ \imath(\sigma)$.
\end{proof}

\begin{example}\label{can-monoid}
  The situation where $\imath: \Iso(\arbcat) \to \arbcat$ is the inclusion corresponds to the canonical bimodule defined in Example~\ref{can-bimod} equipped with a monoid structure coming from the composition of $\arbcat$.
  The unit $\eta: \tilde\rho_{\Iso(\arbcat)} \To \rho_{\arbcat}$ is just the inclusion.
  We will call this the \defn{canonical monoid} for the category $\arbcat$.
\end{example}

\begin{example}
  Take $\arbcat$ to be any locally small category and let $\act = \arbcat_{disc}$ be the discrete subcategory of $\arbcat$.
  Let $\imath: \arbcat_{disc} \to \arbcat$ be the unique identity on objects functor of that form.
  In this case, the image of $\eta: \tilde\rho_{\Gpd} \Rightarrow \rho$ only contains the identities.
\end{example}

\begin{example}
  If a category $\arbcat$ has an underlying groupoid $\Iso(\arbcat) \simeq \V^{\boxtimes}$ which is a free symmetric monoidal category, then there is a functor $\SS \to \Iso(\arbcat)$ coming from the commutativity constraints.
  By composing with the inclusion, we get a functor $\SS \to \arbcat$.
  Hence such a category $\arbcat$ gives rise to an $\SS \mddash \SS$ bimodule.
\end{example}

The existence of a unit $\eta: \tilde\rho_{\act} \Rightarrow \rho$ for a $\act \mddash \act$ bimodule monoid $(\rho, \gamma)$ is essentially the statement that the multiplication $\gamma$ can implement the $\act \mdash \act$ bi-action.
We will see that this condition is precisely what makes unital bimodules resemble the ``category-like'' bimodules.

\begin{prop}
  \label{mon-to-cat}
  A unital $\act \mddash \act$ bimodule monoid $(\rho, \gamma, \eta)$ determines a category $\arbcat := \arbcat(\rho, \gamma)$ with the same objects as $\act$ with morphisms defined by:
  \begin{equation}
    \Hom_{\arbcat}(A,B) := \rho(A,B)
  \end{equation}
  Moreover, the unit $\eta$ induces an identity on objects functor $\imath: \act \to \arbcat$.
\end{prop}
\begin{proof}
  Define composition by the following diagram:
  \begin{equation}
    \begin{tikzcd}[column sep = small]
      {\Hom(X,Y) \times \Hom(Y,Z)}
      \arrow[d, equal]
      \arrow[rr, "\circ", dashed]
      & & {\Hom(X,Z)} \\
      {\rho(X,Y) \times \rho(Y,Z)}
      \arrow[r]
      & {[\rho \pl \rho](X,Z)}
      \arrow[r, "{\gamma_{X,Z}}"']
      & {\rho(X,Z)} \arrow[u, equal]
    \end{tikzcd}
  \end{equation}
  The associativity axiom for a category follows from the associativity of the monoid.

  For the unit, consider the following diagram:
  \begin{equation}
    \label{eq:G-mult-F-mult-commute}
    \begin{tikzcd}
      \tilde\rho_{\act} \pl \tilde\rho_{\act}
      \arrow[d, "\sim"]
      \arrow[r, "\id \pl \eta"]
      & \tilde\rho_{\act} \pl \rho
      \arrow[r, "\eta \pl \id"]
      \arrow[d, "\sim"]
      & \rho \pl \rho
      \arrow[d, "\gamma"] \\
      \tilde\rho_{\act}
      \arrow[r, "\eta"']
      & \rho
      \arrow[r, equal]
      & \rho
    \end{tikzcd}
  \end{equation}
  The left square commutes because the transformation $\tilde\rho_{\act} \pl (-) \Rightarrow (-)$ constitutes a natural isomorphism.
  The right square commutes because $\eta$ is a unit.
  Hence the outer rectangle commutes, which amounts to the existence of a functor $\act \to \arbcat$.
\end{proof}

\begin{cor}
\label{cor:bimodcat}
  There is a one-to-one correspondence between unital monoids of bimodules $\rho: \act^{op} \times \act \to \Set$ and categories $\arbcat$ equipped with a functor $\imath: \act \to \arbcat$.
\end{cor}
\begin{proof}
  The constructions described in Proposition~\ref{cat-to-mon} and Proposition~\ref{mon-to-cat} are inverses to each other.
\end{proof}

\section{Element representations}
\label{sec:element-rep}
In practice, many bimodules $\xi: \act^{op} \times \act \to \C$ seem to be ``structured'' by the elements of another bimodule $\rho: \act^{op} \times \act \to \Set$.
This section is about one way of making this explicit.
Another way will be discussed in Section \ref{sec:relative-bimodules}.

We will consider the category of elements $\el(\rho)$ of a bimodule $\rho: \act^{op} \times \act \to \Set$ and show how a ``representation'' $D: \el(\rho) \to \C$ of the element category naturally determines a bimodule $\chi_D: \act^{op} \times \act \to \C$.
Furthermore, the category of functors $D: \el(\rho) \to \C$ can be equipped with a monoidal structure which is compatible with the plethysm product.
Hence studying these functors will be a useful way of deducing properties of bimodules of the form $\rho \cong \chi_{D}$.

\subsection{Category of elements}
\label{sec:category-of-elements}

Classically, a group action $\rho: G \times X \to X$ induces a category called the \defn{action groupoid}.
The objects are elements $x \in X$.
The morphisms are pairs $(g,x) \in G \times X$ where the source is given by projecting to $X$ and the target is given by the group action:
\begin{equation}
  \begin{tikzcd}
    x \arrow[r,"{(g,x)}"]
    & g \cdot x
  \end{tikzcd}
\end{equation}
In our setting, the corresponding construction is the category of elements.
We will give this in two versions.

\subsubsection{Set version}

Let $\rho: \act \to \Set$ be a $\Set$-valued module, then define the \defn{category of elements} $\el(\rho)$ as follows:
\begin{enumerate}
  \item The objects are pairs $(A, x)$ where $A$ is an object of $\act$ and $x \in \rho(A)$.
  \item The morphisms $(A, x) \to (B, y)$ are given by morphisms $f: A \to B$ of $\act$ such that\linebreak $\rho(f)(x) = y$.
\end{enumerate}

\begin{example}
  Consider the canonical $\Iso(\arbcat)\mddash\Iso(\arbcat)$-bimodule $\rho_\arbcat$ for a category $\arbcat$.
  If $\phi \in \rho_{\arbcat}(X,Y)$ and $\phi' \in \rho_{\arbcat}(X',Y')$ are elements, then a morphism $(\sigma \biact \tau): \phi \to \phi'$ is the same data as the following diagram:
  \begin{equation}
    \begin{tikzcd}
      X
      \arrow[d, "\sigma^{-1}"']
      \arrow[r, "\phi"]
      & Y
      \arrow[d, "\tau"] \\
      X'
      \arrow[r, "\phi'"]
      & Y'
    \end{tikzcd}
  \end{equation}
  Define the involutive functor $\Inv: \Iso(\arbcat)^{\rop} \to \Iso(\arbcat)$ by $\sigma \mapsto \sigma^{-1}$.
  Also let $\imath: \Iso(\arbcat) \to \arbcat$ denote the inclusion.
  Then $\el(\rho_{\arbcat})$ is equivalent to the comma category $(Inv \circ \imath \downarrow \imath)$.
\end{example}

\subsubsection{Monoidal version}

When the target category $\C$ is monoidal, the correct notion of an element of an object $X$ is a morphism $x: 1_{\C} \to X$.
This observation gives a convenient definition of an element category as a category under the monoidal unit $1_{\C}$.
If $\rho: \act \to \C$ is a module such that $\C$ is a monoidal category with unit $1_{\C}$, then define the \defn{category of elements} $\el(\rho)$ as the comma category $(1_{\C} \downarrow \rho)$.
This is given explicitly as follows:
\begin{enumerate}
  \item The objects are morphisms $x: 1_\C \to \rho(A)$ in the target category $\C$ such that $A$ is an object of $\act$.
  \item A morphism from $x: 1_\C \to \rho(A)$ to $y: 1_\C \to \rho(B)$ is given by a morphism $f: A \to B$ such that the following commutes:
  \begin{equation}
    \begin{tikzcd}
      & 1_{\C}
      \arrow[ld, "x"']
      \arrow[rd, "y"]
      & \\
      \rho(A)
      \arrow[rr, "\rho(f)"']
      & & \rho(B)
    \end{tikzcd}
  \end{equation}
\end{enumerate}

\begin{example}
  Let $G$ be a group and let $\rho: \underline{G} \to \kVect$ be a group representation.
  Write $V = \rho(\bullet)$.
  The objects of $\el(\rho)$ are pointings and the morphisms are commuting diagrams:
  \begin{equation}
    \begin{tikzcd}[column sep = small]
      \text{Objects}:
      & k
      \arrow[d, "x"]
      & \text{Morphisms}:
      & & k
      \arrow[ld, "x"']
      \arrow[rd, "y"] & \\
      & V
      & &
      V \arrow[rr, "\rho(g)"]
      & & V
    \end{tikzcd}
  \end{equation}
  To put this in a more familiar form, identify a pointing $x: k \to V$ with the vector $v := x(1)$.           
  The objects become elements $v \in V$ and the morphisms become arrows like the ones in the action groupoid.  For clarity, we write $\rho_{g} := \rho(g)$.
  \begin{equation}
    \begin{tikzcd}[column sep = small]
      \text{Objects}:
      & v \in V
      & \text{Morphisms}:
      & v
      \arrow[r, "{(v,g)}"]
      & {\rho_g(v)}
    \end{tikzcd}
  \end{equation}
\end{example}

\subsection{Element representations}

Motivated by Section~\ref{sec:bimod-repr}, it is natural to ask what information is encoded by the ``representations'' of the category $\el(\rho)$.
We will show that these representations encode a particular kind of bimodule.

\subsubsection{Definition and examples}

\begin{definition}
  An \defn{element representation} $D$ of a $\Set$-valued bimodule $\rho: \act^{\rop} \times \act \to \Set$ is a functor of the form
  \begin{equation}
    D: \el(\rho) \to \C
  \end{equation}
  These naturally assemble into a functor category $\Fun(\el(\rho), \C)$.
\end{definition}

\begin{example}\label{ex:graphs-from-cospans}
  Recall the bimodule of cospans defined in Example~\ref{ex:cospan}:
  \begin{equation}
    \rho_{Cospan}: \Iso(\FinSet)^{\rop} \times \Iso(\FinSet) \to \Set
  \end{equation}
  The objects of $\el(\rho_{\rCospan})$ can be visualized as graphs $x$ with labeled tails, unlabeled vertices and no edges.
  The morphisms of $\el(\rho_{\rCospan})$ are then tail relabelings.
  Define an element representation $D: \el(\rho_{\rCospan}) \to \Set$ so that $D(x)$ is the set of all graphs with labeled tails and unlabeled vertices which become $x$ after contracting edges.
\end{example}

\subsubsection{Characteristic bimodule}

In this section, we will clarify how an element representation canonically determines a bimodule.
This is closely related to what is often called the Grothendieck construction described in \cite{groupe-fondamental} which establishes an equivalence between fibrations $\mathcal E \to \arbcat$ over $\arbcat$ and pseudo-functors $\arbcat \to \mathcal Cat$.

\begin{definition}
  Let $\rho: \act^{\rop} \times \act \to \C$ be a given bimodule.
  Define the functor $\Sigma: \el(\rho) \to \act^{\rop} \times \act$ by sending a pointing $x: 1_\C \to \rho(X,Y)$ to the pair $(X,Y)$.
\end{definition}

\begin{rmk}
  Observe that for each morphism $(\sigma \biact \tau): (A,B) \to (A',B')$ in $\act^{\rop} \times \act$ and each element $x \in \rho(A,B)$, there is a unique ``lift'' $(\sigma \biact \tau): x \to x'$ in $\el(\rho)$.
  In other words, the groupoid $\act^{\rop} \times \act$ ``acts'' on the elements.
\end{rmk}

\begin{definition}
  The \defn{characteristic bimodule} $\chi_D$ of an element representation $D: \el(\rho) \to \C$ is the following left Kan extension:
  \begin{equation}
    \begin{tikzcd}
      \el(\rho)
      \arrow[rd, "\Sigma"']
      \arrow[rr, "D"]
      & {}
      \arrow[d, "\,\lambda", shorten >=0.1cm, shorten <=0.1cm, Rightarrow]
      & \C \\
      & \act^{\rop} \times \act
      \arrow[ru, "\chi_D"']
      &
    \end{tikzcd}
  \end{equation}
  We will call the canonical natural transformation $\lambda: D \Rightarrow \chi_D \circ \Sigma$ the \defn{component} for $\chi_D$.
\end{definition}

\begin{const}
  Applying a general formula for a left Kan extension, the value of $\chi_D(A, B)$ can be given explicitly as the colimit of the following diagram:
  \begin{equation}
    (\Sigma \downarrow (A,B)) \to \el(\rho) \overset{D}{\to} \C
  \end{equation}
  Formally, an object of $(\Sigma \downarrow (A,B))$ is an element $x' \in \rho(A', B')$ together with a morphism $(A', B') \to (A, B)$.
  In many of our examples, $\act = \Gpd$ is a groupoid.
  In this case, there is no harm in working directly with elements $x \in \rho(A, B)$.
\end{const}

\begin{cor}
  \label{cor:chi-is-index-enrich}
  When $\act$ is discrete, the characteristic bimodule reduces to the following coproduct:
  \begin{equation}
    \chi_D(A,B) := \coprod_{x \in \rho(A,B)} D(x)
  \end{equation}
\end{cor}
\begin{proof}
  In this case, the comma category $(\Sigma \downarrow (A,B))$ is itself a discrete category consisting of the elements of $\rho(A,B)$ as objects.
  Hence the colimit degenerates to a coproduct.
\end{proof}

\begin{example}
  Let $D$ be the element representation of graphs defined in Example~\ref{ex:graphs-from-cospans}.
  Although $\act$ is not discrete, $D(\sigma)$ happens to be the identity for each automorphism $\sigma \in Aut(x)$, so there are no further identifications.
  Hence $\chi_D(A, B)$ also happens to be a coproduct $\coprod_{x \in \rho(A, B)} D(x)$ in this case.
  In words, $\chi_D(A,B)$ is the set of graphs whose in-tails are labeled by $A$ and whose out-tails are labeled by $B$.

  Unlike in the discrete case, the action of the element representation $D$ induces a bimodule action for $\chi_D$:
  \begin{equation}
    \begin{tikzcd}
      D(x)
      \arrow[d]
      \arrow[r, "D(\sigma \biact \tau)"]
      & D(y)
      \arrow[d] \\
      {\chi_D(A,B)}
      \arrow[r, "\chi_D(\sigma \biact \tau)"', dashed]
      & {\chi_D(A', B')}
    \end{tikzcd}
  \end{equation}
\end{example}

\subsection{Element Plethysm}
\label{sec:element-plethysm}

We can understand $\chi_{(-)}$ as a functor between element representations and bimodules.
This allows us to think of a particular element representation as conveying structural information about a bimodule.
We carry this further by defining the appropriate notion of a plethysm product for an element representation.

Before giving details, let us outline what this plethysm should be conceptually.
Suppose we are given element representations $D_1, D_2: \el(\rho) \to \C$.
Our goal is to obtain a new element representation $D_1 \epl D_2: \el(\rho) \to \C$.
Loosely speaking, the action on an element $z \in \el(\rho)$ should be obtained by considering factorizations $z = \gamma(x,y)$ then applying $D_1$ on the $x$ factor and $D_2$ on the $y$ factor.
We will make this explicit in 3 steps.

\subsubsection{Step 1: Composable pairs}

Given a bimodule $\rho: \act^{\rop} \times \act \to \Set$, consider the elements of the plethysm product:
\begin{equation}
  z \in (\rho \pl \rho)(A,C) = \int^{B} \rho(A,B) \times \rho(B,C)
\end{equation}
Hence, the element $z$ can be represented as a composable pair $(x,y)$.
Moreover, two pairs $(x,y)$ and $(x',y')$ represent the same element in $(\rho \pl \rho)(A, B)$ if and only if there exists a morphism $\sigma$ in the action category $\act$ such that in $\el(\rho) \times \el(\rho)$ we can write:
\begin{equation}
  \label{eq:dia}
  \begin{tikzcd}
    (x,y')
    \arrow[d, "\id \times (\sigma \biact \id)"']
    \arrow[r, "(\id \biact \sigma) \times \id"]
    & (x', y') \\
    (x,y)
    &
  \end{tikzcd}
\end{equation}

We will use the notation $[x,y] \in (\rho \pl \rho)(A,C)$ to denote the element represented by $(x,y)$ and call $\el(\rho \pl \rho)$ the \defn{category of composable pairs}.

\subsubsection{Step 2: External tensor product \texorpdfstring{$\hat \otimes$}{(x) }}

For each element $z \in (\rho \pl \rho)(A,B)$, the diagrams of \eqref{eq:dia} forms a subcategory $\Pair(z)$ of $\el(\rho) \times \el(\rho)$.
Given two element representations $D_1, D_2: \el(\rho) \to \C$, define the \defn{external tensor} $(D_1 \hat \otimes D_2): \el(\rho \pl \rho) \to \C$ to be the following colimit:
\begin{equation}
  (D_1 \hat \otimes D_2)(z) = \colim(\Pair(z) \hookrightarrow \el(\rho) \times \el(\rho) \overset{D_1 \times D_2}{\to} \C \times \C \overset{\otimes}{\to} \C)
\end{equation}
An action $z \overset{(\sigma \biact \tau)}{\to} z'$ in the bimodule of basic pairs naturally induces a morphism $(D_1 \hat{\otimes} D_2)(z) \to (D_1 \hat{\otimes} D_2)(z')$ in $\C$, so this construction constitutes a functor $D_1 \hat{\otimes} D_2: \el(\rho \pl \rho) \to \C$.

\subsubsection{Step 3: Factorization}

We now have an object $(D_1 \hat \otimes D_2)[x,y]$ for each composable pair $[x,y]$.
To obtain an object for each element $z \in \el(\rho)$, we need to ``factorize'' $z$ into composable pairs.
To do this, observe that a monoid structure $\gamma: \rho \pl \rho \To \rho$ induces a functor $\el(\gamma): \el(\rho \pl \rho) \to \el(\rho)$ which sends a composable pair $[x,y]$ to its composition.
The ``factorization'' is achieved by extending along $\el(\gamma)$ as a left Kan extension.

\begin{definition}
  Given two element representations $D_1, D_2: \el(\rho) \to \C$, define the \defn{element plethysm} $D_1 \epl D_2$ as the following left Kan extension:
  \begin{equation}
    \begin{tikzcd}
      \el(\rho \pl \rho)
      \arrow[rr, "D_1 \hat \otimes D_2"]
      \arrow[rd, "\el(\gamma)"']
      & & \C \\
      & \el(\rho)
      \arrow[ru, "D_1 \epl D_2 := Lan_{\el(\gamma)}(D_1 \hat \otimes D_2)"', dashed] &
    \end{tikzcd}
  \end{equation}
\end{definition}

\subsection{The unit}

Just as the plethysm product depends on the multiplication of the structure bimodule, the monoidal unit for element representations will depend on the unit $\eta: \tilde\rho_{\act} \To \rho$ of its ``base'' bimodule.

\begin{definition}\label{el-rep-unit}
  For a bimodule $\zeta$, define the \defn{trivial element representation} $T: \el(\zeta) \to \C$ to be constantly the monoidal unit $1_{\C}$.
  Then define the \defn{plethysm unit} $U_{\eta}: \el(\rho) \to \C$ as the following left Kan extension:
  \begin{equation}
    \begin{tikzcd}
      \el(\tilde\rho_{\act}) \arrow[rd, "\el(\eta)"'] \arrow[rr, "T"] & & \C \\
      & \el(\rho) \arrow[ru, "U_{\eta} := \Lan_{\el(\eta)}(T)"', dashed] &
    \end{tikzcd}
  \end{equation}
\end{definition}

\begin{lemma}
  \label{kan-lemma}
  If the outer and upper triangles form a left Kan extension, then the lower triangle is also a left Kan extension.
  \begin{equation}
    \begin{tikzcd}
      \bullet \\
      \bullet && \bullet \\
      \bullet
      \arrow["A"', from=1-1, to=2-1]
      \arrow["F", from=1-1, to=2-3]
      \arrow["G"{description}, from=2-1, to=2-3]
      \arrow["B"', from=2-1, to=3-1]
      \arrow["H"', from=3-1, to=2-3]
    \end{tikzcd}
  \end{equation}
\end{lemma}
\begin{proof}
  This follows from a routine adjunction argument:
  \begin{align*}
    H &\Rightarrow Y \\
    F &\Rightarrow Y \circ B \circ A \\
    G &\Rightarrow Y \circ B \qedhere
  \end{align*}
\end{proof}

\begin{lemma}
  \label{epl-unit}
  For any element representation $D: \el(\rho) \to \C$ such that $\C$
  is closed monoidal, there are natural isomorphisms
  \begin{align*}
    L: U_{\eta} \epl D \overset{\sim}{\Rightarrow} D \\
    R: D \epl U_{\eta} \overset{\sim}{\Rightarrow} D
  \end{align*}
\end{lemma}
\begin{proof}
  We will only consider the left version since the right version will be similar.
  Consider the following diagram:
  \begin{equation}
    \begin{tikzcd}
      \el(\tilde\rho_{\act} \pl \rho)
      \arrow[rrd, "T \hat \otimes D"]
      \arrow[d, "\el(\eta \pl \id)"'] & & \\
      \el(\rho \pl \rho)
      \arrow[rr, "U_{\eta} \hat \otimes D" description]
      \arrow[d, "\el(\gamma)"'] &  & \C \\
      \el(\rho) \arrow[rru, "D"'] & &
    \end{tikzcd}
  \end{equation}

  We claim that the top triangle is a left Kan extension.
  Let $P: \el(\rho \pl \rho) \to \C$ be any functor and consider an arbitrary natural transformation:
  \begin{equation}
    T \hat \otimes D \Rightarrow P \circ \el(\eta \pl \id)
  \end{equation}
  By the definition of $\hat \otimes$, this is the same data as a family of morphisms $T(\sigma) \otimes D(y) \to P[\eta(\sigma), y]$.
  For a fixed element $y \in \el(\rho)$, we can take the adjunction to get a family $T(\sigma) \to (P[\eta(\sigma), y])^{D(y)}$.
  By the definition of $U_\eta$ as a left Kan extension, this is in bijection with a family $U_\eta(x) \to (P[x,y])^{D(y)}$.
  Taking the adjunction gives us a family $U_\eta(x) \otimes D(y) \to (P[x,y])$.
  These constitute a natural transformation $U_\eta \hat\otimes D \Rightarrow P$.
  Thus the top triangle is indeed a left Kan extension.

  Since the larger diagram commutes up to isomorphism, it is trivially a left Kan extension.
  Applying Lemma~\ref{kan-lemma}, we conclude that the bottom triangle is a left Kan extension.
  Therefore there is an isomorphism $L: U_\eta \epl D \to D$.
\end{proof}

\begin{example}
  \label{ex:faithful-unit}
  Consider the unital bimodule of cospans described in Example~\ref{ex:cospan-monoid}.
  Taking the graphical point of view, the unit $\eta: \tilde\rho_{\Iso(\FinSet)} \Rightarrow \rho_{Cospan}$ was defined by sending a bijection to its two-line graphical notation.
  Since $\el(\eta): \el(\tilde\rho_{\Iso(\FinSet)}) \to \el(\rho_{Cospan})$ is an embedding onto the bijection graphs, the unit $U_\eta$ will just be:
  \begin{equation}
    U_{\eta}(x) =
    \begin{cases}
      1_{\C}, &x \text{ is a bijection graph} \\
      0_{\C}, &x \text{ is not a bijection graph}
    \end{cases}
  \end{equation}
  Here we can see that $\chi_{U_\eta}$ is isomorphic to $\tilde\rho_{\act}^{\C}$, but this is not always the case.
  We will encounter a counterexample in Example~\ref{ex:not-faithful-unit}.
\end{example}

\subsection{Monoidal structure}

With the definition of the element plethysm $\epl$ and the unit $U_{\eta}$, it is now straightforward to show that the category of element representations is a monoidal category.
We will also show that the element plethysm $\epl$ is compatible with the ordinary plethysm $\pl$ in the sense that $\chi_{(-)}$ is a strong monoidal functor.

\begin{prop}
\label{prop:diamondproduct}
  Fix a bimodule $\rho: \act^{\rop} \times \act \to \C$ with a monoid structure $\gamma$ and a unit $\eta$.
  Then the category of element representations $D: \el(\rho) \to \C$ forms a monoidal category with $\epl$ as the monoidal product and $U_\eta$ as the unit.
\end{prop}
\begin{proof}
  We just need to check associativity.
  Start by defining the associativity constraints:
  \begin{equation}
    a_{D_1D_2D_3}:
    (D_1 \epl D_2) \epl D_3
    \To
    D_1 \epl (D_2 \epl D_3)
  \end{equation}
  To do this, apply Fubini so that the left hand side can be rendered as a single co--wedge consisting of components of the form $(D_1(x_1) \otimes D_2(x_2)) \otimes D_3(x_3)$.
  Likewise, the right hand side can be rendered as a single co--wedge with components of the form $D_1(x_1) \otimes (D_2(x_2) \otimes D_3(x_3))$.
  Then apply the associativity constraints of $\otimes$ to each component to get $a_{D_1D_2D_3}$.
  The coherences for $a$ follow from the coherences of associativity constraints for the monoidal category $\C$.
\end{proof}

\begin{as}\label{commutes-with-coproduct}
  The monoidal product $\otimes$ of the target category $\C$ commutes with coproducts.
\end{as}

\begin{lemma}
  \label{pl-epl-iso}
  There is an isomorphism of bimodules $\mu_{D_1,D_2}: \chi_{D_1} \pl \chi_{D_2} \simTo  \chi_{D_1 \epl D_2}$ natural in $D_1$ and $D_2$.
\end{lemma}
\begin{proof}
  We start by showing that the following diagram is a left Kan extension:
  \begin{equation}
    \begin{tikzcd}
      \el(\rho \pl \rho) \arrow[d, "\el(\gamma)"'] \arrow[rd, "D_1 \hat \otimes D_2"] &    \\
      \el(\rho) \arrow[d, "\Sigma"'] & \C \\
      \act^{op} \times \act \arrow[ru, "\chi_{D_1} \pl \chi_{D_2}"'] &
    \end{tikzcd}
  \end{equation}

  Suppose we have a bimodule $\xi: \act^{op} \times \act \to \C$ and a bimodule morphism $D_1 \hat \otimes D_2 \Rightarrow \xi \circ \Sigma \circ \el(\gamma)$.
  By definition, this natural transformation is a family of morphisms $(D_1 \hat \otimes D_2)[x,y] \to \xi(A,C)$ running over classes of composable pairs $[x,y]$.
  Unpacking further, this gives a family of morphisms running over the composable pairs $(x,y) \in \rho(A,B) \times \rho(B,C)$:
  \begin{equation}
    D_1(x) \otimes D_2(y) \to \xi(A,C)
  \end{equation}
  The members of this family are invariant in the ``$B$ variable'' while still respecting the actions on $A$ and $C$.

  We can then use the universal property of a coproduct to collect these maps.
  Next, we use Assumption~\ref{commutes-with-coproduct} to commute the coproduct:
  \begin{align}
    \coprod_{(x,y) \in \rho(A,B) \times \rho(B,C)} D_1(x) \otimes D_2(y) &\to \xi(A,C) \\
    \left(
      \coprod_{x \in \rho(A,B)} D_1(x)
    \right)
    \otimes
    \left(
      \coprod_{y \in \rho(B,C)} D_2(y)
    \right) &\to \xi(A,C)
  \end{align}
  Since this family respects the actions on $A$ and $C$, we can replace the coproducts with $\chi_{D_1}$ and $\chi_{D_2}$.
  Moreover, since these are invariant in the middle variable, we can take a coend in the middle.
  The net result is a family of morphisms $(\chi_{D_1} \pl \chi_{D_2})(A,C) \to \xi(A,C)$.
  Thus we have the desired a natural transformation $\chi_{D_1} \pl \chi_{D_2} \Rightarrow \xi$.

  Define $D_1 \hat \otimes D_2 \Rightarrow (\chi_{D_1} \pl \chi_{D_2}) \circ \el(\gamma)$ in the natural way.
  Whiskering with the natural transformation $\chi_{D_1} \pl \chi_{D_2} \Rightarrow \xi$ recovers the original  $D_1 \hat \otimes D_2 \Rightarrow \xi \circ \Sigma \circ \el(\gamma)$.
  Therefore this diagram is a left Kan extension, so there is a natural isomorphism $\mu: \chi_{D_1} \pl \chi_{D_2} \Rightarrow \chi_{D_1 \epl D_2}$.
\end{proof}

\begin{prop}
  If $\eta$ is a faithful unit, then there is a canonical isomorphism $\epsilon: \tilde\rho_{\act}^\C \simTo \chi_{U_{\eta}}$ of bimodules.
\end{prop}
\begin{proof}
  For clarity of notation, let $\Sigma_{\zeta}: \el(\zeta) \to \act^{\rop} \times \act$ denote the fibration for a bimodule $\zeta: \act^{\rop} \times \act \to \Set$.
  Consider the following diagram where $T$ is the trivial element representation described in Definition~\ref{el-rep-unit}:
  \begin{equation}
    \begin{tikzcd}
      \el(\tilde\rho_{\act}) \arrow[rrd, "T"] \arrow[d, "\el(\eta)"'] & & \\
      \el(\rho) \arrow[rr, "U_{\eta}" description] \arrow[d, "\Sigma_{\rho}"'] &  & \C \\
      \act^{\rop} \times \act \arrow[rru, "\tilde\rho_{\act}^{\C}"']
      & &
    \end{tikzcd}
  \end{equation}
  Since $\eta$ is faithful, the composition $\el(\tilde\rho_\act) \overset{\el(\eta)}{\to} \el(\rho) \overset{\Sigma_{\rho}}{\to} \act^{\rop} \times \act$ is equal to $\Sigma_{\tilde\rho_{\act}}: \el(\tilde\rho_{\act}) \to \act^{\rop} \times \act$.
  Hence the outer-most triangle is a left Kan extension by the definition of $\tilde\rho_{\act}^{\C}$.
  The top triangle is a left Kan extension by definition.
  By Lemma~\ref{kan-lemma}, the bottom triangle is another left Kan extension.
  By uniqueness of Kan extensions, there is a unique natural isomorphism $\epsilon: \tilde\rho_{\act}^\C \simTo \chi_{U_{\eta}}$.
\end{proof}

\begin{example}[Non-example]
  \label{ex:not-faithful-unit}
  We will modify Example~\ref{ex:faithful-unit} so that $\zeta_{\rCospan}: \SS^{\rop} \times \SS \to \Set$ is a bimodule of cospans with a \emph{trivial action} and define the unit $\eta: \tilde\rho_{\SS} \To \zeta_{\rCospan}$ so that everything in $\tilde\rho_{\SS}(n,n)$ is sent to the two-line graphical notation of $\id_{n}$.
  Now we compute:
\begin{equation}
    U_{\eta}(x) = \colim( \eta \downarrow x \to \el(\tilde\rho_{\SS}) \overset{T}{\to} \C)
  \end{equation}
  The category $\eta \downarrow x$ is empty unless $x$ is an identity graph.
  If $x$ is the identity graph with $n$ vertices, then $\eta \downarrow x$ is essentially the connected component of $\el(\tilde\rho_{\SS})$ of all order $n$ permutations.
  In particular, this means that $\chi_{U_{\eta}}(2,2) \cong U_{\eta}(\id_2) \cong 1_{\C}$.
  However $\tilde\rho_{\SS}^{\C}(2,2) \cong T(\id_2) \oplus T(\tau) \cong 1_{\C} \oplus 1_{\C}$ where $\tau$ is transposition.
  Hence the two bimodules are not isomorphic.
\end{example}

\begin{lemma}
   If the unit $\eta$ is faithful, then the bimodule unit and the element representation unit are compatible in the sense that the following diagram commutes:
\begin{equation}
  \begin{tikzcd}
    {\tilde\rho_{\act}^{\C} \pl \chi_D}
    & {\chi_{U_\eta} \pl \chi_D} \\
    {\chi_D}
    & {\chi_{U_\eta \epl D}}
    \arrow["\sim"', Rightarrow, from=1-1, to=2-1]
    \arrow["{\epsilon \pl \id}", Rightarrow, from=1-1, to=1-2]
    \arrow["\mu", Rightarrow, from=1-2, to=2-2]
    \arrow["\chi_L", Rightarrow, from=2-2, to=2-1]
  \end{tikzcd}
\end{equation}
\end{lemma}
\begin{proof}
  Apply these functors to a fixed object $(A,C) \in Ob(\act^{\rop} \times \act)$ and write out the coends on the top:
  \begin{equation}
  \begin{tikzcd}
    {\int^B \tilde\rho_{\act}^{\C}(A,B) \otimes \chi_D(B,C)}
    & {\int^B \chi_{U_\eta}(A,B) \otimes \chi_D(B,C)} \\
    {\chi_D(A,C)}
    & {\chi_{U_\eta \epl D}(A,C)}
    \arrow["\sim"', from=1-1, to=2-1]
    \arrow["{\int^B \epsilon \otimes \id}", from=1-1, to=1-2]
    \arrow["\mu", from=1-2, to=2-2]
    \arrow["\chi_L", from=2-2, to=2-1]
  \end{tikzcd}
\end{equation}
Moreover, since the top row consists of coends, it is enough to just work with the co--wedges:
 \begin{equation}
  \begin{tikzcd}
    {\tilde\rho_{\act}^{\C}(A,B) \otimes \chi_D(B,C)}
    & {\chi_{U_\eta}(A,B) \otimes \chi_D(B,C)} \\
    {\chi_D(A,C)}
    & {\chi_{U_\eta \epl D}(A,C)}
    \arrow["\sim"', from=1-1, to=2-1]
    \arrow["{\epsilon \otimes \id}", from=1-1, to=1-2]
    \arrow["\mu", from=1-2, to=2-2]
    \arrow["\chi_L", from=2-2, to=2-1]
  \end{tikzcd}
\end{equation}
The $\sim$ morphism is essentially the application of the bimodule action.
Comparing this to the definition of $L$ makes it clear that the diagram commutes.
\end{proof}

\begin{thm}\label{thm:strongmonoidal}
  If the unit of a bimodule monoid has a faithful unit, then $(\chi, \mu, \epsilon)$ is a strong monoidal functor. \qed
\end{thm}

\section{Relative bimodules}\label{sec:relative-bimodules}

\subsection{Relative bimodules}

In the previous sections, we saw how one bimodule $\xi$ can be ``based'' off of another bimodule $\rho$ in the sense that $\xi = \chi_D$ for some element representation $D: \el(\rho) \to \C$.
Another variation on this theme is what we call relative bimodules.
We will describe the relation between these two approaches.

\begin{definition}
  A \defn{relative bimodule} $\xi \overset{\pi}{\To} \rho$ of a fixed bimodule $\rho$ is a pair $(\xi, \pi)$ such that $\xi : \act^{\rop} \times \act \to \C$ is a bimodule and $\pi: \xi \To \rho$ is a natural transformation.
  These naturally assemble into a slice category over a fixed bimodule $\rho$.
\end{definition}

The term ``relative bimodule'' matches Joyal's terminology of ``relative species'' defined in \cite{species}.
We will discuss this in more detail in Example~\ref{Joyal-species}.
This terminology also evokes the notion of relative schemes in algebraic geometry.

\begin{example}\label{ex:glue-pi}
  Recall that we defined $\xi_{Glue}$ as the bimodule of graphs and $\rho_{\rCospan}$ was the bimodule of corollas.
  The bimodule $\xi_{\Glue}$ is naturally a bimodule relative to $\rho_{\rCospan}$ with the projection $\pi: \xi_{Glue} \To \rho_{\rCospan}$ defined by contracting all of the edges of an element and sending the result to the associated cospan.
\end{example}

\subsection{Relative plethysm}
\label{sec:relative-plethysm}

If the base bimodule $\rho$ has a monoid structure $\gamma: \rho \pl \rho \To \rho$ and $\eta: \tilde\rho_{\act} \To \rho$, then there is a natural notion of a plethysm for a relative bimodule.
Given two relative bimodules $(\xi_1 \overset{\pi_1}{\To} \rho)$ and $(\xi_2 \overset{\pi_2}{\To} \rho)$ over $\rho$, we can define the \defn{relative plethysm} for $\rho$ as follows:
\begin{equation}
  (\xi_1 \overset{\pi_1}{\To} \rho)
  \pl_\rho
  (\xi_2 \overset{\pi_2}{\To} \rho)
  := (\xi_1 \pl \xi_2 \overset{\pi}{\To} \rho)
\end{equation}
Where $\pi$ is the composition: $\xi_1 \pl \xi_2 \overset{\pi_1 \pl \pi_2}{\To} \rho \pl \rho \overset{\gamma}{\To} \rho$.

\begin{prop} \label{prop:relative-plethysm}
  The bimodules relative to $\rho$ with a fixed unital monoid $\gamma: \rho \pl \rho \To \rho$ form a monoidal category with $\pl_\rho$ as the monoidal product and $\tilde\rho_{\act} \overset{\eta}{\To} \rho$ as the unit.
\end{prop}
\begin{proof}
  To show associativity, consider the following relative bimodules:
  \begin{equation}
    \xi_1 \overset{\pi_1}{\To} \rho, \quad \xi_2 \overset{\pi_2}{\To} \rho,\quad \xi_3 \overset{\pi_3}{\To} \rho
  \end{equation}
  By the Fubini theorem for coends, we have $(\xi_1 \pl \xi_2) \pl \xi_3 \cong \xi_1 \pl (\xi_2 \pl \xi_3)$.
  Moreover, the two $\pi$-maps coincide because $\gamma$ is associative.

  Next, we will show that $\tilde\rho_{\act} \overset{\eta}{\To} \rho$ is the unit relative bimodule.
  By Proposition~\ref{square-mon-cat}, we have a natural isomorphism $\tilde\rho_{\act} \pl \xi \simTo \xi$.
  To show that this is an isomorphism \emph{over} $\rho$, consider the following diagram:
  \begin{equation}
    \begin{tikzcd}
      \tilde\rho_{\act} \pl \xi
      \arrow[r, "\sim", Rightarrow]
      \arrow[d, "\id \pl \pi"', Rightarrow]
      & \xi \arrow[d, "\pi", Rightarrow] \\
      \tilde\rho_{\act} \pl \rho
      \arrow[d, "\eta \pl \id"', Rightarrow]
      \arrow[r, "\sim", Rightarrow]
      & \rho \\
      \rho \pl \rho
      \arrow[ru, "\gamma"', Rightarrow]
      &
    \end{tikzcd}
  \end{equation}
  By definition, the projection map for the plethysm is the sequence of maps going counter-clockwise on the outside.
  By a routine argument involving the properties of the unit, the above diagram commutes.
  Thus $\tilde\rho_{\act} \pl \xi \simTo \xi$ induces an isomorphism $\tilde\rho_{\act} \pl \xi \to \xi$ in the category of relative bimodules over $\rho$.
  Likewise, we have an isomorphism $\xi \pl \tilde\rho_{\act} \to \xi$.
\end{proof}

\begin{definition}
  A \defn{relative-monoid} is a monoid object in the monoidal category of relative bimodules over $\rho$.
  More explicitly, a relative-monoid is the following data:
  \begin{enumerate}
    \item A relative bimodule $\xi \overset{\pi}{\To} \rho$.
    \item A multiplication $\gamma_{\xi}: \xi \pl \xi \To \xi$ making the following diagram commute:
      \begin{equation}\label{eq:rel-multiply}
        \begin{tikzcd}
          \xi \pl \xi
          \arrow[d, "\pi \pl \pi"', Rightarrow]
          \arrow[r, "\gamma_{\xi}", Rightarrow]
          & \xi \arrow[d, "\pi", Rightarrow] \\
          \rho \pl \rho
          \arrow[r, "\gamma"', Rightarrow]
          & \rho
        \end{tikzcd}
      \end{equation}
    \item A relative monoid is \defn{unital} if there exists a natural transformation $\tilde \eta: \tilde\rho_{\act} \To \xi$ such that $\pi \circ \tilde \eta = \eta$ and the obvious unitality diagram commutes.
  \end{enumerate}
\end{definition}

\begin{example}
  Using the relative bimodule $\xi_{\Glue} \overset{\pi}{\To} \rho_{\rCospan}$ of example~\ref{ex:glue-pi}, we can define a canonical monoid structure $\gamma_{\Glue}: \xi_{\Glue} \pl_{\rho} \xi_{\Glue} \to \xi_{\Glue}$ by gluing the out-tails on one graph to the in-tails of the other graph.

  To better understand why the projection is defined the way it is, consider the diagram~\eqref{eq:rel-multiply} applied to this particular situation:
  \begin{equation}
    \begin{tikzcd}
      \xi_{\Glue} \pl \xi_{\Glue}
      \arrow[d, "\pi \pl \pi"', Rightarrow]
      \arrow[r, "\gamma_{\Glue}", Rightarrow]
      & \xi_{Glue}
      \arrow[d, "\pi", Rightarrow] \\
      \rho_{\rCospan} \pl \rho_{\rCospan}
      \arrow[r, "\gamma"', Rightarrow]
      & \rho_{\rCospan}
    \end{tikzcd}
  \end{equation}
  In other words, this gluing operation $\gamma_{\Glue}$ is coherent with $\gamma$ because doing the graph gluing operation first then contracting is the same as contracting first and then composing cospans.
\end{example}

These relative monoids have an interpretation in our category construction discussed in Proposition~\ref{mon-to-cat}.
\begin{prop}
\label{prop:indexingfunctor}
  An associative unital relative-monoid $\gamma: \xi \pl_\rho \xi \to \xi$ of a relative bimodule $\xi \overset{\pi}{\To} \rho$ naturally determines a functor $F_{\pi}: \arbcat(\gamma_{\xi}) \to \arbcat(\gamma)$.
\end{prop}
\begin{proof}
  Define $\arbcat(\gamma_{\xi})$ just like we did in Proposition~\ref{mon-to-cat} by treating $\gamma_{\xi}$ as a monoid without a relative structure.
  Define the functor $F_{\pi}: \arbcat(\gamma_{\xi}) \to \arbcat(\gamma)$ to be the identity on objects.
  On morphisms, use the natural transformation $\pi: \xi \To \rho$ to define a correspondence between the hom-sets $\Hom_{\arbcat(\gamma_{\xi})}(X,Y) := \xi(X,Y)$ and $\Hom_{\arbcat(\gamma)}(X,Y) := \rho(X,Y)$.
  The functorality of $F_{\pi}$ reduces to the commutativity of the following diagram:
  \begin{equation}
    \label{eq:functorality}
    \begin{tikzcd}
      \xi \pl \xi
      \arrow[r, Rightarrow]
      \arrow[d, Rightarrow]
      & \rho \pl \rho
      \arrow[d, Rightarrow] \\
      \xi
      \arrow[r, Rightarrow]
      & \rho
    \end{tikzcd}
  \end{equation}
  However, the commutativity of \eqref{eq:functorality} is immediate from the monoid structure $\gamma_{\xi}: \xi \pl_\rho \xi \to \xi$.
\end{proof}

\subsection{Relation to element representations}
To establish the connection between element representations and relative bimodules, we first observe that the $\chi$ functor is closely related to the free construction.

\begin{as}
  For our target category $\C$, we assume that the forgetful functor $U: \C \to \Set$ defined by $X \mapsto \Hom(1_\C, X)$ has a left adjoint $F: \Set \to \C$, the free functor.
\end{as}

\begin{lemma}
  For a bimodule $\rho: \act^{op} \times \act \to \Set$, there is a canonical natural isomorphism $\chi_{T} \overset{\sim}{\Rightarrow} F\rho$.
\end{lemma}
\begin{proof}
  Consider a bimodule morphism $\chi_T \Rightarrow \xi$.
  By the definition of the chi construction, this yields a morphism $T \Rightarrow \xi \circ \Sigma$ of element representations.
  Unpacking the definition, this is a family $\{1_\C \to \xi(A, B)\}_{x \in \rho(A,B)}$.
  This is the same thing as a bimodule morphism $\rho \Rightarrow U\xi$.
  By the free/forgetful adjunction, this becomes $F \rho \Rightarrow \xi$.
  By uniqueness of Kan extensions, there is a natural isomorphism $\chi_{T} \Rightarrow F\rho$.
\end{proof}

\begin{cor}
\label{cor:weighttorelative}
  A natural transformation $\epsilon: D \Rightarrow T$ induces a relative bimodule $\chi_{D} \overset{\pi_{\epsilon}}{\Rightarrow} F\rho$ with $\pi_{\epsilon}$ defined as the composition $\chi_D \overset{\chi_{\epsilon}}{\Rightarrow} \chi_{T} \overset{\sim}{\Rightarrow} F\rho$.
  \qed
\end{cor}

\begin{example}
  The natural transformation $\epsilon: D \Rightarrow T$ is more interesting in cases like $\C = \Vect$ where $1_\C$ is not terminal.
  In the vector space case, an evaluation $\epsilon: D \Rightarrow T$ is a family of linear maps $\epsilon_x: D(x) \to k$.
  In other words, $\epsilon$ is a family of dual vectors $\{ \epsilon_x \in D(x)^* \}_{x \in \el(\rho)}$.
  We should then think of $\epsilon$ as a chosen way to ``weight'' or ``integrate'' the element representation $D$.
\end{example}

\begin{cor}
\label{cor:weightsmon}
  These ``weights'' have a natural monoidal product:
  \begin{equation}
    \begin{tikzcd}[row sep = small]
      {D_1 \epl D_2} & {T \epl T} & T \\
      {\chi_{D_1} \pl \chi_{D_2}} & {\chi_{T} \pl \chi_{T}} & {\chi_{T}}
      \arrow["{\epsilon_1 \epl\epsilon_2}", Rightarrow, from=1-1, to=1-2]
      \arrow[Rightarrow, from=1-2, to=1-3]
      \arrow["{\chi_{\epsilon_1} \pl \chi_{\epsilon_2}}"', Rightarrow, from=2-1, to=2-2]
      \arrow[Rightarrow, from=2-2, to=2-3]
    \end{tikzcd}
  \end{equation}
  Noting that $\rho \cong \chi_T$, we see that they ultimately lead to the same definition as the relative plethysm. \qed
\end{cor}

\subsection{Equivalence in the Cartesian closed case}
\label{sec:equiv-cart-closed}

There is an analog of the Grothendieck construction $\Sigma: \el(\rho) \to \act^{op} \times \act$ in any context admitting a consistent comprehension scheme, see \cite{BergerKaufmann}.
In such a setting, there is an equivalence between ``$\el(X)$-structures'' and ``structures over $X$''.
This idea will apply in our setting if the target category $\C$ is suitably $\Set$-like.
We will work out the details of our special case explicitly.

\begin{thm}
  \label{thm:elementrelative}
  If $\C$ is Cartesian closed, then the notions of element representations $D: \el(\rho) \to \C$ and relative bimodules $\xi \overset{\pi}{\Rightarrow} F\rho$ coincide for a fixed base bimodule $\rho$.
\end{thm}
\begin{proof}
  First consider an element representation $D: \el(\rho) \to \C$.
  Since $\C$ is Cartesian closed, the unit is terminal so there is just one choice of evaluation data $\epsilon: D \Rightarrow T$.
  Hence any element representation $D$ induces a canonical relative bimodule $\xi_D \overset{\pi_D}{\Rightarrow} F\rho$ as follows:
  \begin{equation}
    \begin{tikzcd}
      \xi_D
      \arrow[d, "\pi_D"', dashed]
      \arrow[r, equal]
      & \chi_D
      \arrow[d, "\chi_{\epsilon}"] \\
      F\rho
      \arrow[r, "\sim"']
      & \chi_T
    \end{tikzcd}
  \end{equation}

  Conversely, consider a relative bimodule $\xi \overset{\pi}{\Rightarrow} F\rho$.
  Each element $x \in \rho(A,B)$ is formally a pointing $x: pt \to \rho(A,B)$, so we can define $D(x)$ as the following pullback:
  \begin{equation}
    \begin{tikzcd}
      D(x)
      \arrow[d, dashed]
      \arrow[r, dashed]
      & F(pt)
      \arrow[d, "Fx"] \\
      {\xi(A,B)}
      \arrow[r, "\pi"']
      & {F\rho(A,B)}
    \end{tikzcd}
  \end{equation}
  The objects $D(x)$ assemble into a functor $D: \el(\rho) \to \C$.
  Taking colimits, we get
  \begin{equation}
    \begin{tikzcd}
      {\chi_D(A,B)}
      \arrow[d]
      \arrow[r]
      & {\chi_T(A,B)}
      \arrow[d, "\sim"] \\
      {\xi(A,B)}
      \arrow[r, "\pi"']
      & {F\rho(A,B)}
    \end{tikzcd}
  \end{equation}
  By pullback stability of $\C$, this is also a pullback.
  Hence $\chi_D(A,B) \to \xi(A,B)$ is an isomorphism.
\end{proof}

\begin{example}
  \label{Joyal-species}
  Recall that Joyal defines a \defn{species} to be a functor $M:\Iso(\FinSet) \to \Set$.
  An example would be a functor $G$ that sends a finite set of labels $L$ to the collection of all graphs labeled by $L$.
  In \cite{species}, Joyal calls a functor of the form $T_M: \el(M) \to \Set$ a \defn{relative species} to the species $M$ (note this differs from our terminology, we would call this an element representation).
  An example might be the relative species $O: \el(G) \to \Set$ that sends a graph $\G \in \el(G)$ to the set $O(\G)$ of all orientations of $\G$.
  In other words, $O(\G)$ is the set of all directed graphs whose underlying graph is $\G$.
  Joyal proves that the category of relative species $T_M: \el(M) \to \Set$ is equivalent to the category of species over $M$.
  In our framework, $O$ would induce a species $D = \chi_O$ of directed graphs.
  Moreover, the canonical $O \to T$ induces the natural transformation $\pi: D \Rightarrow G$ sending each directed graph to its underlying graph.
\end{example}

\begin{example}
  Recall the bimodule $\xi_{Glue}: \Iso(\FinSet)^{op} \times \Iso(\FinSet) \to \Set$ described in Example~\ref{ex:glue} such that $\xi_{Glue}(A, B)$ is the set of graphs with in-tails labeled by $A$ and out-tails labeled by $B$.
  Also recall the canonical projection $\pi: \xi_{Glue} \Rightarrow \rho_{Cospan}$ which contracts the edges of a graph and send the result to the corresponding cospan described in Example~\ref{ex:glue-pi}.
  We can define an element representation $D: \el(\rho_{Cospan}) \to \Set$ by taking $D(x) := \pi^{-1}(x)$ for each cospan $x \in \el(\rho_{Cospan})$.
  This recovers the element representation in Example~\ref{ex:graphs-from-cospans}.
  Conversely, starting with $D$, there is a unique evaluation $D \Rightarrow T$ which induces the relative bimodule $\xi_{Glue} \Rightarrow \rho_{Cospan}$.
\end{example}

\section{Unique factorization}
 
A good example of unique factorization is the bimodule of cospans $\rho_{\rCospan}$.
Intuitively, we can say that an element of $\rho_{\rCospan}$ is assembled from a collection of connected cospans.
If we collect the connected cospans into a submodule $\nu_{\rCospan}$, there is a certain sense in which the sub-bimodule $\nu_{\rCospan}$ ``generates'' the bimodule $\rho_{\rCospan}$.
In this section, we will formalize this process by introducing the notion of a horizontal extension of a functor which is defined as a left Kan extension.
We will then show how to adapt this idea to monoids.

\subsection{Symmetric monoidal categories}

For the horizontal extension construction, the relationship between symmetric monoidal categories and free symmetric monoidal categories will be important.
We recall a few relevant facts.

\begin{definition}
  A \defn{symmetric monoidal category} $(\C, \otimes)$ is a category $\C$ equipped with the following structures:
  \begin{enumerate}
  \item A functor $\otimes: \C \times \C \to \C$ called the \defn{monoidal product}.
    Hence, for any pair of objects $X$ and $Y$, there is a new object $X \otimes Y$.
    Likewise, for each pair of morphisms $\phi: X \to Y$ and $\psi: X' \to Y'$ there is a new morphism $\phi \otimes \psi: X \otimes X' \to Y \otimes Y'$.
  \item A \defn{unit object} $1_\C$ together with isomorphisms called \defn{unit constraints}: $\lambda: X \otimes 1_{\C} \overset{\sim}{\to} X$ and $\rho: 1_{\C} \otimes X \overset{\sim}{\to} X$.
  \item A collection of isomorphisms called \defn{associativity constraints}:
    \begin{equation*}
      A_{XYX}: (X \otimes Y) \otimes Z \overset{\sim}{\to} X \otimes (Y \otimes Z)
    \end{equation*}
  \item A collection of isomorphisms called \defn{commutativity constraints}:
    \begin{equation*}
      C_{XY}: X \otimes Y \overset{\sim}{\to} Y \otimes X
    \end{equation*}
  \end{enumerate}
  These must satisfy certain conditions.
\end{definition}

\begin{ex}
  The category $\Set$ can be equipped with the Cartesian product as a monoidal product.
  The unit is a fixed singleton set $pt$.
\end{ex}

\begin{definition}
  Given any category $\arbcat$, define the \defn{free symmetric monoidal construction} $\arbcat^{\boxtimes}$ as follows:
  \begin{enumerate}
    \item The objects are bracketed words of objects in $\arbcat$. As notation, we separate the letters by $\boxtimes$:
    \begin{equation}
      (X_1 \boxtimes (X_2 \boxtimes X_3)) \boxtimes X_4
    \end{equation}
    \item The morphisms are likewise generated by bracketed words of morphisms in $\arbcat$ together with the naturally defined associativity and commutativity constraints.
    \item We denote the monoidal product by $\boxtimes$ and define it as concatenation of words.
  \end{enumerate}
\end{definition}

\begin{lemma}
  \label{lemma:structure-functor}
  For a symmetric monoidal category $\M$, there is a canonical functor $\mu_{\M}: \M^{\boxtimes} \to \M$ which turns a formal tensor product into a tensor product in $\M$:
  \begin{equation}
    (X_1 \boxtimes (X_2 \boxtimes X_3)) \boxtimes X_4
    \mapsto
    (X_1 \otimes (X_2 \otimes X_3)) \otimes X_4
  \end{equation}
  We will call $\mu_{\M}: \M^{\boxtimes} \to \M$ the \defn{structure functor} of the monoidal category $\M$. \qed
\end{lemma}

\subsection{Horizontal extension}

Motivated by the graph example, one might start with a bimodule $\rho$ and look for a sub-bimodule $\nu$ such that the ``internal tensor product'' of $\nu$ is the whole bimodule $\rho$.
However, ultimately we are more interested in the case where $\rho$ is ``freely generated'' so an ``external'' construction will be more natural for our purposes.
\label{lemma-structure-functor}

\begin{definition}
  Suppose $F: \M \to \C$ is a functor such that both $\M$ and $\C$ are monoidal categories.
  Let $\mu_{\M}$ and $\mu_{\C}$ be the structure functors described in Lemma~\ref{lemma:structure-functor}.
  Then define the \defn{horizontal extension} $F^{\otimes}: \M \to \C$ as the following left Kan extension, if it exists:
  \begin{equation}
    \label{eq:hor-ext-def}
    \begin{tikzcd}
      \M^\boxtimes
      \arrow[r, "F^{\boxtimes}"]
      \arrow[d, "\mu_{\M}"']
      & \C^\boxtimes
      \arrow[d, "\mu_{\C}"]
      \arrow[ld, Rightarrow, shorten >=0.3cm,shorten <=0.4cm, "\eta"'] \\
      \M
      \arrow[r, "F^{\otimes}"', dashed]
      & \C
    \end{tikzcd}
  \end{equation}
\end{definition}

\begin{const}
  \label{coend-formula}
  The left Kan extension in Equation~\eqref{eq:hor-ext-def} can be computed using a standard coend formula:
  \begin{equation}
    F^{\otimes}(x) =
    \int^{{\bigboxtimes}_i x_i : \M^{\boxtimes}}
    \Hom\left(\bigotimes_i x_i, x\right)
    \odot \left( F(x_1) \otimes \ldots \otimes F(x_k) \right)
  \end{equation}
  Where $\odot$ is the copower mentioned in Definition~\ref{copower}.
\end{const}

\begin{const}
  \label{coeq-formula}
  The horizontal extension of $F: \M \to \C$ can be computed as the following coequalizer:
  \begin{equation}
    \begin{tikzcd}
      {
        \coprod_{(f_i: x_i \to y_i)}
        \Hom(\bigotimes_i y_i, z)
        \odot
        (F(x_1) \otimes \ldots \otimes F(x_k))
      }
      \arrow[d, "(\bigboxtimes_i f_i)^*", shift left]
      \arrow[d, "(\bigboxtimes_i f_i)_*"', shift right] \\
      {
        \coprod_{\bigboxtimes_i z_i}
        \Hom(\bigotimes_i z_i, z)
        \odot
        (F(z_1) \otimes \ldots \otimes F(z_k))
      }
      \arrow[d, dashed] \\
      F^{\otimes}(z)
    \end{tikzcd}
  \end{equation}
\end{const}

\subsection{Examples of horizontal extensions}

We will compute several horizontal extensions starting with the trivial bimodule:
\begin{align*}
  \tau_\act: \act^{\rop} \times \act &\to \Set \\
  (A,B) &\mapsto pt
\end{align*}

\begin{ex}[Non-symmetric]
  \label{non-symmetric-ext-example}
  For the purpose of demonstration, we will first consider the discrete category of natural numbers $\N$ with addition as a monoidal product.
  This requires us to make a straightforward modification to \eqref{eq:hor-ext-def}, Construction~\ref{coend-formula}, and Construction~\ref{coeq-formula} so they use the free monoidal category rather than the free symmetric monoidal category.

  To start the computation, we use the coend formulation given by Construction~\ref{coend-formula}:
  \begin{equation}
    \label{eq:coend-pt-disc}
    \tau_{\N}^{\otimes}(n,m) \cong
    \int^{\bigboxtimes_i (n_i, m_i) }
    \Hom_{\N^{\rop} \times \N} \left( \sum_i(n_i, m_i), (n,m) \right)
    \odot (pt \times \ldots \times pt)
  \end{equation}
  Note that the set $\Hom_{\N^{\rop} \times \N} \left( \sum_i(n_i, m_i), (n,m) \right)$ is a singleton if both $\sum n_i = n$ and $\sum m_i = m$ and it is empty otherwise.

  Since we are considering the non-symmetric version of this construction, we do not need to consider co--wedges coming from commutativity constraints because there are none.
  Because $\N$ is discrete, we are reduced to the following:
  \begin{equation}
    \tau_{\N}^{\otimes}(n,m) \cong \coprod_{\bigboxtimes_i(n_i, m_i)} \Hom_{\N^{\rop} \times \N} \left( \sum_i(n_i, m_i), (n,m) \right)
    \odot pt
  \end{equation}

  Hence $\tau_{\N}^{\otimes}(n,m)$ can be understood as the set of partitions $\sum_i (n_i, m_i) = (n,m)$.
  We can visualize an element of $\tau_{\N}^{\otimes}(n,m)$ as vertices with unmarked flags:
  \begin{equation}
    \begin{tikzpicture}[baseline={(current bounding box.center)}]
      \node [style=White] (0) at (-4, 0) {};
      \node [style=White] (1) at (-4, -1) {};
      \node [style=White] (2) at (-4, -2) {};
      \node [style=none] (3) at (-5, 0.25) {};
      \node [style=none] (4) at (-5, -0.25) {};
      \node [style=none] (5) at (-3, 0) {};
      \node [style=none] (6) at (-3, -2) {};
      \node [style=none] (7) at (-3, -0.75) {};
      \node [style=none] (8) at (-3, -1.25) {};
      \node [style=none] (9) at (-5, -1) {};
      \draw (3.center) to (0);
      \draw (4.center) to (0);
      \draw (0) to (5.center);
      \draw (1) to (7.center);
      \draw (1) to (8.center);
      \draw (9.center) to (1);
      \draw (2) to (6.center);
    \end{tikzpicture}
  \end{equation}
\end{ex}

\begin{ex}[Symmetric]\label{symmetric-ext-example}
  Now consider the symmetric monoidal category $\SS$ where the objects are the natural numbers and addition is the monoidal product.
  The coend formula of Construction~\ref{coend-formula} gives us the following:
  \begin{align*}
    \tau_\SS^{\otimes}(n,m)
    &\cong \int^{\bigboxtimes_i (n_i, m_i) }
      \Hom_{\SS^{\rop} \times \SS} \left( \sum_i(n_i, m_i), (n,m) \right)
      \odot (pt \times \ldots \times pt) \\
    &\cong \int^{\bigboxtimes_i (n_i, m_i) }
      \Hom_{\SS^{\rop} \times \SS} \left( \sum_i(n_i, m_i), (n,m) \right)
      \odot pt
  \end{align*}
  Using the coequalizer formula of Construction~\ref{coeq-formula}, we can further unpack this:
  \begin{equation}
    \begin{tikzcd}
      {
        \coprod_{(\sigma_i \in Aut(n_i, m_i))}
        \Hom_{\SS^{\rop} \times \SS}(\sum_i (n_i, m_i), (n,m))
        \odot
        pt
      }
      \arrow[d, "(\bigboxtimes_i \sigma_i)_* = Id", shift left]
      \arrow[d, "(\bigboxtimes_i \sigma_i)^*"', shift right] \\
      {
        \coprod_{\boxtimes_i (n_i,m_i)}
        \Hom_{\SS^{\rop} \times \SS}(\sum_i (n_i, m_i), (n,m))
        \odot
        pt
      }
      \arrow[d, dashed] \\
      \tau_{\SS}^{\otimes}(n,m)
    \end{tikzcd}
  \end{equation}
  Formally speaking, this means that an element $x \in \tau_{\SS}^{\otimes}(n,m)$ can be represented as two pieces of data: a partition $\sum_i (n_i, m_i) = (n,m)$ and a pair of permutations $\lambda \in \SS_n$ and $\rho \in \SS_m$.

  Building upon the graphical notation of Example~\ref{non-symmetric-ext-example}, we can incorporate the permutation data in this representation as the ``wires'' crossing:
  
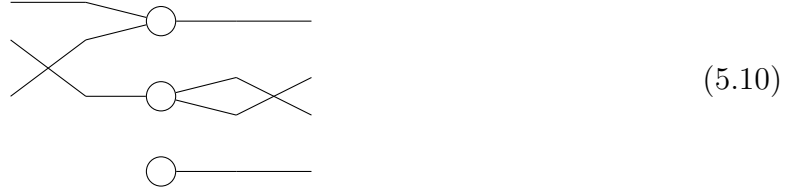
\begin{equation}
    \label{eq:corolla-crossed}
    \begin{tikzpicture}[baseline={(current bounding box.center)}]
      \node [style=White] (0) at (-4, 0) {};
      \node [style=White] (1) at (-4, -1) {};
      \node [style=White] (2) at (-4, -2) {};
      \node [style=none] (3) at (-5, 0.25) {};
      \node [style=none] (4) at (-5, -0.25) {};
      \node [style=none] (5) at (-3, 0) {};
      \node [style=none] (6) at (-3, -2) {};
      \node [style=none] (7) at (-3, -0.75) {};
      \node [style=none] (8) at (-3, -1.25) {};
      \node [style=none] (9) at (-5, -1) {};
      \node [style=none] (10) at (-6, 0.25) {};
      \node [style=none] (11) at (-6, -0.25) {};
      \node [style=none] (12) at (-6, -1) {};
      \node [style=none] (13) at (-2, 0) {};
      \node [style=none] (14) at (-2, -0.75) {};
      \node [style=none] (15) at (-2, -1.25) {};
      \node [style=none] (16) at (-2, -2) {};
      \draw (3.center) to (0);
      \draw (4.center) to (0);
      \draw (0) to (5.center);
      \draw (1) to (7.center);
      \draw (1) to (8.center);
      \draw (9.center) to (1);
      \draw (2) to (6.center);
      \draw (4.center) to (12.center);
      \draw (11.center) to (9.center);
      \draw (10.center) to (3.center);
      \draw (8.center) to (14.center);
      \draw (7.center) to (15.center);
      \draw (6.center) to (16.center);
      \draw (5.center) to (13.center);
    \end{tikzpicture}
  \end{equation}
  The coequalizer tells us that two representations are equivalent if they differ by permutations on the flags of individual corollas.
  For example, pictures \eqref{eq:corolla-crossed} above and \eqref{eq:corolla-uncrossed} below are two representations of the same element:
  
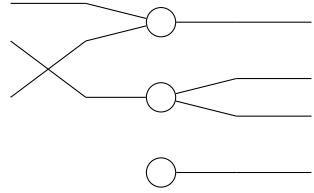
\begin{equation}\label{eq:corolla-uncrossed}
    \begin{tikzpicture}[baseline={(current bounding box.center)}]
      \node [style=White] (0) at (-4, 0) {};
      \node [style=White] (1) at (-4, -1) {};
      \node [style=White] (2) at (-4, -2) {};
      \node [style=none] (3) at (-5, 0.25) {};
      \node [style=none] (4) at (-5, -0.25) {};
      \node [style=none] (5) at (-3, 0) {};
      \node [style=none] (6) at (-3, -2) {};
      \node [style=none] (7) at (-3, -0.75) {};
      \node [style=none] (8) at (-3, -1.25) {};
      \node [style=none] (9) at (-5, -1) {};
      \node [style=none] (10) at (-6, 0.25) {};
      \node [style=none] (11) at (-6, -0.25) {};
      \node [style=none] (12) at (-6, -1) {};
      \node [style=none] (13) at (-2, 0) {};
      \node [style=none] (14) at (-2, -0.75) {};
      \node [style=none] (15) at (-2, -1.25) {};
      \node [style=none] (16) at (-2, -2) {};
      \draw (3.center) to (0);
      \draw (4.center) to (0);
      \draw (0) to (5.center);
      \draw (1) to (7.center);
      \draw (1) to (8.center);
      \draw (9.center) to (1);
      \draw (2) to (6.center);
      \draw (4.center) to (12.center);
      \draw (11.center) to (9.center);
      \draw (10.center) to (3.center);
      \draw (8.center) to (15.center);
      \draw (7.center) to (14.center);
      \draw (6.center) to (16.center);
      \draw (5.center) to (13.center);
    \end{tikzpicture}
  \end{equation}
  On the other hand, the co--wedges for the commutativity constraints tell us that we can permute the vertices as long as we do the appropriate block permutation which preserves the connectivity of the wires.
  \begin{equation}
    \begin{tikzpicture}[baseline={(current bounding box.center)}]
      \node [style=White] (0) at (-4, 0) {};
      \node [style=White] (1) at (-4, -2) {};
      \node [style=White] (2) at (-4, -1) {};
      \node [style=none] (3) at (-5, 0.25) {};
      \node [style=none] (4) at (-5, -0.25) {};
      \node [style=none] (5) at (-3, 0) {};
      \node [style=none] (6) at (-3, -1) {};
      \node [style=none] (7) at (-3, -1.75) {};
      \node [style=none] (8) at (-3, -2.25) {};
      \node [style=none] (9) at (-5, -1) {};
      \node [style=none] (10) at (-6, 0.25) {};
      \node [style=none] (11) at (-6, -0.25) {};
      \node [style=none] (12) at (-6, -1) {};
      \node [style=none] (13) at (-2, 0) {};
      \node [style=none] (14) at (-2, -1) {};
      \node [style=none] (15) at (-2, -1.5) {};
      \node [style=none] (16) at (-2, -2.25) {};
      \draw (3.center) to (0);
      \draw (4.center) to (0);
      \draw (0) to (5.center);
      \draw (1) to (7.center);
      \draw (1) to (8.center);
      \draw (9.center) to (1);
      \draw (2) to (6.center);
      \draw (4.center) to (12.center);
      \draw (11.center) to (9.center);
      \draw (10.center) to (3.center);
      \draw (8.center) to (15.center);
      \draw (7.center) to (14.center);
      \draw (6.center) to (16.center);
      \draw (5.center) to (13.center);
    \end{tikzpicture}
  \end{equation}
\end{ex}

\begin{ex}[Boxed corollas]
  \label{boxed-example}
  By iterating this construction, we will see that $(\tau_{\SS}^\otimes)^{\otimes}$ is the bimodule of boxed corollas.
  Again, we start with the coend formulation:
  \begin{equation}
    (\tau_\SS^{\otimes})^\otimes(n,m) =
    \int^{\bigboxtimes_i (n_i, m_i) }
    \Hom \left( \sum_i(n_i, m_i), (n,m) \right)
    \odot (\tau_{\SS}^\otimes(n_1, m_1) \times \ldots \times \tau_{\SS}^\otimes(n_k, m_k))
  \end{equation}
  Using the coequalizer formula~\ref{coeq-formula}, we can unpack this:
  \begin{equation}
    \begin{tikzcd}
      {
        \coprod_{(\sigma_i \in \Aut(n_i, m_i))}
        \Hom(\sum_i (n_i, m_i), (n,m))
        \odot
        (\tau_{\SS}^\otimes(n_1, m_1) \times \ldots \times \tau_{\SS}^\otimes(n_k, m_k))
      }
      \arrow[d, "(\bigboxtimes_i \sigma_i)^*", shift left]
      \arrow[d, "(\bigboxtimes_i \sigma_i)_*"', shift right] \\
      {
        \coprod_{\bigboxtimes_i (n_i,m_i)}
        \Hom(\sum_i (n_i, m_i), (n,m))
        \odot
        (\tau_{\SS}^\otimes(n_1, m_1) \times \ldots \times \tau_{\SS}^\otimes(n_k, m_k))
      }
      \arrow[d, dashed] \\
      \tau_{\SS}^{\otimes}(n,m)
    \end{tikzcd}
  \end{equation}
  We are now in a situation similar to Example~\ref{symmetric-ext-example}.
  Here, each element $B \in (\tau_{\SS}^\otimes)^{\otimes}(n,m)$ corresponds to a sequence of elements $x_i \in \tau_{\SS}^\otimes$ and a pair of permutations $\lambda \in \SS_n$ and $\rho \in \SS_m$.
  Graphically, we can think of $B$ as a collection of boxed corollas:
  
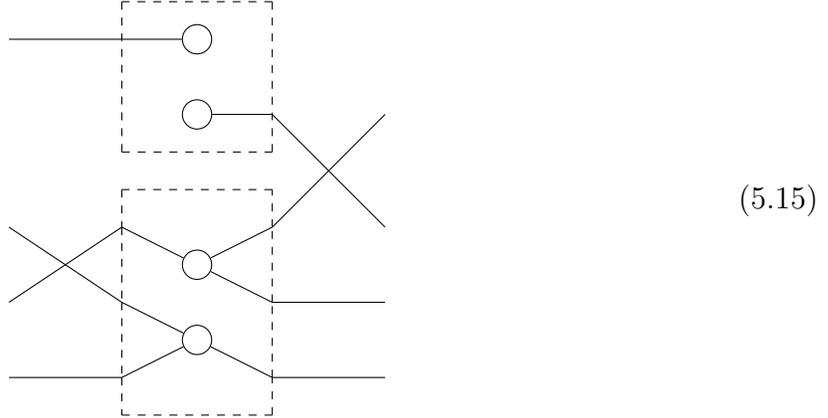
\begin{equation}
\begin{tikzpicture}[baseline={(current bounding box.center)}]
		\node [style=White] (0) at (-4, -2) {};
		\node [style=White] (1) at (-4, -1) {};
		\node [style=none] (2) at (-3, -2.5) {};
		\node [style=none] (3) at (-3, -1.5) {};
		\node [style=none] (4) at (-3, -0.5) {};
		\node [style=none] (5) at (-5, -0.5) {};
		\node [style=none] (6) at (-5, -1.5) {};
		\node [style=none] (7) at (-5, -2.5) {};
		\node [style=none] (8) at (-5, 0) {};
		\node [style=none] (9) at (-3, 0) {};
		\node [style=none] (10) at (-3, -3) {};
		\node [style=none] (11) at (-5, -3) {};
		\node [style=White] (12) at (-4, 1) {};
		\node [style=White] (13) at (-4, 2) {};
		\node [style=none] (14) at (-3, 1) {};
		\node [style=none] (17) at (-5, 2) {};
		\node [style=none] (20) at (-5, 2.5) {};
		\node [style=none] (21) at (-3, 2.5) {};
		\node [style=none] (22) at (-3, 0.5) {};
		\node [style=none] (23) at (-5, 0.5) {};
		\node [style=none] (24) at (-1.5, -0.5) {};
		\node [style=none] (25) at (-1.5, 1) {};
		\node [style=none] (26) at (-1.5, -1.5) {};
		\node [style=none] (27) at (-1.5, -2.5) {};
		\node [style=none] (28) at (-6.5, -0.5) {};
		\node [style=none] (29) at (-6.5, 2) {};
		\node [style=none] (30) at (-6.5, -1.5) {};
		\node [style=none] (31) at (-6.5, -2.5) {};
		\draw (7.center) to (0);
		\draw (6.center) to (0);
		\draw (5.center) to (1);
		\draw (0) to (2.center);
		\draw (1) to (3.center);
		\draw (1) to (4.center);
		\draw [style=Dash] (11.center) to (10.center);
		\draw [style=Dash] (10.center) to (2.center);
		\draw [style=Dash] (2.center) to (3.center);
		\draw [style=Dash] (3.center) to (4.center);
		\draw [style=Dash] (4.center) to (9.center);
		\draw [style=Dash] (9.center) to (8.center);
		\draw [style=Dash] (8.center) to (5.center);
		\draw [style=Dash] (5.center) to (6.center);
		\draw [style=Dash] (6.center) to (7.center);
		\draw [style=Dash] (7.center) to (11.center);
		\draw (17.center) to (13);
		\draw (12) to (14.center);
		\draw [style=Dash] (23.center) to (22.center);
		\draw [style=Dash] (22.center) to (14.center);
		\draw [style=Dash] (21.center) to (20.center);
		\draw [style=Dash] (20.center) to (17.center);
		\draw [style=Dash] (21.center) to (14.center);
		\draw [style=Dash] (23.center) to (17.center);
		\draw (14.center) to (24.center);
		\draw (4.center) to (25.center);
		\draw (3.center) to (26.center);
		\draw (2.center) to (27.center);
		\draw (6.center) to (28.center);
		\draw (5.center) to (30.center);
		\draw (29.center) to (17.center);
		\draw (31.center) to (7.center);
    \end{tikzpicture}
  \end{equation}
\end{ex}

\subsection{Heredity and unique factorization monoids}\label{sec:heredity}

Given a bimodule $\rho$, we are interested in looking for smaller bimodules $\nu$ such that $\rho \cong \nu^{\otimes}$.
Motivated by that point of view, we make the following definition:

\begin{definition}
  If $\rho$ and $\nu$ are bimodules such that $\rho \cong \nu^{\otimes}$, then we say $\rho$ has the \defn{unique factorization property} with respect to $\nu$.
  We will call $\nu$ the \defn{basic bimodule} of $\rho$.
\end{definition}

\begin{ex}
  The bimodule of cospans $\rho_{\rCospan}$ has the unique factorization property with respect to the bimodule $\nu_{\rCospan}$ of corollas.
\end{ex}

\begin{definition}\label{def:hereditary}
  A bimodule morphism $\Phi: \xi \to \zeta$ between bimodules $\xi \cong \nu_{\xi}^{\otimes}$ and $\zeta \cong \nu_{\zeta}^{\otimes}$ with the unique factorization property is \defn{hereditary} if the following is a pullback which extends when we take horizontal extensions:
  \begin{equation}
  \begin{tikzcd}
      \nu_{\xi} \arrow[d, tail] \arrow[r, "\phi", dashed] & \nu_{\zeta} \arrow[d, tail] \\
      \xi \arrow[r, "\Phi"']& \zeta
    \end{tikzcd}
  \end{equation}
\end{definition}

\subsubsection{Multiplication}
In this section, we will work out the consequences of hereditary morphisms for monoids.
Later in Section~\ref{sec:ufcs-bimodules}, we will compare this to the unique factorization categories of \cite{KMoManin}.

Consider a multiplication $\gamma: \rho \pl \rho \To \rho$.
If $\rho$ has the unique factorization property with respect to $\nu$, then we can define the \defn{bimodule of basic pairs} $\beta$ as the pullback of $\nu$ along $\gamma$.
Intuitively, the pull-back condition on $\beta$ says that if a composition yields a basic element, then it must come from $\beta$.
We know that the sub-bimodule $\beta^{\otimes 2}$ of $\rho \pl \rho$ maps into the sub-bimodule $\nu^{\otimes 2}$.
However, generally speaking, it is possible for there to be a larger bimodule that still maps into $\nu^{\otimes 2}$.
If $\gamma$ is hereditary, then this cannot happen by definition.

\subsubsection{Unit}

To uniquely factorize the unit $\eta: \tilde\rho_{\act} \To \rho$, we will need a basic bimodule for $\tilde\rho_{\act}$.
Such a bimodule exists if the category $\act$ is \defn{factorizable} in the sense that there is an inclusion functor $\imath: \V \to \act$ such that $\act \cong \V^{\boxtimes}$.

\begin{definition}
  \label{def:nu-G}
  Suppose the action category factorizes as $\act \cong \V^{\boxtimes}$.
  Then define a set-valued bimodule $\tilde \nu_{\act}: \act^{\rop} \times \act \to \Set$ as the following left Kan extension:
  \begin{equation}
    \begin{tikzcd}
      \V^{\op} \times \V
      \arrow[dr, "\imath^{\rop} \times \imath"']
      \arrow[rr, "\tilde\rho_{\V}"]
      && \Set \\
      &\act^{\op} \times \act
      \arrow[ru, "\tilde\nu_{\act} := \Lan_{\imath^{\rop} \times \imath}(\rho_{\V})"', dashed]
      &
    \end{tikzcd}
  \end{equation}
\end{definition}

\begin{lemma}\label{unique-fac-for-rho-G}
  If the category $\act$ factorizes as $\act \cong \V^{\boxtimes}$, then $\tilde\rho_{\act}^{\C}$ has the unique factorization property with respect to $\tilde\nu_{\act}^{\C}$.
\end{lemma}
\begin{proof}
  It is easy to see that $\tilde\rho_{\act} \cong \tilde\nu_{\act}^{\otimes}$ as set-valued bimodules.
  Since left adjoints like the free functor $F$ preserve left Kan extensions, we get the desired isomorphism:
  \begin{equation*}
    \tilde\rho_{\act}^{\C}
    = F \circ \tilde\rho_{\act}
    \cong F \circ \tilde\nu_{\act}^{\otimes}
    \cong (F \circ \tilde\nu_{\act})^{\otimes}
    = (\tilde\nu_{\act}^{\C})^\otimes \qedhere
  \end{equation*}
\end{proof}

\subsubsection{Hereditary subcategory}

We will show that the hereditary morphisms of Definition~\ref{def:hereditary} form a monoidal subcategory under certain conditions.

\begin{lemma}
  \label{lem:brick-wall}
    Assuming $\act \cong \V^\boxtimes$, the bimodule $H_\mu: \Gpd^{\boxtimes op} \times \Gpd^\boxtimes \to \Set$ defined by $H_\mu(\bigboxtimes_i X_i, \bigboxtimes_j Y_j) = \Hom_{\Gpd}(\bigotimes_i X_i, \bigotimes_j Y_j)$ is factorizable.
\end{lemma}
\begin{proof}
  The elements of $H_\mu(\bigboxtimes_{i \in I} B_i, \bigboxtimes_{j \in J} B'_j )$ can be depicted as colored span diagrams like the one in Figure~\ref{fig:H_mu-picture}.
  Let $h$ be the sub-bimodule of connected colored spans.
  Then $H_\mu$ factorizes as $H_\mu \cong h^\otimes$ for essentially the same reason the bimodule of spans factorizes.
\end{proof}

\begin{figure}
  \begin{subfigure}{0.45\linewidth}\label{fig:factoring-span}
    \centering
    \begin{tikzcd}[sep = small]
      & B \arrow[ld, no head] & B \arrow[ld, no head] \arrow[d, no head] & B \arrow[d, no head]  & B \arrow[d, no head] & \\
      b \arrow[d]  & b \arrow[rd] & b \arrow[ld] & b \arrow[d]  & b \arrow[d] \\
      b' \arrow[rd, no head] & b' \arrow[d, no head]  & b' \arrow[d, no head] & b' \arrow[d, no head] & b' \arrow[dl, no head] \\
      & B' & B' & B' &
    \end{tikzcd}
    \caption{}
  \end{subfigure}
  \begin{subfigure}{0.45\linewidth}
    \label{fig:brick-wall}
    \centering
    \begin{tikzpicture}
      \draw[thick] (0,1.5) rectangle (1,0);
      \draw[thick] (1,1.5) rectangle (3,0);
      \draw[thick] (3,1.5) rectangle (4,0);
      \draw[thick] (4,1.5) rectangle (5,0);
      \draw[thick] (0,0) rectangle (2,-1.5);
      \draw[thick] (2,0) rectangle (3,-1.5);
      \draw[thick] (3,0) rectangle (5,-1.5);
    \end{tikzpicture}
    \caption{}
  \end{subfigure}
  \caption{
    In (A), the top and bottom rows are objects of $\act^\boxtimes$ with each individual $B$ and $B'$ an object of $\act$.
    The two middle rows are objects of $\act$ with each individual $b$ and $b'$ an object of $\V$.
    The collection of line segments connecting a $B$-object to several $b$-objects represents an isomorphism from $B$ to the monoidal product of those $b$-objects.
    The arrows in the middle are morphisms in $\V$.
    These can also be pictured as ``brick wall'' diagrams (B) like those found in \cite{KMoManin}.
  }
  \label{fig:H_mu-picture}
\end{figure}

\begin{nota}
  \label{nota-hereditary-tensor}
  To save space, we will write $H^A_B = \Hom_{\act}(A,B)$ and denote the Cartesian product by juxtaposition, so $H^A_BH^B_C = H^A_B \times H^B_C$.
\end{nota}

\begin{lemma}
  \label{hereditary-tensor}
  Suppose the action category factorizes as $\act \cong \V^{\boxtimes}$.
  If $\rho_1 \cong \nu_1^{\ot}$ and $\rho_2 \cong \nu_2^{\ot}$ are two bimodules with the unique factorization property, then $\rho_1 \pl \rho_2$ also has the unique factorization property with a basic bimodule $\beta$ defined by
  \begin{equation}
    \beta(\bar A, \bar C) \cong
    \int\limits^{A_i B_i}
    \int\limits^{B'_j C_j}
    \cp{
      h^{(B_i)}_{(B'_j)}
      H^{\bar A}_{(A_i)}
      H^{(C_j)}_{\bar C}
    }{
      \bigotimes_i \nu_1(A_i,B_i)
      \otimes
      \bigotimes_j \nu_2(B'_j,C_j)
    }
  \end{equation}
\end{lemma}
\begin{proof}
  We will define an isomorphism $\beta^\otimes \simTo \nu_1^\otimes \pl \nu_2^\otimes$.
  Start by using the coend formulation of horizontal extension to write out $\beta^\otimes(A,C)$ explicitly.
  \begin{equation}
    \int\limits^{\bar A_k \bar C_k}
    H^A_{(\bar A_k)} H^{(\bar C_k)}_C
    \odot
    \bigotimes_k
    \int\limits^{A_{i} B_{i}}
    \int\limits^{B'_{j} C_{j}}
    \cp{
      h^{(B_i)}_{(B'_j)}
      H^{\bar A_k}_{(A_{i})}
      H^{(C_{j})}_{\bar C_k}
    }{
      \bigotimes_{i} \nu_1(A_{i},B_{i})
      \otimes
      \bigotimes_{j} \nu_2(B'_{j},C_{j})
    }
  \end{equation}
  Let $K = \{1, \ldots, n\}$ denote the indexing set for the outermost coend.
  Let $I_k$ denote the index set of the $\int^{A_iB_i}$ coend in the $k^{th}$ part of the monoidal product.
  Likewise, let $J_k$ denote the index set of the $\int^{B'_jC_j}$ coend in the $k^{th}$ part of the monoidal product.
  Apply Fubini and commutativity of colimits and monoidal products to get the following:
  \begin{multline}
    \int^{\bar A_k \bar C_k}
    \int^{(A_{i} B_{i})_{i \in I_1}}
    \cdots
    \int^{(A_{i} B_{i})_{i \in I_n}}
    \int^{(B'_{j} C_{j})_{j \in J_1}}
    \cdots
    \int^{(B'_{j} C_{j})_{j \in J_n}} H^A_{(\bar A_k)} H^{(\bar C_k)}_C\\
    \odot
    \bigotimes_k
    \cp{ h^{(B_i)}_{(B'_j)} H^{\bar A_k}_{(A_{i})} H^{(C_{j})}_{\bar C_k} }{
      \bigotimes_{i} \nu_1(A_{i},B_{i})
      \otimes
      \bigotimes_{j} \nu_2(B'_{j},C_{j}) }
  \end{multline}
  Then apply interchange and collect the $A$-variables and $C$-variables.
  \begin{align*}
    &\int\limits^{\bar A_k \bar C_k}
    \int\limits
    \cdots
    \int\limits
    H^A_{(\bar A_k)} H^{(\bar C_k)}_C
    \odot
    \cp{
      \prod_k
      \left(
      h^{(B_i)}_{(B'_j)}
      H^{\bar A_k}_{(A_{i})}
      H^{(C_{j})}_{\bar C_k}
      \right)
    }{
      \bigotimes_k
      \left(
      \bigotimes_{i} \nu_1(A_{i},B_{i})
      \otimes
      \bigotimes_{j} \nu_2(B'_{j},C_{j})
      \right)
    } \\
    &\int\limits^{\bar A_k \bar C_k}
    \int\limits
    \cdots
    \int\limits
    \left(
    H^A_{(\bar A_k)}
    \prod_k
    H^{\bar A_k}_{(A_{i})}
    \right)
    \left(
    H^{(\bar C_k)}_C
    \prod_k
    H^{(C_{j})}_{\bar C_k}
    \right)
    \left(
    \prod_k
    h^{(B_i)}_{(B'_j)}
    \right)
    \odot
    \bigotimes_k
    \left(
      \bigotimes_{i} \nu_1(A_{i},B_{i})
      \otimes
      \bigotimes_{j} \nu_2(B'_{j},C_{j})
    \right)
  \end{align*}

  Let $H^A_{(\bar A_k)} \prod_k H^{\bar A_k}_{(A_i)_{i \in I_k}} \to H^A_{(A_i)}$ and $H_C^{(C_k)} \prod_k H_{C_k}^{(C_j)_{j \in J_k}} \to H_C^{(C_j)}$ be the maps defined by monoidal products and composition in the category $\act$.
  For each $k \in K$, write $D_k = (\bigboxtimes_{i \in I_k} B_i, \bigboxtimes_{j \in J_k} B'_j)$ and define $D \in (\Gpd^{op \boxtimes} \times \Gpd^{\boxtimes})^\boxtimes$ by $D = \bigboxtimes_{k \in K} D_k$.
  Let $pt \to H^{\mu D}_{([B_i],[B'_j])}$ be the pointing for the identity morphism.
  For reasons of notation, let $\prod_k h^{(B_i)}_{(B'_j)} \to \mu h^\boxtimes(D)$ be the identity.
  Collecting these together gives a map of copowers:
  \begin{equation}
    \label{eq:copower-part}
    \left(
    H^A_{(\bar A_k)}
    \prod_k
    H^{\bar A_k}_{(A_{i})}
    \right)
    \left(
    H^{(\bar C_k)}_C
    \prod_k
    H^{(C_{j})}_{\bar C_k}
    \right)
    \left(
    \prod_k
    h^{(B_i)}_{(B'_j)}
    \right)
    \to
    H^A_{(A_i)} H_C^{(C_j)} H^{\mu D}_{([B_i],[B'_j])} \mu h^\boxtimes(D)
  \end{equation}
  Similarly, there is a isomorphism for the right part given by applying commutativity constraints and associativity:
  \begin{align}
    \label{eq:object-part}
    \bigotimes_{k \in K} \left( \bigotimes_{i \in I_k} \nu_1(A_i,B_i) \otimes \bigotimes_{j \in J_k} \nu_2(B'_j,C_j) \right)
    &\to
      \left(
      \bigotimes_{k \in K}  \bigotimes_{i \in I_k} \nu_1(A_i,B_i)
      \right)
      \otimes
      \left(
      \bigotimes_{k \in K} \bigotimes_{j \in J_k} \nu_2(B'_j,C_j)
      \right) \\
    &\cong \bigotimes_{i \in I} \nu_1(A_i,B_i)
      \otimes \bigotimes_{j \in J} \nu_2(B'_j,C_j)
  \end{align}

  We will show that the copower of \eqref{eq:copower-part} and \eqref{eq:object-part} induce a morphism from the previous coend into the following target coend:
  \begin{equation}
    \label{eq:target-coend}
    \int\limits^{(A_i,B_i)}
    \int\limits^{(B'_j,C_j)}
    \int\limits^D
    H^A_{(A_i)}
    H^{\mu D}_{([B_i],[B'_j])}
    \mu h^\boxtimes(D)
    H^{(C_j)}_C
    \odot
    \left(
    \bigotimes_i \nu_1(A_i,B_i)
    \otimes
    \bigotimes_j \nu_2(B'_j,C_j)
    \right)
  \end{equation}
  The necessary compatibility for the $\bar A_k$ and $\bar C_k$ variables in the source coend is straight forward since the map \eqref{eq:copower-part} is defined by composition.
  The necessary compatibility for the $A_i$, $B_i$, $B'_j$, and $C_j$ variables in the source coend follows from the defining properties of the target coend~\eqref{eq:target-coend}.
  
  We can always choose a representation of the target coend \eqref{eq:target-coend} where the $H^{\mu D}_{([B_i],[B'_j])}$ component is the identity.
  This gives a well-defined morphism going the opposite way since if there are isomorphisms for the $A_i$, $B_i$, $B'_j$, and $C_j$ variables that realize that two representations of \eqref{eq:target-coend} are equivalent, then the same isomorphism will realize the two representations on the other side are also equivalent.
  This new morphism is inverse to the previous one, so this is an isomorphism.

  Now use Fubini to consolidate the part involving the $D$ variable and apply the definition of horizontal extension.
  \begin{align*}
    \int\limits^{(A_i,B_i)}
      \int\limits^{(B'_j,C_j)}
      H^A_{(A_i)}
      \left(
      \int\limits^{D: (\Gpd^{op\boxtimes} \times \Gpd^{\boxtimes})^\boxtimes}
      \!\!\! H^{\mu D}_{([B_i],[B'_j])}
      \odot
      \mu h^\boxtimes(D)
      \right)
      H^{(C_j)}_C
      \odot
      \left(
      \bigotimes_i \nu_1(A_i,B_i)
      \otimes
      \bigotimes_j \nu_2(B'_j,C_j)
    \right) \\
    = \int\limits^{(B'_j,C_j)}
      H^A_{(A_i)} h^{\otimes}\left(\dbigboxtimes_i B_i, \dbigboxtimes_j B'_j \right) H^{(C_j)}_C
      \odot
      \left(
      \bigotimes_i \nu_1(A_i,B_i)
      \otimes
      \bigotimes_j \nu_2(B'_j,C_j)
      \right)
  \end{align*}
  Since $\Gpd \cong \V^\otimes$, Lemma~\ref{lem:brick-wall} implies that $H_\mu \cong h^\otimes$.
  Continuing our conventions, we write $H^{(B_i)}_{(B'_j)} = H_\mu(\dbigboxtimes_i B_i, \dbigboxtimes_j B'_j)$.

  Apply Yoneda to introduce the $B$-variable.
  Then use Fubini, the assumption that the monoidal product commutes with colimits, and interchange to obtain the desired plethysm:
  \begin{align*}
    &\int\limits^{(A_i,B_i)}
      \int\limits^{(B'_j,C_j)}
      \int\limits^B
      H^A_{(A_i)} H^{(B_i)}_B H^B_{(B'_j)} H^{(C_j)}_C
      \odot
      \left(
      \bigotimes_i \nu_1(A_i,B_i)
      \otimes
      \bigotimes_j \nu_2(B'_j,C_j)
      \right) \\
    &\cong \int\limits^B
      \left(
      \int\limits^{(A_i,B_i)}
      H^A_{(A_i)} H^{(B_i)}_B
      \odot
      \bigotimes_i \nu_1(A_i,B_i)
      \right)
      \otimes
      \left(
      \int\limits^{(B'_j,C_j)}
      H^B_{(B'_j)} H^{(C_j)}_C
      \odot
      \bigotimes_j \nu_2(B'_j,C_j)
      \right) \\  
    &= (\nu_1^\otimes \pl \nu_2^\otimes)(A,C)
      \qedhere
  \end{align*}
\end{proof}

\begin{prop}
\label{prop:monoialsubcat}
  For bimodules with a factorizable action category and a Cartesian closed target category, the uniquely factorizable bimodules and hereditary morphisms form a monoidal subcategory under plethysm.
\end{prop}
\begin{proof}
  By the gluing property of pullbacks, hereditary morphisms are closed under composition.
  By Lemma~\ref{hereditary-tensor}, the plethysm product of two factorizable bimodules is factorizable, so this category is closed under taking monoidal products of two objects.
  Consider the plethysm product $\Psi = \Phi_1 \pl \Phi_2$ of a pair of hereditary morphisms $\Phi_1: \xi_1 \To \zeta_1$ and $\Phi_2: \xi_2 \To \zeta_2$ between factorizable bimodules.
  Using the explicit factorization given in Lemma~\ref{hereditary-tensor}, we obtain a morphism $\psi: \beta_1 \To \beta_2$.
  Since the target category is Cartesian closed, it satisfies pullback stability.
  A routine argument shows that $\psi$ gives a pullback which extends when we take horizontal extensions.
  Hence $\Phi_1 \pl \Phi_2$ is hereditary.
\end{proof}

\section{Basic element representations and relative bimodules}
\subsection{Basic element representations}\label{sec:basic-el-rep}

In this section, we will show that if a bimodule $\rho$ has the unique factorization property, then there is a complementary notion of basic element representations with their own notion of a plethysm product.

\subsubsection{Basic actions and extensions}

Throughout this section, we will fix a bimodule $\rho: \act^{\rop} \times \act \to \Set$ with the unique factorization property with respect to $\nu: \act^{\rop} \times \act \to \Set$.
This naturally leads us to consider the following type of representation:

\begin{definition}
  \defn{Basic element representations} are functors of the form $\el(\nu) \to \C$.
\end{definition}

\begin{const}
  Recall that the objects of $\el(\nu)^\boxtimes$ are words of elements of $\nu$.
  Intuitively, one would expect there to be a way of taking a word of elements and ``joining'' them to obtain an element of $\nu^\otimes$.
  To achieves this, we will define a functor $\mu: \el(\nu)^\boxtimes \to \el(\nu^\otimes)$ as a composition of 3 intermediate functors.
  \begin{enumerate}
    \item First, recall elements of $\nu$ are formally morphisms $1_{\C} \to \nu(A,B)$.
    Hence objects of $\el(\nu)^\boxtimes$ are represented by sequences $(x_i: 1_{\C} \to \nu(A_i,B_i))_{i=1}^n$ of pointings.
    Observe that the monoidal category structure of $\C$ readily induces a new pointing $x$ from such a sequence:
    \begin{equation}
      \begin{tikzcd}
        1_{\C}
        \arrow[rr, "x"]
        \arrow[rd, "\sim"']
        & & {\bigotimes_{i=1}^n \nu(A_i, B_i)} \\
        & \bigotimes_{i=1}^n 1_{\C}
        \arrow[ru, "\bigotimes_i x_i"'] &
      \end{tikzcd}
    \end{equation}
    In other words, the monoidal structure of $\C$ provides a functor:
    \begin{equation}
      \el(\nu)^\boxtimes \to \el(\mu_{\C} \circ \nu^{\boxtimes})
    \end{equation}
    \item There is another functor induced by the natural transformation $\eta$ occurring in Definition~\eqref{eq:hor-ext-def} of the horizontal extension:
    \begin{equation}
      \el(\eta): \el(\mu_{\C} \circ \nu^{\boxtimes}) \to \el(\nu^{\otimes} \circ \mu_{\act^{\rop} \times \act})
    \end{equation}
    \item Finally, there is a functor $\el(\nu^{\otimes} \circ \mu_{\act^{\rop} \times \act}) \to \el(\nu^{\otimes})$ which sends the pointings of the form~\eqref{expanded} to pointings of the form~\eqref{collapsed}.
    \begin{align}
      1_{\C} & \to \rho^{\otimes}(\mu_{\act^{\rop} \times \act}(\bigboxtimes_i(A_i,B_i))) \label{expanded} \\
      1_{\C} & \to \rho^{\otimes}({\textstyle \bigotimes_i} A_i, {\textstyle \bigotimes_i} B_i) \label{collapsed}
    \end{align}
  \end{enumerate}
  The desired functor is given by composing these three functors.
\end{const}

\begin{ex}
  Suppose $\nu = \nu_{\rCospan}$ is the bimodule of basic cospans.
  Then $\mu(c_1 \boxtimes \ldots \boxtimes c_n)$ is the disjoint union of those basic cospans.
\end{ex}

As the notation suggests, $\mu$ behaves somewhat like a restricted monoidal product.
This allows us to extend a basic element representation to a whole element representation using a construction similar to horizontal extension.

\begin{definition}
  \label{horizontal-extension-elements}
  Consider a basic element representation $D: \el(\nu) \to \C$.
  Define the \defn{(element-wise) horizontal extension} $D^\otimes$ as the following left Kan extension:
  \begin{equation}
    \begin{tikzcd}
      \el(\nu)^\boxtimes
      \arrow[d, "\mu"']
      \arrow[r, "D^\boxtimes"]
      & \C^{\boxtimes}
      \arrow[ld, Rightarrow, shorten >=0.3cm,shorten <=0.4cm, "\eta"']
      \arrow[d, "\mu_{\C}"] \\
      \el(\nu^\otimes)
      \arrow[r, dashed, "D^\otimes"']
      & \C
    \end{tikzcd}
  \end{equation}
\end{definition}

We now have two notions of extension: one on bimodules and one on element representations.
Since bimodules and element representations are connected by the chi-construction, one would hope that the operations $(-)^\otimes$ and $\chi_{(-)}$ commute.
The following lemma shows that is indeed the case.

\begin{lemma}
  \label{chi-otimes-commute}
  Suppose $\C$ is a monoidal category such that $\otimes$ commutes with colimits.
  Then for any basic element representation $D: \el(\nu) \to \C$, we have a natural isomorphism
  \begin{equation}
    \chi_{D}^{\otimes} \simTo \chi_{D^\otimes}
  \end{equation}
\end{lemma}
\begin{proof}
  Recall that both constructions are iterated left Kan extensions:
  \begin{equation}
    \begin{tikzcd}
      {\el(\nu)^{\boxtimes}} & \C^\boxtimes & {\el(\nu)^{\boxtimes}} \\
      {\el(\nu^\otimes)} & \C & {(\act^{\rop} \times \act)^\boxtimes} & {\C^{\boxtimes}} \\
      \act^{\rop} \times \act && {\act^{\rop} \times \act} & \C
      \arrow["\mu"', from=1-1, to=2-1]
      \arrow["\Sigma"', from=2-1, to=3-1]
      \arrow["{D^\boxtimes}", from=1-1, to=1-2]
      \arrow["{\mu_\C}", from=1-2, to=2-2]
      \arrow["{\chi_{D^\otimes}}"', dashed, from=3-1, to=2-2]
      \arrow["{D^\otimes}", dashed, from=2-1, to=2-2]
      \arrow["{\Sigma^\boxtimes}"', from=1-3, to=2-3]
      \arrow["{\mu_{\act^{\rop} \times \act}}"', from=2-3, to=3-3]
      \arrow["{(\chi_D)^\boxtimes}"', dashed, from=2-3, to=2-4]
      \arrow["{D^\boxtimes}", from=1-3, to=2-4]
      \arrow["{\mu_\C}", from=2-4, to=3-4]
      \arrow["{(\chi_D)^\otimes}"', dashed, from=3-3, to=3-4]
    \end{tikzcd}
  \end{equation}
  Hence, both $\chi_{D^\otimes}$ and $(\chi_D)^\otimes$ are left Kan extensions of the functor $\mu_\C \circ D^\boxtimes$.
  It suffices to verify that $\Sigma \circ \mu = \mu_{\act^{\rop} \times \act} \circ \Sigma^\boxtimes$, which is straight forward.
\end{proof}

\subsubsection{Basic element plethysm}

We will start by considering what we called the \emph{compatible tensor product} $D_1^\otimes \hat \otimes D_2^\otimes$ which is the source of the Kan extensions $D_1^\otimes \epl D_2^\otimes$.
The following lemma says that this compatible tensor product is determined by its basic part.

\begin{lemma}
  The horizontal extension of the following restriction is isomorphic to $D_1^\otimes \hat \otimes D_2^\otimes$.
  \begin{equation}
    \begin{tikzcd}
      {\el(\beta)} & {\el(\rho \pl \rho)} & \E
      \arrow["{\el(\iota)}", from=1-1, to=1-2]
      \arrow["{D_1^\otimes \hat \otimes D_2^\otimes}", from=1-2, to=1-3]
    \end{tikzcd}
  \end{equation}
\end{lemma}
\begin{proof}
  Recall that $(D_1^\otimes \hat \otimes D_2^\otimes)[x, y]$ was defined to be the coend for the following co--wedge:
  \begin{equation}
    \begin{tikzcd}
      D_1^\otimes(x) \otimes D_2^\otimes(y')
      \arrow[d, "\id \otimes D_2^\otimes(\sigma \biact \id)"']
      \arrow[rr, "D_1^\otimes(\id \biact \sigma) \otimes \id"]
      & & D_1^\otimes(x') \otimes D_2^\otimes(y')
      \arrow[d, dashed] \\
      D_1^\otimes(x) \otimes D_2^\otimes(y)
      \arrow[rr, dashed]
      & & {(D_1^\otimes \hat \otimes D_2^\otimes)[x,y]}
    \end{tikzcd}
  \end{equation}
  Assuming the monoid has the unique factorization property, any composable pair $[x,y] \in (\rho \pl \rho)(A, B)$ can be factored into a collection of elements $[x_i, y_i] \in \beta(A_i, B_i)$.
  Since $D_1^\otimes$ and $D_2^\otimes$ are defined as horizontal extensions of basic element representations, we may rearrange the components of the co--wedge:
  \begin{equation}
    \begin{tikzcd}
      D_1^\otimes(x_1) \otimes D_2^\otimes(y'_1) \otimes \ldots
      \arrow[d, "\id \otimes D_2^\otimes(\sigma_1 \biact \id)"']
      \arrow[rr, "D_1^\otimes(\id \biact \sigma_1) \otimes \id"]
      & & D_1^\otimes(x'_1) \otimes D_2^\otimes(y'_1) \otimes \ldots \\
      D_1^\otimes(x_1) \otimes D_2^\otimes(y_1) \otimes \ldots
      & &
    \end{tikzcd}
  \end{equation}
  This rearrangement process is a natural isomorphism, hence this will carry over when we take the coend giving us the following natural isomorphism:
  \begin{equation*}
    (D_1^\otimes \hat \otimes D_2^\otimes)[x,y]
    \cong \bigotimes_i (D_1^\otimes \hat \otimes D_2^\otimes)[x_i,y_i]
  \end{equation*}
  Therefore the horizontal extension of this restriction is indeed naturally isomorphic to the original $D_1^\otimes \hat \otimes D_2^\otimes$.
\end{proof}

\begin{definition}
  Define the \defn{basic element plethysm} of two basic element representations $D_1, D_2: \el(\nu) \to \C$ as the following left Kan extension:
  \begin{equation}
    \begin{tikzcd}
      {\el(\beta)} & {\el(\rho \pl \rho)} & \C \\
      & {\el(\nu)}
      \arrow[tail, from=1-1, to=1-2]
      \arrow["{D_1^\otimes \hat \otimes D_2^\otimes}", from=1-2, to=1-3]
      \arrow["{\el(\gamma_0)}"', from=1-1, to=2-2]
      \arrow["{D_1 \bepl D_2}"', dashed, from=2-2, to=1-3]
    \end{tikzcd}
  \end{equation}
\end{definition}

\begin{prop}
  \label{diamond-uf-property}
  The element representation $D_1^\otimes \epl D_2^\otimes: \el(\rho) \to \E$ has the unique factorization property with respect to $D_1 \bepl D_2: \el(\nu) \to \E$.
\end{prop}
\begin{proof}
  We start by considering horizontally extending the defining diagram for $D_1 \bepl D_2$:
  \begin{equation}
    \label{eq:extend-kan}
    \begin{tikzcd}[column sep = small]
      {\el(\beta)} && \E & \leadsto & {\el(\rho \pl \rho)} && \E \\
      & {\el(\nu)} && \leadsto && {\el(\rho)}
      \arrow[""{name=0, anchor=center, inner sep=0}, from=1-1, to=1-3]
      \arrow[from=1-1, to=2-2]
      \arrow[from=1-5, to=2-6]
      \arrow[""{name=1, anchor=center, inner sep=0}, from=1-5, to=1-7]
      \arrow[dashed, from=2-2, to=1-3]
      \arrow[dashed, from=2-6, to=1-7]
      \arrow[shorten <=3pt, shorten >=3pt, Rightarrow, from=0, to=2-2]
      \arrow[shorten <=3pt, shorten >=3pt, Rightarrow, from=1, to=2-6]
    \end{tikzcd}
  \end{equation}
  By the previous lemma, the functor $\el(\rho \pl \rho) \to \E$ in diagram~\eqref{eq:extend-kan} is isomorphic to $D_1^\otimes \hat \otimes D_2^\otimes$.
  Since the monoid $(\rho, \gamma)$ has the unique factorization property, the functor $\el(\rho \pl \rho) \to \el(\rho)$  in diagram~\eqref{eq:extend-kan} is isomorphic to $\el(\gamma)$.
  It remains to show that the extended element representation $\el(\rho) \to \E$ and the corresponding natural transformation define a left Kan extension.

  Consider an arbitrary element representation $A: \el(\rho) \to \E$ and a natural transformation $D_1^\otimes \hat \otimes D_2^\otimes \To A \circ \el(\gamma)$.
  Abbreviating $D = (D_1^\otimes \hat \otimes D_2^\otimes)$ and abusing notation for $A$, this is a family of morphisms $D[x,y] \to A[x,y]$.
  Factor $[x,y]$ into basic composable pairs $\{[x_i, y_i]\}_{i=1}^n$, then use this factorization to shuffle terms on the level of coends:
  \begin{equation}
    \bigotimes_i D_1^\otimes(x_i) \otimes D_2^\otimes(y_i) \to A[x,y]
  \end{equation}
  It is then clear that this factors uniquely through the extension defined above.
\end{proof}

\subsubsection{Monoidal categories and monoids}

In this section, we will show that basic element representations constitute a monoidal category with $\bepl$ as a monoidal product.

\begin{definition}
  Define $V_{\eta}: \el(\nu) \to \C$ as the left Kan extension given by extending the trivial functor $T: \el(\tilde\nu_{\act}) \to \C$ along the unit $\el(\eta_0): \el(\tilde\nu_\act) \to \el(\nu)$:
  \begin{equation}
    \begin{tikzcd}
      \el(\tilde\nu_{\act}) \arrow[rd, "\el(\eta_0)"'] \arrow[rr, "T"] & & \C \\
      & \el(\nu) \arrow[ru, "V_{\eta} := \Lan_{\el(\eta_0)}(T)"', dashed] &
    \end{tikzcd}
  \end{equation}
\end{definition}

\begin{prop}
  \label{elem-element-rep-mon-cat}
  Let $\rho: \act^{\rop} \times \act \to \C$ be a bimodule with the unique factorization property with respect to $\nu: \act^{\rop} \times \act \to \C$.
  Then the basic element representations form a monoidal category with $\bepl$ as a monoidal product and $V_\eta$ as a monoidal unit.
\end{prop}
\begin{proof}
  The associativity constraint for $\bepl$ comes from repeatedly applying Lemma~\ref{pl-epl-iso} and the associativity constraint of $\epl$.
  The unit constraint is defined essentially the same way it was defined in Proposition~\ref{epl-unit}.
  However some care is required when defining the morphisms.
  Instead of obscuring what is ultimately a simple proof, we will abuse notation somewhat and set up with the following diagram:
  \begin{equation}
    \begin{tikzcd}
      \el(\tilde\nu_{\act}^\otimes \pl \nu) \arrow[d, tail] \arrow[dd, "\sim"', bend right] \arrow[rrd, "T^{\otimes} \hat \otimes D"] & & \\
      \el(\beta) \arrow[d, "\gamma_0"] \arrow[rr, "V_{\eta}^\otimes \hat \otimes D" description] & & \C \\
      \el(\nu) \arrow[rru, "D"'] & &
    \end{tikzcd}
  \end{equation}
  Then proceed in the same way we did in Proposition~\ref{epl-unit}.
\end{proof}

\subsection{Basic relative bimodules}

\subsubsection{Properties}

Fix a bimodule monoid $(\rho, \gamma)$ such that $\rho$ has the unique factorization property with respect to $\nu$.

\begin{definition}
  A \defn{basic relative bimodule} $\xi \Rightarrow \nu$ is a bimodule over the basic bimodule $\nu$.
\end{definition}

\begin{definition}
  Define the \defn{basic-relative-product}\/ $\pl_{(\nu)}$ of two basic relative bimodules $\xi_1 \overset{\pi_1}{\To} \nu$ and $\xi_2 \overset{\pi_2}{\To} \nu$ as the following pullback:
  \begin{equation}
    \begin{tikzcd}
      & {\xi_1 \pl_{(\nu)} \xi_2} & \beta & \nu \\
      {} & {\xi_1^\otimes \pl \xi_2^\otimes} & {\rho \pl \rho} & \rho
      \arrow[tail, from=1-4, to=2-4]
      \arrow["{\pi_1^\otimes \pl \pi_2^\otimes}"', from=2-2, to=2-3]
      \arrow["\gamma"', from=2-3, to=2-4]
      \arrow[tail, from=1-3, to=2-3]
      \arrow["{\gamma_0}", from=1-3, to=1-4]
      \arrow[dashed, from=1-2, to=1-3]
      \arrow[dashed, tail, from=1-2, to=2-2]
      \arrow["\lrcorner"{anchor=center, pos=0.125}, draw=none, from=1-3, to=2-4]
      \arrow["\lrcorner"{anchor=center, pos=0.125}, draw=none, from=1-2, to=2-3]
    \end{tikzcd}
  \end{equation}
\end{definition}

\begin{prop}
\label{prop:basic-elt-mon}
  If there is a groupoid $\V$ such that $\act \cong \V^{\boxtimes}$ and the unit of the base bimodule has the unique factorization property, then basic relative bimodules over $\nu: \act^{\rop} \times \act \to \C$ form a monoidal category with $\pl_{(\nu)}$ as a monoidal product.
\end{prop}
\begin{proof}
  To check associativity, consider the following triple of relative bimodules:
  \begin{equation}
    \xi_1 \overset{\pi_1}{\To} \nu \quad \xi_2 \overset{\pi_2}{\To} \nu \quad \xi_3 \overset{\pi_3}{\To} \nu
  \end{equation}
  To keep the notation simple, let $\xi_{12} \overset{\pi_{12}}{\To} \nu$ denote the basic relative plethysm of the left pair and let $\xi_{23} \overset{\pi_{23}}{\To} \nu$ denote the basic relative plethysm of the right pair.
  Consider the following diagram:
  \begin{equation}
    \begin{tikzcd}
      \xi_{12} \pl_{(\nu)} \xi_3
      \arrow[d, dashed, tail]
      \arrow[r, dashed]
      & \beta
      \arrow[d, dashed, tail]
      \arrow[r, dashed]
      & \nu
      \arrow[d, tail] \\
      \xi_{12}^{\otimes} \pl \xi_{3}^{\otimes}
      \arrow[r, "\pi_{12} \pl \pi_3"']
      \arrow[d, "\sim"']
      & \rho \pl \rho
      \arrow[r, "\gamma"']
      \arrow[d, equal]
      & \rho \arrow[d, equal] \\
      \xi_1^\otimes \pl \xi_{23}^{\otimes}
      \arrow[r, "\pi_1 \pl \pi_{23}"']
      & \rho \pl \rho
      \arrow[r, "\gamma"']
      & \rho
    \end{tikzcd}
  \end{equation}
  The lower left square commutes by properties of $\pl$ and the lower right commutes tautologically.
  Since the outer square is still a pullback, the two relative bimodules associated with $\xi_{12} \pl_{(\nu)} \xi_3$ and $\xi_{1} \pl_{(\nu)} \xi_{23}$ are naturally isomorphic.

  We claim that $\eta_0: \tilde\nu_{\act} \To \nu$ serves as the unit basic relative bimodule.
  To show this, let $\xi \overset{\pi}{\To} \nu$ be an arbitrary basic relative bimodule and consider the following pullback:
  \begin{equation}
    \begin{tikzcd}
      \xi \pl_{(\nu)} \tilde\nu_{\act}
      \arrow[d, dashed, tail]
      \arrow[r, dashed]
      & \beta
      \arrow[d, dashed, tail]
      \arrow[r, dashed]
      & \nu
      \arrow[d, tail] \\
      \xi^{\otimes} \pl \tilde\nu_{\act}^{\otimes}
      \arrow[r, "\pi^\otimes \pl \eta_0^\otimes"']
      & \rho \pl \rho
      \arrow[r, "\gamma"']
      & \rho
    \end{tikzcd}
  \end{equation}
  Since $\eta_0^\otimes: \tilde\nu_{\act}^\otimes \To \nu^\otimes$ is isomorphic to the unit $\eta: \tilde\rho_\act \To \rho$ of (ordinary) relative bimodules, we are reduced to the following:
  \begin{equation}
    \begin{tikzcd}
      \xi \pl_{(\nu)} \tilde\nu_{\act}
      \arrow[d, dashed, tail]
      \arrow[r, dashed]
      & \nu
      \arrow[d, tail] \\
      \xi^{\otimes}
      \arrow[r,"\pi"]
      & \rho = \nu^{\otimes}
    \end{tikzcd}
  \end{equation}
  Setting the morphism to $\pi: \xi \To \nu$ would also make this a pullback.
  Since pullbacks are unique up to isomorphism, we have $\xi \pl_{(\nu)} \tilde\nu_{\act} \cong \xi$ and $\pi \pl_{(\nu)} \eta$ is equivalent to $\pi$.
\end{proof}

\subsubsection{Examples}

\begin{ex}\label{basic-rel-monoid-properad}
  Consider the bimodule $\rho_{\rCospan}$ of cospans which has the unique factorization property with respect to the bimodule $\nu_{\rCospan}$ of basic cospans.
  A relative bimodule structure $\xi \overset{\pi}{\To} \nu_{\rCospan}$ associates every element $x \in \xi(S, T)$ to a ``corolla'' $\pi(x) \in \nu_{\rCospan}(S,T)$.
  This means a bimodule relative to $\nu_{\rCospan}$ is a $\Iso(\FinSet) \mddash \Iso(\FinSet)$ bimodule.
  We can think of $\pi$ as assigning a ``shape'' to an element $x \in \xi(S, T)$.
  A multiplication $\gamma_{\xi}: \xi \pl_{(\nu)} \xi \to \xi$ takes a collection of elements which form a ``connected shape'' and gives another element.
  In other words, the monoids are properads.
\end{ex}

\begin{ex}
  Let $\rho_{surj}: \SS \times \SS \to \Set$ be the bimodule of surjections.
  Then $\rho_{surj}$ has the unique factorization property with respect to a bimodule $\nu_{surj}$.
  Consider a basic relative bimodule $\xi \overset{\pi}{\To} \nu_{\surj}$.
  The projection $\pi$ maps $\xi(\underline{n}, \underline{1})$ to the unique element in $\nu_{\surj}(\underline{n}, \underline{1})$.
  On the other hand, $\xi(\underline{n}, \underline{m})$ is necessarily empty if $m \neq 1$.
  This means that a basic relative bimodule is the same data as an $\SS$-module.
  A multiplication $\gamma: \xi \pl_{(\nu)} \xi \to \xi$ allows us to combine $\SS$-modules in a manner coherent with the base bimodule of finite surjections, so this is the same data as a (non-unital) operad.
  If a unit $\eta: \tilde\nu_{\SS} \to \xi$ is given, it selects an element of the $\SS$-module of type $1 \to 1$ which satisfies the appropriate unitality conditions.
\end{ex}

\section{Decorations}\label{sec:decorations}

The element construction can be repeated for both level 2 and level 3 structures.
To distinguish this second application of the element construction, it will be called a decoration. This is in keeping with the terminology of \cite{decorated}. We will then establish a compatibility between this notion of a decoration and the chi-construction.

\subsection{Decorations for element representations}\label{sec:decorated-el-rep}

For this section, fix a bimodule $\rho: \act^{\rop} \times \act \to \Set$ with a monoid structure $\gamma: \rho \pl \rho \To \rho$ and a $\Set$-valued element representation $D: \el(\rho) \to \Set$.

\subsubsection{Element representation}

\begin{definition}
   Define the \defn{decorated element category} to be $\el(D)$. Explicitly,
  \begin{enumerate}
    \item The objects are tuples $(A,B; x; u)$ such that $u \in D(x)$ and $x \in \rho(A,B)$.
    \item The morphisms $(A,B; x; u) \to (A',B'; x'; u')$ are given by morphisms $(\sigma \biact \tau): (A,B) \to (A', B')$ in the category $\act^{\rop} \times \act$ such that
      \begin{align}
        x &\;\;\overset{\mathclap{(\sigma \biact \tau)}}{\to}\;\; x' \\
        u &\;\;\overset{\mathclap{D(\sigma \biact \tau)}}{\to}\;\; u'
      \end{align}
  \end{enumerate}
\end{definition}

\begin{definition}
  Define a \defn{decorated\, element\, representation} to be a functor $E: \el(D) \to\! \C$ such that $D: \el(\rho) \to \Set$ is an element representation and $\rho: \act^{\rop} \times \act \to \Set$ is a $\Set$-valued bimodule.
\end{definition}

\subsubsection{Plethysm}

Suppose the element representation $D: \el(\rho) \to \Set$ has a monoid structure $\gamma_D: D \epl D \to D$.
We can then replicate the element plethysm construction for decorated element representations.

\begin{definition}
  Let $z$ be an element of $\el(\rho)$, then elements $w$ in $(D \epl D)(z)$ are represented by elements $u \in D(x)$ and $v \in D(y)$ such that $z = \gamma[x,y]$.
  If $u' \in D(x')$ and $v' \in D(y')$ is another pair such that $z = \gamma[x',y']$, then they are equivalent if only if there exists a morphism $\sigma$ in the groupoid $\act$ such that:
  \begin{align}
    x \overset{(\id \biact \sigma)}{\to} x'
    &\text{ and }
      y' \overset{(\sigma \biact \id)}{\to} y
      \text{ in } \el(\rho) \\
    u \overset{(\id \biact \sigma)}{\to} u'
    &\text{ and }
      v' \overset{(\sigma \biact \id)}{\to} v
      \text{ in } \el(D)
  \end{align}
  Like before, we use the notation $[u,v]$ to represent this equivalence class.
\end{definition}

\begin{definition}
  Recall that we used the fibration structure of $\el(\rho)$ to define the $\hat \otimes$ objects.
  The same thing holds for $\el(D)$, so the construction carries over.
  In this case, given decorated element representations $E_1, E_2: \el(D) \to \C$, we define the object $(E_1 \hat \otimes E_2)[u,v]$ as a co--wedge:
  \begin{equation}
    \begin{tikzcd}
      E_1(u) \otimes E_2(v')
      \arrow[d, "\id \otimes E_2(\sigma \biact \id)"']
      \arrow[rr, "E_1(\id \biact \sigma) \otimes \id"]
      & & E_1(u') \otimes E_2(v')
      \arrow[d, dashed] \\
      E_1(u) \otimes E_2(v)
      \arrow[rr, dashed]
      & & {(E_1 \hat \otimes E_2)[u,v]}
    \end{tikzcd}
  \end{equation}
\end{definition}

\begin{definition}
  Given an element representation $D: \el(\rho) \to \Set$ with an monoid structure and two decorated element representations $E_1, E_2: \el(D) \to \C$, define the \defn{decorated element plethysm} $E_1 \epl_D E_2$ as the following left Kan extension:
  \begin{equation}
    \begin{tikzcd}
      \el(D \epl D)
      \arrow[rd, "\el(\gamma_D)"']
      \arrow[rr, "E_1 \hat \otimes E_2"]
      & & \C \\
      & \el(D)
      \arrow[ru, "E_1 \epl_D E_2 := \Lan_{\el(\gamma_D)}(E_1 \hat \otimes E_2)"', dashed] &
    \end{tikzcd}
  \end{equation}
\end{definition}

\subsection{Decorations and Enrichments}\label{sec:dec-enrich}
We will establish an equivalence between each of their monoid structures.

\begin{lemma}
  Given an element representation $D: \el(\rho) \to \Set$, there is a functor $\Lambda: \el(D) \to \el( \chi_D)$ which induces an isomorphism of categories.
\end{lemma}
\begin{proof}
  Define the functor $\Lambda: \el(D) \to \el( \chi_D)$ by sending an element $u: pt \to D(x)$ in $\el(D)$ to the composition $pt \overset{u}{\to} D(x) \overset{\lambda_x}{\to}   \chi_D(A,B)$.
  By the universal property of the coproduct, a morphism $(u;x;A,B) \to (u';x';A',B')$ induces a morphism $(\lambda_x \circ u;A,B) \to (\lambda_{x'} \circ u';A',B')$ which defines the functor on morphisms.

  Conversely, any element $\tilde u: pt \to  \chi_D(A,B)$ uniquely factors through some element $u: pt \to D(x)$ since we are using the standard chi-construction and $\chi_D$ is $Set$-valued.
  We can then define the functor $\el( \chi_D) \to \el(D)$ on objects by sending $\tilde u$ to $u$.
  The two functors are inverses, so $\el(D)$ and $\el( \chi_D)$ are isomorphic.
\end{proof}

\begin{thm}\label{thm:pluselement}
  Given the monoid structure $\gamma_D: D \epl D \to D$, the following monoidal categories are equivalent:
  \begin{enumerate}
  \item The monoidal category of decorated element representations $E: \el(D) \to \C$ with monoidal product $\epl_D$ with $\gamma_D$ as the underlying monoid structure.
  \item The monoidal category of (ordinary) element representations $F: \el(\chi_D) \to \C$ with $\epl$ as a monoidal product and with \eqref{eq:chi-gamma} as the underlying monoid structure.
    \begin{equation}
      \label{eq:chi-gamma}
      \gamma: \chi_D \pl \chi_D \To \chi_{D \epl D} \To \chi_D
    \end{equation}
  \end{enumerate}
\end{thm}
\begin{proof}
  The isomorphism of the previous lemma induces an isomorphism
  \begin{equation}
    \Lambda^*: \Fun(\el( \chi_D), \C) \to \Fun(\el(D), \C)
  \end{equation}
  It remains to establish the following isomorphism for any pair $F_1, F_2: \el(\chi_D) \to \C$:
  \begin{equation}
    \Lambda^* F_1 \epl_D \Lambda^* F_2 \To \Lambda^*(F_1 \epl F_2)
  \end{equation}
  This follows immediately from the uniqueness of left Kan extensions:
  \begin{equation}
    \begin{tikzcd}
      \el(D \epl D) \arrow[r, "\sim"] \arrow[d, "\gamma_D"'] & \el(\chi_{D \epl D}) \arrow[r, "\sim"] \arrow[d, "\chi_{\gamma_D}"] & \el(\chi_D \pl \chi_D) \arrow[r, "F_1 \hat \otimes F_n"] \arrow[d, "\tilde \gamma"'] & \C \\
      \el(D) \arrow[r, "\Lambda"'] & \el(\chi_D) \arrow[r, equal] & \el(\chi_D) \arrow[ru, "F_1 \epl F_2"'] &
    \end{tikzcd}
  \end{equation}
\end{proof}

\subsection{Decorated basic element representations}

For non-Cartesian decorations, it is necessary to decorate by decomposed elements.
The same thing is true in the bimodule context: we must work with basic representations $\el(\nu) \to \C$ for a $\Set$-valued bimodule $\rho: \act^{\rop} \times \act \to \Set$ with the unique factorization property with respect to $\nu$.

\begin{definition}
  For an element representation $D: \el(\nu) \to \Set$, define a \defn{basic decorated element representation} to be a functor $E: \el(D) \to \C$.
\end{definition}

\begin{example}
  Let $\nu: \Iso(\FinSet)^{\rop} \times \Iso(\FinSet) \to \Set$ be the basic bimodule for finite sets.
  Define $\P: \el(\nu) \to \Set$ so that for a corolla $x \in \nu(A, \{b\})$, we define $\P(x)$ to be the set of orders on $A$.
  Now we can picture the objects of the element category $\el(\P)$ as corollas with an ordered set of in-tails $A$ and a single out-tail $b$.
\end{example}

\subsubsection{Decorated basic plethysm}

\begin{definition}
  Suppose there is a monoid structure $D \bepl D \to D$ and two decorated element representations $E_1, E_2: \el(D) \to \C$, define the \defn{decorated element plethysm} $E_1 \bepl_D E_2$ as the following left Kan extension:
  \begin{equation}
    \begin{tikzcd}
      \el(D \bepl D)
      \arrow[rd, "\el(\gamma_D)"']
      \arrow[rr, "E_1 \hat \otimes E_2"]
      & & \C \\
      & \el(D)
      \arrow[ru, "E_1 \bepl_D E_2 := \Lan_{\el(\gamma_D)}(E_1 \hat \otimes E_2)"', dashed] &
    \end{tikzcd}
  \end{equation}
\end{definition}

\begin{thm}
\label{thm:pluseltbasic}
  Given the monoid structure $\gamma_D: D \bepl D \to D$, the following monoidal categories are equivalent:
  \begin{enumerate}
  \item The monoidal category of decorated basic element representations $E: \el(D) \to \C$ with monoidal product $\bepl_D$ with $\gamma_D$ as the underlying monoid structure.
  \item The monoidal category of (ordinary) element representations $F: \el(\chi_D) \to \C$ with $\bepl$ as a monoidal product.
  \end{enumerate}
\end{thm}
\begin{proof}
  The proof of Theorem~\ref{thm:pluselement} goes through with routine modifications.
  This time, the Kan extension diagram looks like this:
  \begin{equation}
    \begin{tikzcd}
      \el(D \bepl D) \arrow[r, "\sim"] \arrow[d, "\gamma_D"'] & \el(\chi_{D \bepl D}) \arrow[r, "\sim"] \arrow[d, "\chi_{\gamma_D}"] & \el(\beta_D) \arrow[r, "F_1 \hat \otimes F_n"] \arrow[d, "\tilde \gamma"'] & \C \\
      \el(D) \arrow[r, "\Lambda"'] & \el(\chi_D) \arrow[r, equal] & \el(\chi_D) \arrow[ru, "F_1 \epl F_2"'] &
    \end{tikzcd}
  \end{equation}
  where $\beta_D$ is basic bimodule coming from the hereditary condition.
\end{proof}

\section{Algebras}
\label{sec:algebras}

There are two notions of an algebra which generalize the classical notions, those given by natural transformations to the Hom operad \cite{MSS,feynman} and those given by
the functors out of an indexed enrichment such as those described in \cite[Section~4.1]{feynman} and \cite[Section~2.3--2.5]{KMoManin} which are available for relative bimodules.
We will show that both of these notions can be described in the language of bimodules and show that they are equivalent.

\subsection{Algebras of element representations}\label{sec:algebra-el-rep}

Representations of monoidal bimodules are most naturally viewed as natural transformations to a reference functor.
This specializes to the fact that algebras of general operad--like structures are defined relative to some other structure, namely a Hom functor, see Example~\ref{ex:referenceindex}.
For element representations there are two versions for algebras, which we prove to be equivalent.
The second one corresponds to that of a functor out of an indexed enrichment, see Example~\ref{ex:algindexenrich}.

\subsubsection{Algebra structure}
 
\begin{definition}\label{reference}
   A \defn{reference bimodule} is a chosen bimodule $\alpha: \act^{\rop} \times \act \to \C$.
   For a fixed structure bimodule $\rho: \act^{\rop} \times \act \to \Set$, 
   a \defn{$\rho$-algebra} is a natural transformation $\rho \To \alpha$.
 
   The bimodule $\alpha$ also determines the \defn{$\alpha$-reference element representation} $E_{\alpha}$ in $\C$ as the following composition:
   \begin{equation}
     \el(\rho) \overset{\Sigma}{\to} \act^{\rop} \times \act \overset{\alpha}{\to} \C
   \end{equation}
\end{definition}

\begin{definition}\label{O-reference}
  It is common for the reference bimodule $\alpha$ to consist of a $\Gpd$-module structure $\O: \Gpd \to \Iso(\C)$ for a groupoid $\Gpd$ combined with the canonical bimodule of a category.
  Because of this, we will single out the following special case:
  \begin{enumerate}
  \item The \defn{$\O$-reference bimodule} $\alpha_\O$ in a category $\C$ is the following composition:
    \begin{equation}
      \Gpd^{\rop} \times \Gpd \overset{\O^{\rop} \times \O}{\to} \Iso(\C)^{\rop} \times \Iso(\C) \overset{\rho_\C}{\to} \Set
    \end{equation}
  \item The \defn{$\O$-reference element representation} $E_{\O}$ in $\C$ is the following composition:
    \begin{equation}
      \el(\rho) \overset{\Sigma}{\to} \Gpd^{\rop} \times \Gpd \overset{\alpha_{\O}}{\to} \Set
    \end{equation}
  \end{enumerate}
\end{definition}

\begin{example}
  \label{ex:cospan-algebra}
  Fix the groupoid $\act = \SS$ and consider the bimodule $\rho: \SS^{\rop} \times \SS \to \Set$ defined so that $\rho(n,m)$ is the set of equivalence classes of cospans $\{1, \ldots n\} \rightarrow V \leftarrow \{1, \ldots, m\}$.
  Assuming $\C$ is a monoidal category, define a functor $\O: \SS \to \C$ so that $\O(n) := A^{\otimes n}$ for a chosen object $A$ in $\C$ and give $\O(n)$ the natural $\SS_n$-action.
  In this case, the $\O$-reference element representation assigns each cospan $x \in \rho(n,m)$ to an object $E_\O(x) = \rho_{\C}(A^{\otimes n}, A^{\otimes m})$.
  Hence $E_{\O}$ can be understood as the analog of the endomorphism PROP in $\C$.
\end{example}

\begin{definition}
  Fix an element representation $D: \el(\rho) \to \E$.
  \begin{enumerate}
  \item Define a \defn{$D$-algebra (as an element representation)} to be a morphism of element representations $D \To E_\alpha$.
  \item Define a \defn{$D$-algebra (as a bimodule)} to be a morphism of bimodules $\chi_D \To \alpha$.
  \end{enumerate}
\end{definition}

\begin{example}
\label{ex:referenceindex}
  Consider these definitions in the special case of Definition~\ref{O-reference}:
  \begin{enumerate}
  \item A choice of morphism $D \To E_{\O}$ is similar to the usual definition of a $\P$-algebra as an operad morphism $\P \to \mathcal{E}nd$ where $\mathcal{E}nd$ is the endomorphism operad.
  \item Unpacking the definition, a morphism $\chi_D \To \rho_{\O}$ is a family of maps $\chi_D(A, B) \to \rho_{\C}(\O A, \O B)$.
    This is similar to the definition of an algebra in terms of an indexed enrichment with the element category serving as the analog of the plus construction and $\chi$ serving as the analog of an indexed enrichment.
    In particular, if $\chi_D$ is a monoid,
    this data determines a functor of $\arbcat(\chi_D,\gamma)\to \C$. On the underlying groupoid of objects it is given by $\O$ and on the morphisms by the natural transformation.
  \end{enumerate}

\end{example}

\begin{thm}
  \label{thm:algebras-el-rep}
  There is a one-to-one correspondence between $D$-algebras as element representations and $D$-algebras as bimodules.
\end{thm}
\begin{proof}
  Since $\chi_D$ is a left Kan extension and $E_\alpha = \alpha \circ \Sigma$ by definition, we have the following adjunction between the two types of structure morphisms:
  \begin{align*}
    D &\To E_{\alpha} \\
    \chi_D &\To \alpha \qedhere
  \end{align*}
  \end{proof}

\subsubsection{Compatibility of monoid structures}

Now assume that the structure bimodule $\rho$ and the reference bimodule $\alpha$ have a monoid structure.
An algebra for an element plethysm monoid $(D, \gamma_D)$ is an element representation algebra which is compatible with the monoid structure of the reference bimodule.
We will again show that there are two equivalent ways of formulating this.

\begin{lemma}\label{E-plethysm-compatibility}
  For a given $\alpha$, we have $E_{\alpha \pl \alpha} \cong E_{\alpha} \epl E_\alpha$.
  Hence a multiplication $\gamma: \alpha \pl \alpha \To \alpha$ induces a morphism $E_{\gamma}: E_{\alpha} \epl E_{\alpha} \To E_{\alpha}$.
\end{lemma}
\begin{proof}
  For a pair of elements $x,y$ such that $\Sigma(x) = (A,B)$ and $\Sigma(y) = (B,C)$, define $p_{x,y}$ below using the definition of the plethysm product.
  \begin{equation}
    p_{x,y}: \alpha(A,B) \otimes \alpha(B,C) \to (\alpha \pl \alpha)(A,C)
  \end{equation}
  This can be adapted to a family $p_{[x,y]}$ that runs over composable pairs $[x,y] \in \el(\rho \pl \rho)$.
  Moreover, these assemble into a natural transformation $p: (\alpha\Sigma) \hat\otimes (\alpha\Sigma) \To (\alpha \pl \alpha)\Sigma \circ \el(\gamma)$.
  Since $E_{\alpha \pl \alpha} = (\alpha \pl \alpha)\Sigma$ by definition, it suffices to show that Diagram~\eqref{eq:E-plethysm-compatibility} is a left Kan extension.
  \begin{equation}
    \label{eq:E-plethysm-compatibility}
    \begin{tikzcd}
      \el(\rho \pl \rho)
      \arrow[rd, "\el(\gamma)"']
      \arrow[rr, "(\alpha\Sigma) \hat\otimes (\alpha\Sigma)"]
      & & \C \\
      & \el(\rho)
      \arrow[ru, "(\alpha \pl \alpha)\Sigma"'] &
    \end{tikzcd}
  \end{equation}

  Consider an element representation $D: \el(\rho) \to \C$ together with a natural transformation $(\alpha\Sigma) \hat\otimes (\alpha\Sigma) \To D \circ \el(\gamma)$.
  Similar to before, this corresponds to a family of morphisms $\alpha(A,B) \otimes \alpha(B,C) \to D(z)$ running over $\gamma[x,y] = z$.
  This canonically determines a morphism $(\alpha \pl \alpha)\Sigma \To D$ of element representations that factors through $p$.
\end{proof}

\begin{definition}
  Fix an element representation $D: \el(\rho) \to \E$ with a multiplication.
  \begin{enumerate}
  \item Using Lemma~\ref{E-plethysm-compatibility}, define a \defn{$D$-algebra (as an element representation monoid)} to be a morphism of element representations $\kappa: D \To E_\alpha$ making the following diagram commute:
    \begin{equation}
      \begin{tikzcd}
        D \epl D \arrow[r] \arrow[d, "\kappa \pl \kappa"']
        & D \arrow[d, "\kappa"] \\
        E_\alpha \epl E_\alpha \arrow[r]
        & E_{\alpha}
      \end{tikzcd}
    \end{equation}
  \item Define a \defn{$D$-algebra (as a bimodule monoid)} to be a morphism of bimodules $\kappa: \chi_D \To \alpha$ making the following diagram commute:
    \begin{equation}
      \begin{tikzcd}
        \chi_D \pl \chi_D
        \arrow[r]
        \arrow[d, "\kappa \pl \kappa"']
        & \chi_D \arrow[d, "\kappa"] \\
        \alpha \pl \alpha \arrow[r]
        & \alpha
      \end{tikzcd}
    \end{equation}
  \end{enumerate}
\end{definition}

\begin{thm}
  \label{thm:algebras-monoids}
  There is a one-to-one correspondence between $D$-algebras as element representation monoids and $D$-algebras as bimodule monoids.
\end{thm}
\begin{proof}
  By Theorem~\ref{thm:algebras-el-rep}, there is a one-to-one correspondence between the two types of structure maps.
  For the compatibility relations, consider the following commutative diagrams:
  \begin{figure}[h]
    \centering
    \begin{subfigure}{0.22\textwidth}
      \centering
      \begin{tikzcd}
        D \epl D \arrow[r] \arrow[d]
        & D \arrow[d] \\
        E_\alpha \epl E_\alpha \arrow[r]
        & E_{\alpha}
      \end{tikzcd}
      \caption{}
    \end{subfigure}
    \begin{subfigure}{0.22\textwidth}
      \centering
      \begin{tikzcd}
        D \epl D \arrow[r] \arrow[d]
        & D \arrow[d] \\
        E_{\alpha\pl\alpha} \arrow[r]
        & E_{\alpha} 
      \end{tikzcd}
      \caption{}
    \end{subfigure}
    \begin{subfigure}{0.22\textwidth}
      \centering
      \begin{tikzcd}
        \chi_{D \epl D} \arrow[r] \arrow[d]
        & \chi_D \arrow[d] \\
        \alpha\pl\alpha \arrow[r]
        & \alpha
      \end{tikzcd}
      \caption{}
    \end{subfigure}
    \begin{subfigure}{0.22\textwidth}
      \centering
      \begin{tikzcd}
        \chi_{D} \pl \chi_{D} \arrow[r] \arrow[d]
        & \chi_D \arrow[d] \\
        \alpha\pl\alpha \arrow[r]
        & \alpha
      \end{tikzcd}
      \caption{}
    \end{subfigure}
  \end{figure}
  
  Lemma~\ref{E-plethysm-compatibility} allows us to go between (A) and (B).
  The adjunction given by the left Kan extension allows us to go between (B) and (C).
  Lemma~\ref{pl-epl-iso} allows us to go between (C) and (D).
  \end{proof}
  
 \begin{ex}
 \label{ex:algindexenrich}
 In the situation of Example~\ref{ex:referenceindex}, if $\C$ is monoidal, this is equivalent to the algebraic notion of a module, \cite[Section 2.2.3]{KMoManin}. That is there is an action 
 \begin{equation}
   A:  \O\pl\rho\to \O
 \end{equation}
 which sends an element $x\in \rho_{\C}(\unit,\O(X))$ and an element $\phi\in \rho(X,Y)$
 to $\gamma_{\rho_\C}(x,N(\phi))\in \O(Y)$.

 \end{ex} 
  Thus the two notions of an algebra generalize the classical notions, in particular those of algebras over operads using plethysm or natural transformations to the Hom operad \cite{MSS,feynman}. If $\C$ is monoidal, this is then equivalent to the algebraic notion of a module, \cite[Section~2.2.3]{KMoManin}.
    This also identifies the  reference functors of  \cite[Section~6.1]{DDecDennis} as special reference bimodules.

\subsection{Basic algebras}\label{sec:basic-algebras}

For a bimodule $\rho: \act^{\rop} \times \act \to \Set$ with the unique factorization property with respect to $\nu: \act^{\rop} \times \act \to \Set$, there is also a notion of an algebra of a basic element representation $D: \el(\nu) \to \C$.

\subsubsection{Algebra structure}

\begin{definition}
\label{def:strongalgebra}
  A \defn{(strong) algebra} of a basic element representation $D: \el(\nu) \to \C$ can be defined as either:
  \begin{enumerate}
  \item A morphism $\chi_{D} \To \alpha$ of bimodules.
  \item A morphism $D \To e_{\alpha}$ of basic element representations where $e_{\alpha}$ is the composition $\el(\nu) \overset{\iota}{\to} \el(\rho) \overset{E_{\alpha}}{\to} \C$.
  \end{enumerate}
\end{definition}

\begin{prop}\label{prop:strongequiv}
  The two definitions of strong algebras are equivalent.
\end{prop}
\begin{proof}
  Observe that $\Sigma_{\nu}: \el(\nu) \to \act^{\rop} \times \act$ can be factored as $\el(\nu) \overset{\iota}{\to} \el(\rho) \overset{\Sigma_{\rho}}{\to} \act^{\rop} \times \act$.
  The rest is now routine.
\end{proof}

\begin{definition}
  Recall the special case described in Definition~\ref{O-reference} where the reference bimodule $\alpha$ is built out of a $\Gpd$-module structure $\O: \Gpd \to \Iso(\C)$ and the canonical bimodule of a category.
  If $\Gpd \cong \V^{\boxtimes}$, then we can consider $\O$ to be a strong monoidal functor.
  In this case, it suffices to specify a functor $\V \to \Iso(\C)$.
\end{definition}

\begin{example}\label{ex:strong-algebra}
  Let $\rho: \SS^{\rop} \times \SS \to \Set$ be the bimodule where $\rho(n,m)$ is the set of maps of the form $\{1, \ldots, n\} \to \{1, \ldots, m\}$ and let $\nu$ be its basic bimodule.
  Here $\V = 1$ and a functor $1 \to \Iso(\C)$ is just an object $A$ in $\C$.
  This leads to the reference functor $\O(n) := A^{\otimes n}$.

  Now consider an element representation $D: \el(\nu) \to \C$.
  \begin{enumerate}
  \item Using indexed enrichments, an algebra can be described as a strong monoidal functor $\F_D \to \C$.
    Since the functor is strong monoidal, its image is determined up to isomorphism by the image of its irreducible part.
    This datum is more or less the same thing as a bimodule morphism $\chi_D \To \alpha_{\O}$.
  \item The unique element $c_n \in \rho(n,1)$ is assigned to an object $e_\O(c_n) = \rho_{\C}(A^{\otimes n}, A)$.
    Hence an algebra in the element representation form is a family of morphism $D(c_n) \to \rho_{\C}(A^{\otimes n}, A)$, just like the classical definition of an operad.
  \end{enumerate}
\end{example}

\subsubsection{Compatibility and Mergers}\label{sec:compatiblity-mergers}

Just like (ordinary) element representations, there is also a notion of compatibility of monoid structures for basic element representations.
In this case, we need to consider reference bimodules $\alpha$ with an additional structure which we call a merger.

\begin{definition}
  The natural transformation $\boxminus: \alpha^\otimes \To \alpha$ is a \defn{merger} if it is an algebra of the monad $(-)^\otimes$.
  We will defer the discussion of the monad structure of $(-)^\otimes$ until Section~\ref{sec:horizontal-extension-monad}.
\end{definition}

\begin{lemma}\label{lemma:lax-monoidal}
  Let $F$ and $G$ be two functors in $\C^{\M}$ where $\C$ and $\M$ are monoidal, then
  \begin{enumerate}
  \item There is a one-to-one correspondence between natural transformations $F^\otimes \To G$ and families of maps of the following form:
    \begin{equation}
      F(m_1) \otimes \ldots \otimes F(m_n)
      \to
      G(m_1 \otimes \ldots \otimes m_n)
    \end{equation}
  \item A natural transformation $F^{\otimes} \To F$ is precisely the data of a lax monoidal structure on the functor $F$.
  \end{enumerate}
\end{lemma}
\begin{proof}
  Recall that for a functor $F: \M \to \C$, the functor $F^{\otimes}$ is defined as a left Kan extension:
  \begin{equation}
    \Lan_{\mu_{\M}}(\mu_\C \circ F^{\boxtimes})
  \end{equation}
  The left Kan extension construction is a left adjoint, so we have the following correspondence:
  \begin{align}
    \C^{\M}(F^{\otimes}, G)
    &= \C^{\M}(Lan_{\mu_{\M}}(\mu_\C \circ F^{\boxtimes}), G) \\
    &\cong \C^{\M^{\boxtimes}}(\mu_\C \circ F^{\boxtimes}, G \circ \mu_{\M})
  \end{align}
  The resulting bijection gives us our desired correspondence.
  The second statement follows immediately from the first.
\end{proof}

\begin{ex}[Corollas]
  Recall the trivial bimodule $\tau_\SS: \SS^{\rop} \times \SS \to \Set$ defined by $\tau_\SS(n,m) = pt$.
  In Example~\ref{symmetric-ext-example}, we showed that $\nu := \tau_\SS$ could be understood as the bimodule of basic corollas and $\rho := \tau_{\SS}^\otimes$ could be understood as the bimodule of aggregated corollas.
  There is a canonical natural transformation $\boxminus: \rho \To \nu$ sending aggregates in $\rho(n,m)$ to the only element in $\nu(n,m)$ such that the pair $(\nu, \rho \To \nu)$ satisfies the conditions for a $(-)^\otimes$-algebra.
\end{ex}

\begin{ex}[Monoidal category]
  Consider the canonical bimodule $\rho_{\C}: \Iso(\C)^{\rop} \times \Iso(\C) \to \Set$ for a monoidal category $\C$.
  The monoidal structure of $\C$ induces morphisms of the following form:
  \begin{equation}
    \rho_{\C}(A_1, B_1) \times \ldots \times \rho_{\C}(A_n, B_n) \to \rho_{\C}\left(\bigotimes_i A_i, \bigotimes_i B_i\right)
  \end{equation}
  By Lemma~\ref{lemma:lax-monoidal}, this can be packaged as a merger $\rho_{\C}^\otimes \To \rho_{\C}$.
\end{ex}

\begin{lemma}\label{lemma:basic-to-ordinary-ref-op}
  For any bimodule $\alpha$, there is a canonical morphism $e_\alpha^\otimes \To E_{\alpha^\otimes}$ of element representations.
\end{lemma}
\begin{proof}
  Observe that $\Sigma_\rho \cong \Sigma_\nu^\otimes$.
  Putting the horizontal extension diagrams of $\Sigma_\nu^\otimes$ and $\alpha^\otimes$ together in Diagram~\eqref{eq:otimes-Sigma-alpha}, we obtain a natural transformation $\mu_\C \circ (\alpha \circ \Sigma_\nu)^\boxtimes \To \alpha^\otimes \circ \Sigma_\rho \circ \mu$.
  \begin{equation}
    \label{eq:otimes-Sigma-alpha}
    \begin{tikzcd}
      \el(\nu)^\boxtimes
      \arrow[d, "\mu"]
      \arrow[r, "\Sigma_\nu^\boxtimes"]
      & (\act^{\rop} \times \act)^\boxtimes
      \arrow[d, "\mu_{\act}"]
      \arrow[r, "\alpha^\boxtimes"]
      & \C^\boxtimes
      \arrow[d, "\mu_{\C}"] \\
      \el(\rho) \arrow[r, "\Sigma_\rho"']
      & \act^{\rop} \times \act
      \arrow[r, "\alpha^\otimes"']
      & \C
      \arrow[Rightarrow, shorten >=0.3cm,shorten <=0.4cm, from=1-2, to=2-1]
      \arrow[Rightarrow, shorten >=0.3cm,shorten <=0.4cm, from=1-3, to=2-2]
    \end{tikzcd}
  \end{equation}
  We can rewrite this as $\mu_\C \circ e_\alpha^\boxtimes \To E_{\alpha^\otimes} \circ \mu$.
  Then the definition of horizontal extension gives our desired morphism $e_\alpha^\otimes \To E_{\alpha^\otimes}$.
\end{proof}

\begin{lemma}
  \label{lemma:basic-endomorphism}
  If $\alpha$ has a monoid and merger structure, then there is a canonical morphism $e_\alpha \bepl e_\alpha \To e_\alpha$ of basic element representations.
\end{lemma}
\begin{proof}
  Using Lemma~\ref{lemma:basic-to-ordinary-ref-op} and the merger structure, there is a natural transformation $e_\alpha^\otimes \To E_{\alpha^\otimes} \To E_\alpha$.
  Now consider the following diagram:
  \begin{equation}
    \begin{tikzcd}
      {\el(\beta)} & {\el(\rho \pl \rho)} && \C \\
      & {\el(\nu)} & {\el(\rho)}
      \arrow[from=1-1, to=2-2]
      \arrow[from=1-1, to=1-2]
      \arrow[from=2-2, to=2-3]
      \arrow[""{name=0, anchor=center, inner sep=0}, "{e_\alpha^\otimes \hat \otimes e_\alpha^\otimes}", bend left = 10, from=1-2, to=1-4]
      \arrow[from=1-2, to=2-3]
      \arrow["{E_\alpha}"', from=2-3, to=1-4]
      \arrow[""{name=1, anchor=center, inner sep=0}, "{E_\alpha \hat\otimes E_\alpha}"', bend right = 10, from=1-2, to=1-4]
      \arrow[shorten <=2pt, shorten >=2pt, Rightarrow, from=0, to=1]
    \end{tikzcd}
  \end{equation}
  Using the definition of $\bepl$ as a left Kan extension and $e_\alpha = E_\alpha \circ \iota$, we obtain a morphism\linebreak $e_\alpha \bepl e_\alpha \To e_\alpha$.
\end{proof}

\begin{definition}\label{def:strongalgebra-as-monoid}
  Let $D: \el(\nu) \to \C$ be a basic element representation equipped with a monoid structure $D \bepl D \To D$.
  A \defn{$D$-algebra (as a basic element representation monoid)} is a morphism $\kappa: D \To e_{\alpha}$ of basic element representations making the following diagram commute:
  \begin{equation}
    \label{eq:basic-algebra}
    \begin{tikzcd}
      D \bepl D \arrow[r] \arrow[d, "\kappa \bepl \kappa"']
      & D \arrow[d, "\kappa"] \\
      e_\alpha \bepl e_\alpha \arrow[r]
      & e_{\alpha}
    \end{tikzcd}
  \end{equation}
\end{definition}

\begin{example}
  Consider the situation described in Example~\ref{ex:strong-algebra} where $\rho: \SS^{\rop} \times \SS \to \Set$ is the factorizable bimodule defined so that $\rho(n,m)$ is the set of maps $\{1, \ldots, n\} \to \{1, \ldots, m\}$.
  Let $\O: \SS \to \C$ be strong monoidal and consider the associated reference bimodule $\alpha_{\O}$.
  This has a canonical merger structure $\alpha_{\O}^\otimes \To \alpha_{\O}$ determined by $\O$ being a strong monoidal functor and $\C$ being a monoidal category.
  If we think of $e_{\O}$ as rooted corollas decorated with the morphisms of $\Hom(A^{\otimes n}, A)$, then the element representation $e_{\O} \bepl e_{\O}$ can be pictured as connected 2-level trees whose vertices are decorated with the morphisms of $\Hom(A^{\otimes n}, A)$.
  Under this interpretation, the morphism $e_\O \bepl e_\O \To e_\O$ described in Lemma~\ref{lemma:basic-endomorphism} merges so that the second level is sent to an element of $\Hom(A^{\otimes(n_1 + \ldots + n_k)}, A^{\otimes k})$ which is then composed with the element $\Hom(A^{\otimes k}, A)$ decorating the first level.
  This is precisely the operadic composition of the endomorphism operad.
  
  Now consider a basic element representation $D: \el(\nu) \to \C$ equipped with a multiplication $\gamma: D \bepl D \To D$.
  We saw in Example~\ref{ex:strong-algebra} that a morphism $\kappa: D \To e_{\O}$ is analogous to the structure maps for an algebra over an operad.
  The commutativity of Diagram~\eqref{eq:basic-algebra} implies that the morphism $\kappa: D \To e_{\O}$ is compatible with the two operadic compositions.
\end{example}

\subsection{Lax algebras}\label{sec:lax-algebras}

Naturally, Definition~\ref{def:strongalgebra-as-monoid} can be expressed as a bimodule morphism.
Consider the following diagram where the top square is the horizontal extension of \eqref{eq:basic-algebra} and the bottom rectangle commutes by the definition of $e_\alpha \bepl e_\alpha \To e_\alpha$.
\begin{equation}\label{eq:defwmerger}
  \begin{tikzcd}[row sep = small]
    D^\otimes \epl D^\otimes \arrow[r] \arrow[d] & D^\otimes \arrow[d] \\
    e_\alpha^\otimes \epl e_\alpha^\otimes \arrow[r] \arrow[d] & e_{\alpha}^\otimes \arrow[d] \\
    E_{\alpha^\otimes} \epl E_{\alpha^\otimes} \arrow[d]       & E_{\alpha^\otimes} \arrow[d] \\
    E_\alpha \epl E_\alpha \arrow[r] & E_\alpha
  \end{tikzcd}
\end{equation}
By Theorem~\ref{thm:algebras-monoids}, the composition $D^\otimes \To E_\alpha$ of the vertical morphism in \eqref{eq:defwmerger} determines a morphism $\chi_D^\otimes \To \alpha$ by adjunction.
Applying this to all of \eqref{eq:defwmerger} gives us the following diagram:
\begin{equation}
  \label{eq:lax-diagram}
  \begin{tikzcd}
    \chi_D^\otimes \pl \chi_D^\otimes \arrow[r] \arrow[d]
    & \chi_D^\otimes \arrow[d] \\
    \alpha \pl \alpha \arrow[r]
    & \alpha
  \end{tikzcd}
\end{equation}

The morphisms $\chi_D^\otimes \To \alpha$ and the algebra structure described in Diagram~\eqref{eq:lax-diagram} satisfy some fairly strong conditions tied to the merger structure on $\alpha$.
If we consider arbitrary morphisms $\chi_D^\otimes \To \alpha$, then we obtain a weaker notion that we call a lax algebra.

\goodbreak
\begin{definition}\label{def:laxalgebra}
  A \defn{lax algebra} of a basic element representation $D: \el(\nu) \to \C$ can be defined as either
  \begin{enumerate}
  \item A morphism $\chi_{D^\otimes} \To \alpha$ of bimodules
  \item A morphism $D^{\otimes} \To E_{\alpha}$ of element representations.
  \end{enumerate}
\end{definition}

This natural categorical construction allows several generalizations, e.g.\ change the merger structure of $\alpha$, the monoidal structure of $\alpha$, or use an $\alpha$ that does not have a merger structure. The utility is evidenced by the following characterization.

\begin{cor}
    Unraveling the definitions for a morphism $D^{\otimes} \To E_{\alpha}$ in a similar manner as Lemma~\ref{lemma:lax-monoidal}, this is a natural family of morphisms
    \begin{equation}
      D(x_1) \otimes \ldots \otimes D(x_n) \to E_{\alpha}(x_1 \otimes \ldots \otimes x_n)
    \end{equation}
\end{cor}

\begin{definition}
  The compatibility conditions for lax algebras are analogous to the ones for (ordinary) algebras.
  \begin{enumerate}
  \item A \defn{lax $D$-algebra (as an element representation monoid)} is a morphism $\kappa: D^\otimes \To E_\alpha$ making the following diagram commute:
    \begin{equation}
      \begin{tikzcd}
        D^\otimes \epl D^\otimes \arrow[r] \arrow[d, "\kappa \pl \kappa"']
        & D^\otimes \arrow[d, "\kappa"] \\
        E_\alpha \epl E_\alpha \arrow[r]
        & E_{\alpha}
      \end{tikzcd}
    \end{equation}
  \item A \defn{lax $D$-algebra (as a bimodule monoid)} is a morphism $\kappa: \chi_{D^\otimes} \To \alpha$ making the following diagram commute:
    \begin{equation}
      \begin{tikzcd}
        \chi_{D^\otimes} \pl \chi_{D^\otimes}
        \arrow[r]
        \arrow[d, "\kappa \pl \kappa"']
        & \chi_{D^\otimes} \arrow[d, "\kappa"] \\
        \alpha \pl \alpha \arrow[r]
        & \alpha
      \end{tikzcd}
    \end{equation}
  \end{enumerate}
\end{definition}

\begin{prop}\label{prop:laxequiv}
  The two definitions of a lax algebra are equivalent and the two notions of compatibility are also equivalent.
\end{prop}
\begin{proof}
  This is a special case of Theorem~\ref{thm:algebras-el-rep} and Theorem~\ref{thm:algebras-monoids}.
\end{proof}
\begin{rmk}
    We expect that this will play a role in the study of $B_+$ operators and twists like in \cite[Section~2.2]{hoch2}. These connections will be studied in the future.
\end{rmk}
\subsection{Horizontal extension as a monad}
\label{sec:horizontal-extension-monad}

In Section~\ref{sec:compatiblity-mergers}, we defined a merger $\boxminus: \alpha^\otimes \To \alpha$ to be an algebra of the natural monad structure for horizontal extension.
In this section, we will formally establish this monad structure by making extensive use of the coend formula \ref{coend-formula} and the technique of coend calculus.

\begin{as}
  Monoidal products of the target category commute with colimits.
\end{as}

\begin{lemma}
If $F^{\otimes}$ can be computed point-wise, then there is a canonical natural transformation:
  \begin{equation}
    \eta: (-) \To (-)^\otimes
  \end{equation}
\end{lemma}
\begin{proof}
  By the density formula, we have the following natural isomorphism:
  \begin{equation}
    F(-) \simTo \int^{x : \M} \Hom(x,-) \odot F(x)
  \end{equation}

  Let $\M^{\boxtimes n}$ denote the full subcategory of $\M^{\boxtimes}$ consisting of length $n$ words.
  Observe that the category $\M^\boxtimes$ is then a disjoint union of the subcategories $\{ \M^{\boxtimes n}\}_{n \in \N}$.
  This allows us to write the coend formula in components:
  \begin{align*}
    F^\otimes(a)
    &= \int^{x_1 \boxtimes \ldots \boxtimes x_n : \M^{\boxtimes}}
      \Hom(x_1 \otimes \ldots \otimes x_n, a)
      \odot (F(x_1) \otimes \ldots \otimes F(x_n)) \\
    &\cong F^\otimes(a) \cong
      \coprod_{n \in \N}
      \int^{x_1 \boxtimes \ldots \boxtimes x_n : \M^{\boxtimes n}}
      \Hom(x_1 \otimes \ldots \otimes x_n, a)
      \odot (F(x_1) \otimes \ldots \otimes F(x_n))
  \end{align*}
  Hence, we have the following component morphism:
  \begin{equation}
    \int^{x : \M} \Hom(x,-) \odot F(x)
    \To
    \coprod_{n \in \N}
    \int^{x_1 \boxtimes \ldots \boxtimes x_n : \M^{\boxtimes n}}
    \Hom(x_1 \otimes \ldots \otimes x_n, -)
    \odot (F(x_1) \otimes \ldots \otimes F(x_n))
  \end{equation}
  Composing this component with the previous natural transformation gives the desired natural transformation $\eta: F \To F^{\otimes}$.
\end{proof}

\begin{example}
  For the other half of the monad structure, consider the bimodule $\tau_{\SS}^{\otimes}$ of corollas and the bimodule $(\tau_{\SS}^{\otimes})^{\otimes}$ of boxed corollas described in Example~\ref{boxed-example} as motivating examples.
  One would expect there to be a natural transformation $(\tau_{\SS}^{\otimes})^{\otimes} \To \tau_{\SS}^{\otimes}$ which ``unboxes'' the corollas.
  This unboxing natural transformation is a special case of a more general construction for horizontal extensions described by the following lemma.
\end{example}

\begin{lemma}
The monoidal category structure of $\M$ induces a canonical natural transformation  
  \begin{equation}
    \mu: ((-)^{\otimes})^\otimes \To (-)^{\otimes}
  \end{equation}
\end{lemma}
\begin{proof}
  Recall the coend formulation of $F^\otimes(x)$:
  \begin{equation}
    F^\otimes(x_1) \cong
    \int^{\boxtimes y_1^j : \M^{\boxtimes}}
    \Hom\left(\bigotimes y_1^j, x_1\right)
    \odot \left( \bigotimes_j F(y_1^j) \right)
  \end{equation}
  Using the assumption that the monoidal product commutes with colimits we get:
  \begin{equation}
    \label{many-F-otimes}
    F^\otimes(x_1) \otimes \ldots \otimes F^\otimes(x_k)
    \cong
    \int^{\boxtimes y_1^j : \M^{\boxtimes}}
    \cdots
    \int^{\boxtimes y_k^j : \M^{\boxtimes}}
    \left(
      \bigotimes_{i}
      \Hom(\bigotimes_j y_i^j, x_i)
    \right)
    \odot
    \left(
      \bigotimes_{i}
      \left( \bigotimes_j F(y_i^j) \right)
    \right)
  \end{equation}
  Take the coend formula for $(F^{\otimes})^\otimes$ and plug in Equation~\eqref{many-F-otimes} to get the second line.
    \begin{align*}
    (F^{\otimes})^\otimes
    & \cong \int^{\bigboxtimes_i x_i : \M^\boxtimes} \Hom\left(\bigotimes_i x_i, x\right) \odot \left( \bigotimes_i F^\otimes(x_i) \right) \\
    & \cong
      \int^{\bigboxtimes_i x_i : \M^\boxtimes} \Hom\left(\bigotimes_i x_i, x\right) \odot
      \int^{\boxtimes y_1^j : \M^{\boxtimes}}
      \cdots
      \int^{\boxtimes y_k^j : \M^{\boxtimes}}
      \left( \bigotimes_{i} \Hom\left(\bigotimes_j y_i^j,x_i\right) \right)
      \odot
      \left( \bigotimes_{i,j} F(y_i^j) \right)
  \end{align*}

  Now we will rearrange this.
  First move all the coends to the outside and consolidate the two copowers into one copower.
  \begin{equation}
    (F^{\otimes})^\otimes \cong
      \int^{\bigboxtimes_i x_i : \M^\boxtimes}\!\!
      \int^{\boxtimes y_1^j : \M^{\boxtimes}}\!\!
      \cdots
      \int^{\boxtimes y_k^j : \M^{\boxtimes}}\!
      \left(
      \Hom\left(\bigotimes_i x_i, x\right)
      \otimes
      \bigotimes_{i}
      \Hom\left(\bigotimes_j y_i^j, x_i\right)
      \right)
      \odot
      \left( \bigotimes_{i,j} F(y_i^j) \right)
  \end{equation}
  Now apply Fubini so that the coend involving the $\bigboxtimes_i x_i$ is the innermost coend and the $y_i^j$ become a single coend.
  \begin{equation}
    \label{eq:almost-reduced}
    (F^{\otimes})^\otimes(x)
    \cong
    \int^{\boxtimes y_i^j : \M^{\boxtimes}}
    \int^{\bigboxtimes_i x_i : \M^\boxtimes}
    \left(
      \Hom\left(\bigotimes_i x_i, x\right)
      \otimes
      \bigotimes_{i}
      \Hom\left(\bigotimes_j y_i^j, x_i\right)
    \right)
    \odot
    \left(
      \bigotimes_{i,j} F(y_i^j)
    \right)
  \end{equation}

  So far, we have taken this iterated horizontal extension and turned it into something that resembles a single horizontal extension but with a more complex copower.
  We will use the structure of $\M$ to reduce this to a true horizontal extension.
  First, the monoidal product of $\M$ provides a natural morphism \eqref{mu-reduction}.
  \begin{equation}
  \label{mu-reduction}
  \bigotimes_i \Hom\left(\bigotimes_j y_i^j, x_i\right) \to \Hom\left(\bigotimes_{i,j} y_i^j, \bigotimes_i x_i\right)
  \end{equation}
  Second, the composition structure of $\M$ provides a natural morphism \eqref{gamma-reduction}.
  \begin{equation}
    \label{gamma-reduction}
    \int^{\bigboxtimes_i x_i : \M^\boxtimes}
    \Hom\left(\bigotimes_i x_i, x\right)
    \otimes
    \Hom\left(\bigotimes_{i,j} y_i^j, \bigotimes_i x_i\right)
    \to
    \Hom\left(\bigotimes_{i,j} y_i^j, x\right)
  \end{equation}
  Taking the copower of \eqref{eq:almost-reduced} and applying \eqref{mu-reduction} then \eqref{gamma-reduction}, we get our desired natural morphism $\mu_x: (F^\otimes)^\otimes(x) \to F^\otimes(x)$ for each object $x$.
\end{proof}

\begin{thm}\label{thm:mergermonad}
  Considering the special case of bimodules such that $\M = \act^{op} \times \act$,
  the functor $(-)^\otimes$ forms a monad with $\mu: ((-)^\otimes)^\otimes \To (-)^\otimes$ as the multiplication and $\eta: (-) \To (-)^\otimes$ as the unit when  the action category factorizes as $\act \cong \V^\otimes$.
\end{thm}
\begin{proof}
  Associativity follows from Fubini's theorem and the associativity of the composition in the source category $\M$.
  To verify unitality, we need to check that the following diagram commutes for an arbitrary functor $F: \M \to \C$:
  \begin{equation}
    \begin{tikzcd}
      F^{\otimes}
      \arrow[rd, "Id"', Rightarrow]
      \arrow[r, "\eta_{F^\otimes}", Rightarrow]
      & (F^\otimes)^{\otimes}
      \arrow[d, "\mu", Rightarrow]
      & F^\otimes
      \arrow[ld, "Id", Rightarrow]
      \arrow[l, "\eta_F^\otimes"', Rightarrow] \\
      & F^\otimes &
    \end{tikzcd}
  \end{equation}
  The left side is straight-forward.
  For the right side, the transformation $\eta_{F}^\otimes$ involves sending factors $F(x_1) \otimes \ldots \otimes F(x_k)$ to factors $F^\otimes(x_1) \otimes \ldots \otimes F^\otimes(x_k)$.
  In this degenerate case, the $\int^{\bigboxtimes_j y_i^j :\M^\boxtimes}$ coends can be reduced to $\int^{y_i:\M}$ coends.
  Hence, we can consider the following version of \eqref{eq:almost-reduced} using this reduction:
    \begin{equation}
    \int^{\bigboxtimes_i y_i : \M^\boxtimes}
    \int^{\bigboxtimes_i x_i : \M^\boxtimes}
    \left(
      \Hom\left(\bigotimes_i x_i, x\right)
      \otimes
      \bigotimes_{i}
      \Hom(y_i, x_i)
    \right)
    \odot
    \left(
      \bigotimes_{i} F(y_i)
    \right)
  \end{equation}
  Using the assumption that $\M = \act^{op} \times \act$, write $x = (A, B)$, $x_i = (A_i, B_i)$, and $y_i = (A'_i, B'_i)$.
  Rewriting this using the notation described in \ref{nota-hereditary-tensor} and used in Lemma~\ref{hereditary-tensor}, we get
  \begin{equation}
    \int^{\bigboxtimes_i (A'_i, B'_i) : \M^\boxtimes}
    \int^{\bigboxtimes_i (A_i, B_i) : \M^\boxtimes}
    \left(
      H^A_{(A_i)}
      H^{(B_i)}_B
      \prod H^{A_i}_{A'_i}
      \prod H^{B'_i}_{B_i}
    \right)
    \odot
    \left(
      \bigotimes_{i} F(A'_i, B'_i)
    \right)
  \end{equation}
  Applying the appropriate versions of \eqref{mu-reduction} and \eqref{gamma-reduction} to this yields the following
  \begin{equation}
    \int^{\boxtimes_i (A'_i, B'_i) : \M^\boxtimes}
    H^A_{A'_i}
    H^{B'_i}_B
    \odot
    \left(
      \bigotimes_{i} F(A'_i, B'_i)
    \right)
  \end{equation}
  Because of the factorization assumption $\act \cong \V^\otimes$, this agrees with the type of reduction one gets from applying the density formula on both sides.
  Thus the net result is that we obtain the same thing that we started with, so the right side also commutes.
\end{proof}

\section{Plethysms and plus constructions}\label{sec:plethysms-plus-const}

In this section, we connect the theory of factorizable bimodules to the classical cases of operads and operad--like structures.

\subsection{Recollection on Feynman categories and UFCs}\label{sec:recollection-FC-UFC}

To make contact with the classical theory, we will use Feynman categories and the related notion of a unique factorization category.
We give a brief account here and refer to \cite{feynman, KMoManin} for the details.

\begin{definition}[\cite{KMoManin}]\label{def:UFC}
  A \defn{unique factorization category} (UFC) is a symmetric monoidal category $\M$ along with a tuple $(\V, \imath, \Indec, \jmath)$ such that:
  \begin{enumerate}
  \item $\V$ is a groupoid and $\imath: \V \to \Iso(\M)$ is a functor such that the composition $\V^{\boxtimes} \overset{\imath^{\boxtimes}}{\to} \Iso(\M)^{\boxtimes} \overset{\mu}{\to} \Iso(\M)$ is an equivalence of categories.
  \item $\Indec$ is a groupoid and $\jmath: \Indec \to \Iso(\M\downarrow \M)$ is a functor such that the composition $\Indec^{\boxtimes} \overset{\jmath^{\boxtimes}}{\to} \Iso(\M \downarrow \M)^{\boxtimes} \overset{\mu}{\to} \Iso(\M \downarrow \M)$ is an equivalence of categories.
  \item The slice categories of $\M$ are essentially small.
  \end{enumerate}
\end{definition}

\begin{rmk}
  In \cite{KMoManin}, we introduced the notion of a \emph{prehereditary} UFC.
  There are several characterizations of this and we refer the reader to \cite[Section~6]{KMoManin} for a detailed discussion.
  In Proposition~\ref{prop:UFC}, we will show how this is related to the notion of heredity described in Section~\ref{sec:heredity}.
\end{rmk}

\begin{example}[Finite Sets]
  The category $\FinSet$ of finite sets is a unique factorization category whose basic objects are singletons and whose basic morphisms are of the form $X \to \{y\}$.
\end{example}

\begin{example}[Cospans]
  Let $\text{nd-}\Cospan$ be the subcategory of cospans with all objects and where both legs are surjective.
  In particular, this means that $\varnothing \rightarrow \varnothing \leftarrow \varnothing$ is the only element of $\Hom(\varnothing, \varnothing)$.
  This forms a hereditary UFC with cospans $S \rightarrow \{v\} \leftarrow T$ as the basic morphisms.
\end{example}

In terms of UFCs, one can define the Feynman categories of \cite{feynman} as follows:
\label{def:FC}
\begin{definition}
  A \defn{Feynman category} is a unique factorization category $\F$ such that $\Indec = \Iso(\F \downarrow \V)$.
\end{definition}

\begin{example}[Modular operads]
  There are many examples of Feynman categories indexed over Borisov--Manin graphs, see \cite[Sections 2.1--2.3]{feynman}.
  In particular, there is a Feynman category $\modular$ whose basic objects are genus-marked corollas and whose basic morphisms are connected graphs.
  This is the Feynman category that corepresents modular operads.
\end{example}

\begin{proposition}[\cite{KMoManin}]
  Feynman categories are prehereditary unique factorization categories.
\end{proposition}

\subsection{UFCs and bimodules}
\label{sec:ufcs-bimodules}

Recall that in Section~\ref{sec:cat-unital-monoids} we established a correspondence between unital monoids and categories.
We will see that UFCs are the categorical analogs of bimodules with unique factorization properties.
There are some subtleties involving $\rho(1_{\Gpd}, 1_{\Gpd})$ which we address in two remarks before proving that a UFC canonically determines a bimodule.

\begin{rmk}
  Let $\M$ be a UFC with $\Gpd$ as its underlying groupoid.
  Strictly speaking, the canonical bimodule $\rho_{\M}$ can not be factorized because of the element $\id_{1_{\M}} \in \rho_{\M}(1_{\Gpd},1_{\Gpd})$.
  If $\id_{1_{\M}}$ is absent from the basic bimodule, then it can't be in the horizontal extension either.
  On the other hand, if $\id_{1_{\M}}$ is present in the basic bimodule, then the horizontal extension will have strings of $\id_{1_{\M}}$ of different lengths.
  These strings represent distinct elements in the set $\nu^{\otimes}(1_{\Gpd},1_{\Gpd})$, which is not what we want.
  This means it is necessary to work with a reduced bimodule $\bar \rho_\M$ without the element $\id_{1_{\M}} \in \rho_{\M}(1_{\Gpd},1_{\Gpd})$.
\end{rmk}

\begin{rmk}
    \label{rmk:zero-to-zero}
    In general, $\rho(1_{\Gpd}, 1_{\Gpd})$ can be nonempty even for a factorizable bimodule.
    To see this, consider the coend formula for horizontal extension in this case:
    \begin{equation}
        \nu^\otimes(1_{\Gpd}, 1_{\Gpd}) = \int^{\boxtimes_i (A_i, B_i)} \Hom_{\Gpd^{op} \times \Gpd}(\otimes_i (A_i, B_i), (1_{\Gpd}, 1_{\Gpd})) \odot ( \nu(A_i, B_i) \otimes \ldots)
    \end{equation}
    If we assume that $\Gpd$ is strictly factorizable for the sake of demonstration, then the hom sets will be a singleton if the $A_i$ and $B_i$ are all units and empty otherwise.
    This means the only interesting cowedges look like Diagram~\eqref{zero-zero-coend} where $C$ is a commutativity constraint:
    \begin{equation}
        \label{zero-zero-coend}
        \begin{tikzcd}
        {\Hom(\ldots) \odot ( \nu(1_{\Gpd}, 1_{\Gpd}) \otimes \ldots )}
        \arrow[r, "C_*"]
        \arrow[d, "C^*"']
        & \Hom(\ldots) \odot ( \nu(1_{\Gpd}, 1_{\Gpd}) \otimes \ldots) \\
        {\Hom(\ldots) \odot ( \nu(1_{\Gpd}, 1_{\Gpd}) \otimes \ldots )}
        & 
        \end{tikzcd}
    \end{equation}
    The action $C^*$ on the left side of $\odot$ is the identity and the action $C_*$ on the right side of $\odot$ rearranges the factors.
    Hence representations of the elements in $\nu^\otimes(1_{\Gpd}, 1_{\Gpd})$ are equivalent if they differ by a commutativity constraint.
\end{rmk}

\begin{example}
    The bimodule $\rho_{Cospan}$ is an example of this where $\nu_{\Cospan}(\varnothing, \varnothing)$ is a singleton containing the cospan $\varnothing \rightarrow \{*\} \leftarrow \varnothing$.
\end{example}

\begin{prop}
  \label{prop:UFC}
  For a strict prehereditary UFC $\M$ with an underlying groupoid $\Gpd = \Iso(\M)$, let $\bar \rho_\M$ be the canonical $\Gpd \mdash \Gpd$ bimodule monoid without the element $\id_{1_{\M}} \in \rho_{\M}(1_{\M},1_{\M})$.
  Then $\bar \rho_{\M}$ has the unique factorization property and both the canonical multiplication $\gamma_{\M}$ and the canonical unit are hereditary morphism.
\end{prop}
\begin{proof}
  Since $\M$ is a strict UFC, there is a category $\Indec$ such that $\Iso(\M \downarrow \M) = \Indec^\boxtimes$.
  Then, for any pair of objects $A$ and $B$, define $\nu_{\M}(A,B)$ to be the subset of elements $x \in \rho_\M(A,B)$ such that $x$ belongs in the image of $\Indec$.
  The collection $\nu_{\M}$ is a bimodule because $\Indec$ is a subcategory of $\Iso(\M \downarrow \M)$.
  The isomorphism $\bar \rho_{\M} \cong \nu_{\M}^\otimes$ follows from unique factorization of morphisms.

  Let $\beta_{\M}$ be the pullback of $\nu_{\M}$ along $\gamma_{\M}$.
  The prehereditary property implies that the connected components must be precisely the elements of $\beta_{\M}$.
  Since $\beta_{\M}$ is the bimodule of connected components, we have $\beta_{\M}^{\otimes} \cong \bar \rho_{\M} \pl \bar \rho_{\M}$.
  It follows that $\gamma_{\M}$ is a hereditary morphism.
  The isomorphism property of a UFC guarantees that the counter-clockwise composition factors into the clockwise composition.
  \begin{equation}
    \begin{tikzcd}
      \nu_{\Gpd} \arrow[d, tail] \arrow[r, dashed]
      & \nu_{\M} \arrow[d, tail] \\
      \bar \rho_{\Gpd} \arrow[r, "\eta_{\M}"']
      & \bar \rho_{\M}
    \end{tikzcd}
  \end{equation}
  On the other hand, there is not a larger sub-bimodule with this property since a non-basic element would necessarily go to another non-basic element.
  Hence this diagram is a pullback.
  Since $\bar \rho_{\Gpd} \cong \bar \nu_{\Gpd}^{\otimes}$ by construction, this remains a pullback after horizontal extension.
\end{proof}

\subsection{Relation to the plus construction}

Perhaps the most important property of prehereditary unique factorization categories is the following proposition:

\begin{proposition}[\cite{KMoManin}]
  The plus construction of a prehereditary unique factorization category is a Feynman category.
\end{proposition}

Many Feynman categories are plus constructions of some UFC.
We will elaborate on some examples of this in the next section.
For further examples see \cite{KMoManin, monaco2022calculations}.
In this section, we will give an important structural consequence of this by showing that if a Feynman category $\F$ can be expressed as the plus construction of a UFC, then the $\F\dashops$ can be expressed as plethysm monoids of an element representation.

\begin{lemma}\label{monoid-to-op}
  Let $\M$ be a prehereditary UFC.
  Then a basic element representation $D: \el(\nu_\M) \to \C$ together with a multiplication $\gamma: D \bepl D \to D$ determines a strong monoidal functor $\O_\gamma: \M^+ \to \C$.
\end{lemma}
\begin{proof}
  Recall that the basic element representation $D: \el(\nu_\M) \to \C$ induces an element representation $D^{\otimes}: \el(\rho_\M) \to \C$.
  For an object $\phi \in \M^+$, we define $\O_{\gamma}(\phi) := D^{\otimes}(\phi)$.

  Now we address the morphisms.
  Recall that the morphisms of $\M^+$ are generated by two-level graphs.
  Moreover, we only need to consider the connected case, since $\O_{\gamma}$ is intended to be strong monoidal.
  With this in mind, let $\G: \phi \to \psi$ be a morphism in $\M^+$ such that $\phi = \phi_1 \otimes \phi_2$ and the composition $\phi_1 \circ \phi_2 = \psi$ is a basic object.

  Because $\psi$ is basic, we have $[\phi_1, \phi_2] \in \el(\beta)$ by definition.
  Consider the quotient $q: D^{\otimes}(\phi) \to (D^{\otimes} \hat \otimes D^{\otimes})[\phi_1, \phi_2]$.
  Now define $\O_{\gamma}(\G): \O_{\gamma}(\phi) \to \O_{\gamma}(\psi)$ as the following composition:
  \begin{equation}
    \begin{tikzcd}[column sep = small]
      \O_{\gamma}(\phi)
      \arrow[rrr, dashed, "\O_{\gamma}(\G)"]
      \arrow[d, equal]
      & & & \O_{\gamma}(\psi) \\
      D^{\otimes}(\phi)
      \arrow[r, "q"']
      & {(D^{\otimes} \hat \otimes D^{\otimes})[\phi_1, \phi_2]}
      \arrow[r]
      & (D \bepl D)(\psi)
      \arrow[r, "\gamma"']
      & D(\psi)
      \arrow[u, equal]
    \end{tikzcd}
  \end{equation}
  By inducting, this is gives a strong monoidal functor defined on any morphism.
  Hence $\O_{\gamma}$ is an \(\plus{\M}\)-op.
\end{proof}

\goodbreak
\begin{lemma}\label{op-to-monoid}
  Conversely, if $\O$ is an $\M^+ \mddash \mathcal{O}p$ for a prehereditary UFC $\M$, then there is a canonical monoid structure consisting of:
  \begin{enumerate}
    \item A basic element representation $D_\O: \el(\nu_\M) \to \C$ obtained from the objects and basic isomorphisms in the plus construction.
    \item A multiplication $\gamma_\O: D_\O \bepl D_\O \To D_\O$ obtained from the ``gamma-morphisms'' of the plus construction.
  \end{enumerate}
\end{lemma}
\begin{proof}
  By identifying $\Iso(\M^+)$ and $\el(\rho_\M)$, we can define a basic element representation $D_\O: \el(\nu_\M) \to \C$ by $D_\O(\phi) := \O(\phi)$ and $D_\O(\sigma \biact \sigma') := \O(\sigma \biact \sigma')$.
  For the multiplication structure, observe that for every basic morphism $\G: \phi_1 \otimes \phi_2 \to \phi_1 \circ \phi_2$ in $\M^+$, the corresponding morphism $\O(\G): \O(\phi_1) \otimes \O(\phi_2) \to \O(\psi)$ in $\C$ factors through $(\O \hat \otimes \O)[\phi_1, \phi_2] \to \O(\psi)$ because of the equivariance between $\phi_1$ and $\phi_2$.
  Writing $\O(\psi) = \O(\gamma[\phi_1, \phi_2])$, we see that these morphisms assemble into a natural transformation $\bar \gamma$:
  \begin{equation}
    \begin{tikzcd}
      \el(\beta)
      \arrow[rd, "\el(\gamma_\M)"']
      \arrow[rr, "D_\O^{\otimes} \hat \otimes D_\O^{\otimes}"]
      & {}
      \arrow[d, "\bar \gamma", shorten >=0.1cm, shorten <=0.1cm, Rightarrow]
      & \C \\
      & \el(\nu_\M)
      \arrow[ru, "D_\O"']
      &
    \end{tikzcd}
  \end{equation}
  Because $D \bepl D$ is a left Kan extension, this induces a unique natural transformation $\gamma_\O: D \bepl D \To D$ such that $\bar \gamma$ is equal to $\eta$ followed by a whiskering of $\gamma_{\O}$.
\end{proof}

\begin{thm}\label{main-thm}
  If a Feynman category $\F$ can be expressed as a plus construction $\F = \M^+$ for a UFC $\M$, then the constructions of Lemmas \ref{monoid-to-op} and \ref{op-to-monoid} provide a one-to-one correspondence between non-unital $\bepl$-monoids in the category of basic element representations of the form $D: \el(\nu_\M) \to \C$ and $\F \mddash \ops$.
\end{thm}
\begin{proof}
  Both $\M^+ \mddash \ops$ and $\bepl$-monoids have equivariance properties.
  For $\M^+ \mddash \ops$, this is encoded in the category $\M^+$.
  For $\bepl$-monoids, it is encoded in the plethysm $\bepl$.
  With this in mind, it is clear that starting with a monoid $\gamma$, then sending it to an op $\O_\gamma$, then sending it to a monoid $\gamma_{\O_\gamma}$ does not lose any information.
  Likewise, starting with an op $\O$, then sending it to a monoid $\gamma_{\O}$, then back to an op $\O_{\gamma_\O}$ also does not lose any information.
  This establishes the desired equivalence.
\end{proof}

\begin{cor}
  There is a one-to-one correspondence between unital $\bepl$-monoids in the category of basic element representations of the form $D: \el(\nu_\M) \to \C$ and $\plusgcp{\M} \mddash \ops$.
\end{cor}
\begin{proof}
  By the definition of $V_\eta$ as a left Kan extension, a unit $H: V_\eta \To D$ is the same thing has a natural transformation $\tilde H: T \To D \circ \el(\eta_0)$.
  In other words, it is a family of pointings $u_\sigma: 1_{\C} \to D(\eta_0(\sigma))$ running over all basic isomorphisms.
  The naturality of $H$ and $\tilde H$ is the same thing as the groupoid compatibility of $u_\sigma$.
  The other compatibility conditions follow from the left and right unit constraints.
\end{proof}

\begin{cor}
  There is a one-to-one correspondence between $\epl$-monoids in the category of element representations of the form $D: \el(\rho_\arbcat) \to \C$ and $\catplus{\arbcat} \mddash \ops$.
\end{cor}
\begin{proof}
  This follows from a simpler version of the arguments in Lemma~\ref{monoid-to-op} and Lemma~\ref{op-to-monoid}, and Theorem~\ref{main-thm} that uses arbitrary morphisms instead of basic morphisms.
\end{proof}

\subsection{Examples of plethysm products}

We will work out the consequences of the main theorem in three cases: operads, properads, and hyper modular operads.
In each case, we will unpack the definitions and compare them to existing definitions.

\subsubsection{Operads}\label{sec:operads-special-case}
The Feynman category for operads can be expressed as the following plus construction:
\begin{equation}
  \operads \cong \FinSet^+
\end{equation}

\begin{specialcase}[\bf Element representation]
  By definition, a functor $\O: \el(\nu_{\FinSet}) \to \mathcal C$ assigns an object in $\C$ to each morphism $X \to \{y\}$.
  Since the groupoid of singletons has exactly one morphism between any of its objects, there is no harm in assuming that $\O(X \to \{y\})$ is equal to the same object in $\C$ for any choice of singleton $\{y\}$.
  In that case, this is the same data as a species in the sense of Joyal~\cite{species}.
\end{specialcase}

\begin{specialcase}[\bf Plethysm]
  Let $\O$ and $\P$ be two basic element representations of the form $\el(\nu_{\FinSet}) \to \mathcal C$.
  Then the plethysm product evaluated at $t_S: S \to pt$ is the following colimit:
  \begin{equation}
    (\O \bepl \P)(t_S)
    \cong \ucolim_{\gamma[f, g] \downarrow t_S} \O^{\otimes}(f) \hat{\otimes} \P(g)
  \end{equation}
  
  Just by definition, this colimit has a family of morphisms $\O(f) \hat \otimes \P(g) \to (\O \bepl \P)(S)$ for each equivalence class $[f,g]$.
  If we fix $f$ in the equivalence class $[f,g]$, then $g$ is uniquely determined:
  \begin{equation}
    S \overset{f}{\to} T \overset{!\exists t_T}{\to} pt 
  \end{equation}
  We can think of the pair $(f,g)$ as a ``$T$-partition'' $S = \bigcup_{t \in T} f^{-1}(t)$ of $S$ while the class $[f,g]$ represents an unindexed partition of $S$.
  Note that we \emph{do allow} empty sets in a partition.

  It is often more useful to work with an indexed partition even if the particular choice of indexing does not matter.
  For instance, set $T = \{1, \ldots, n\}$ and consider a partition $S = \bigcup_{i=1}^n S_i$.
  Using ``species notation'' $\O(S) = \O(S \to pt)$, we have the following component:
  \begin{equation}
    (\O(S_1) \otimes \ldots \otimes \O(S_n)) \otimes \P(\{S_1, \ldots, S_n\}) \to (\O \bepl \P)(S)
  \end{equation}
  By equivariance, such a component is invariant under commutativity constraints:
  \begin{equation}
    \begin{tikzcd}[sep = small]
      {(\O(S_1) \otimes \ldots \otimes \O(S_n)) \otimes \P(\{S_1, \ldots, S_n\})} \arrow[dd, "C \otimes \id"'] \arrow[rd] & \\
      & (\O \bepl \P)(S) \\
      {(\O(S_{\sigma(1)}) \otimes \ldots \otimes \O(S_{\sigma(n)})) \otimes \P(\{S_1, \ldots, S_n\})} \arrow[ru] &
    \end{tikzcd}
  \end{equation}
\end{specialcase}

\begin{specialcase}[\bf Unit]
  Consider basic element representations $\el(\nu_{\FinSet}) \to \C$.
  Here, the bimodule $\nu_{\Gpd} = \nu_{\Iso(\FinSet)}$ is defined so that $\nu_{\Iso(\FinSet)}(A,B) = \rho_{\FinSet}(A,B)$ if both $A$ and $B$ are singleton and is empty otherwise.
  Taking the left Kan extension, we get the element representation $V_{\eta}: \el(\rho_{\FinSet}) \to \C$ such that $V_{\eta}(f) = 1_{\mathcal C}$ when $f$ is a map between two singletons and $V_{\eta}(f) = 0_{\mathcal C}$ otherwise.
  If we view these as species, we get
  \begin{equation}
    V_{\eta}(S) =
    \begin{cases}
      1_{\C}, & S \text{ is a singleton} \\
      0_{\C}, & \text{otherwise}
    \end{cases}
  \end{equation}
\end{specialcase}

\subsubsection{Reduced Properads}\label{sec:reduced-properads-special-case}

If we let $\properads_{nd}$ be the Feynman category of directed corollas with at least one in-tail and one out-tail, then we can use \cite{KMoManin} to express it as a plus construction:
\begin{equation}
  \properads_{nd} = \text{nd-}\Cospan^+
\end{equation}
We will show how this leads to a monoid definition of a reduced properad.

\begin{rmk}
\label{rmk:reduced}
The assumption of being reduced is technical and can be resolved in several ways which will be further addressed in future work.
The need to restrict comes from the fact that in a strict monoidal category $\C$, the ring $\Hom(1_\C, 1_\C)$ is always commutative by the Eckmann-Hilton argument.
We thank Jan Steinebrunner for pointing this out. 

In the language of bi--modules, the
canonical bimodule $\rho_{\Cospan}$ is factorizable though and in general there is no issue with factoring the ``zero-to-zero'' part of a bimodule as we saw in Remark~\ref{rmk:zero-to-zero}.
The difference becomes visible in the plus construction where in the definition of properads as a functor from the Feynman category $\properads$ corepresenting properads, the corollas which do not have any legs are not commutative, but have iso-- and automorphisms among them.
If the commutativity of the operations $\mathcal{P}(\unit,\unit)$ is forced, then this passes to a quotient of the FC.
The story for PROPs is similar.
The relevant category of graphs is not commutative on corollas without flags, but the strong monoidal functors out of the FC $\props$ corepresenting PROPs take as values commutative algebras on the no--flag corollas. cf.\ \cite[Sections~2.2.2, 2.2.4, 2.9.1]{feynman}.
This is analogous to the fact that the strong monoidal functors from $(\FinSet,\amalg)$ are commutative algebras, see e.g.\ \cite{feynmanrep}.
\end{rmk}

\begin{specialcase}[\bf Element representation]
The groupoid $\el(\nu_{\text{nd-}\Cospan})$ is essentially the same thing as a category of individual corollas that are directed and labeled.
The morphism are relabelings which are contravariant on the in-flags and covariant on the out-flags.
Hence, a basic element representation $\O: \el(\nu_{\text{nd-}\Cospan}) \to \C$ is the same thing as an $\Iso(\FinSet)$-$\Iso(\FinSet)$--bimodule in the target category $\C$.
\end{specialcase}

\begin{specialcase}[\bf Plethysm]
Consider the plethysm product for two basic element representations $\O, \P: \el(\nu_{\text{nd-}\Cospan}) \to \C$ in terms of components.
Using the definition of the basic element plethysm as a left Kan extension and the definition of $\hat \otimes$, there is a morphism $\O^{\otimes}(c_0) \otimes \P^{\otimes}(c_1) \to (\O \bepl \P)(c)$ for each ``connected'' composition of cospans $(c_0,c_1)$ in the class $[c_0,c_1] \in \beta(A,C)$ that evaluates to $c$.
These morphisms are equivariant with respect to the action in the middle.
\begin{equation}
  \begin{tikzcd}[row sep = small]
    \O^{\ot}(c_0) \otimes \P^{\ot}(c_1) \arrow[dd, "\sim"'] \arrow[rd] & \\
    & (\O \bepl \P)(c) \\
    \O^{\ot}((\id \Downarrow \sigma) \cdot c_0) \otimes \P^{\ot}((\sigma^{-1} \Downarrow \id) \cdot c_1) \arrow[ru] & 
  \end{tikzcd}
\end{equation}
These can be understood as components for Vallette's definition of the plethysm product given in \cite{Vallette}.
\end{specialcase}

\begin{specialcase}[\bf Unit]
  Similar to the operad case, taking the left Kan extension yields an element representation $V_{\eta}: \el(\rho_{\FinSet}) \to \C$ such that for a cospan $c = (S \rightarrow V \leftarrow T)$ we have $V_{\eta}(c) = 1_{\mathcal C}$ when $S$, $V$, and $T$ are singletons and $V_{\eta}(c) = 0_{\mathcal C}$ otherwise.
\end{specialcase}

\subsubsection{Hyper modular operads}\label{sec:hyper-modular-special-case}

In this section, we will describe hyper (modular) operad in the sense of Getzler and Kapranov~\cite{GKmodular}.
More or less by definition, the Feynman category of hyper modular operads is $\modular^+$ where $\modular$ is the Feynman category of modular operads.

\begin{specialcase}[\bf Element representation]
The groupoid $\Gpd_{\modular}$ has genus-labeled aggregates as objects and the symmetries of these corollas as the morphisms.
An element $x \in \rho_{\modular}(A, B)$ is a genus-labeled graph built out of the aggregate $A$ and whose connected components form the aggregate $B$.
In particular, the elements of $\nu_{\modular}$ are the connected genus-labeled graphs, so a functor $\O: \el(\nu_{\modular}) \to \C$ would be a hyper modular operad in $\C$.
\end{specialcase}

\begin{specialcase}[\bf Plethysm]
  First note that a composable pair $\rho_{\modular} \pl \rho_{\modular}$ can be pictured as something that looks like this:
  
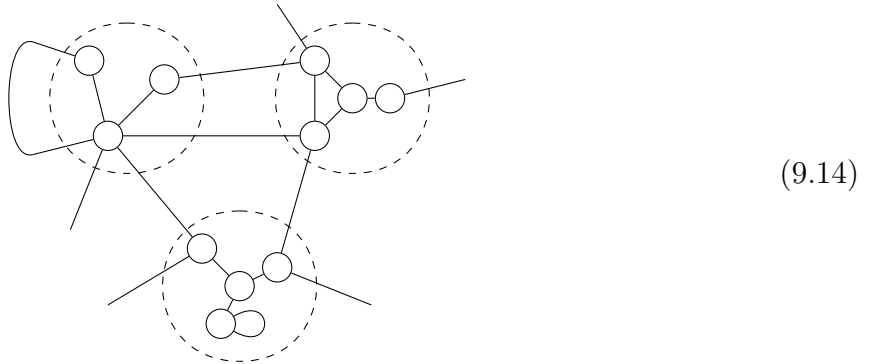
\begin{equation}
    \begin{tikzpicture}[baseline={(current bounding box.center)}]
      \node [style=White] (0) at (1.5, 1) {};
      \node [style=White] (1) at (1, 1.5) {};
      \node [style=White] (2) at (1, 0.5) {};
      \node [style=White] (3) at (2, 1) {};
      \node [style=White] (4) at (-2, 1.5) {};
      \node [style=White] (5) at (-1, 1.25) {};
      \node [style=White] (6) at (-1.75, 0.5) {};
      \node [style=White] (7) at (-0.5, -1) {};
      \node [style=White] (8) at (-0.25, -2) {};
      \node [style=White] (9) at (0, -1.5) {};
      \node [style=White] (10) at (0.5, -1.25) {};
      \node [style=none] (12) at (-1.5, 2) {};
      \node [style=none] (13) at (-1.5, 0) {};
      \node [style=none] (14) at (1.5, 2) {};
      \node [style=none] (15) at (1.5, 0) {};
      \node [style=none] (16) at (0, -0.5) {};
      \node [style=none] (17) at (0, -2.5) {};
      \node [style=none] (18) at (-2.75, 1.75) {};
      \node [style=none] (19) at (0.5, 2.25) {};
      \node [style=none] (20) at (3, 1.25) {};
      \node [style=none] (21) at (-1.75, -1.75) {};
      \node [style=none] (22) at (1.75, -1.75) {};
      \node [style=none] (23) at (-2.75, 0.25) {};
      \node [style=none] (24) at (-2.25, -0.75) {};
      \draw (1) to (2);
      \draw (1) to (0);
      \draw (2) to (0);
      \draw (0) to (3);
      \draw (4) to (6);
      \draw (5) to (6);
      \draw (8) to (9);
      \draw (9) to (7);
      \draw [in=30, out=-30, loop] (8) to ();
      \draw (9) to (10);
      \draw [dashed, bend right=90, looseness=1.75] (12.center) to (13.center);
      \draw [dashed, bend left=90, looseness=1.75] (12.center) to (13.center);
      \draw [dashed, bend right=90, looseness=1.75] (14.center) to (15.center);
      \draw [dashed, bend left=90, looseness=1.75] (14.center) to (15.center);
      \draw [dashed, bend right=90, looseness=1.75] (16.center) to (17.center);
      \draw [dashed, bend left=90, looseness=1.75] (16.center) to (17.center);
      \draw (2) to (10);
      \draw (6) to (2);
      \draw (5) to (1);
      \draw (6) to (7);
      \draw (19.center) to (1);
      \draw (20.center) to (3);
      \draw (22.center) to (10);
      \draw (21.center) to (7);
      \draw (18.center) to (4);
      \draw (6) to (23.center);
      \draw (6) to (24.center);
      \draw [bend left=255, looseness=0.75] (18.center) to (23.center);
    \end{tikzpicture}
  \end{equation}
  Let $\gamma_0: \beta_{\modular} \To \nu_{\modular}$ denote the basic morphism associated with the multiplication $\gamma: \rho_{\modular} \pl \rho_{\modular} \To \rho_{\modular}$.
  Up to equivalence, an element of $\beta_{\modular}$ consists of a single ``big graph'' $b \in \el(\beta_{\modular})$ together with a collection of ``small graphs'' $s_v \in \el(\nu_{\modular})$ indexed over the vertices of $b$.
  Let $\O, \P: \el(\rho_{\modular}) \to \C$ be two element representations:
  Hence the plethysm product is the following colimit:
  \begin{equation}
    (\O \bepl \P)(x)
    \cong \ucolim_{\gamma\left[ \{s_v\}, b \right] \downarrow x}
    \left( \bigotimes_{v \in \mVert(b)} \O(s_i) \right) \hat{\otimes} \P(b)
  \end{equation}
\end{specialcase}

\begin{specialcase}[\bf Unit]
  There is a canonical bimodule morphism $\eta: \rho_{\Gpd_{\modular}} \To \rho_{\modular}$ which embeds the genus-labeled aggregates into the genus-labeled graphs.
  This is clearly hereditary.
  This determines a basic element representation $V_\eta: \el(\nu_{\M}) \to \C$ such that $V_{\eta}(x) = 1_\C$ if $x \in \nu_{\M}$ is a corolla and is $0_{\C}$ otherwise.
\end{specialcase}

\subsection{Decorations for Feynman categories}

In the language of Feynman categories, we will establish the formula connecting enrichments with decorations $(\F_D)^+ \cong (\F^+)_{\dec D}$ that was announced as \cite[Proposition 4.12]{feynmanrep}.

\begin{thm}
\label{thm:feynmandecplus}
  For an $\F$-op in a Cartesian closed category, we have the following equivalence between decoration and enrichment:
  \begin{equation}
    \label{eq:feynmanrep}
    (\F_D)^+ \cong (\F^+)_{\dec D}
  \end{equation}
\end{thm}
\begin{proof}
  On the left hand side, an object is an element $a \in \Hom_{\F_{D}}(X, Y)$ of a set which is defined to be the coproduct $\coprod_{\phi \in \Hom_{\F}(X,Y)} D(\phi)$.
  Since the target category is Cartesian closed, there exists an element $a_{\phi} \in D(\phi)$ such that $a$ is the image of $a_{\phi}$ under the component $D(\phi) \to \Hom_{\F_{D}}(X, Y)$.
  Now define a correspondence between $a$ on the left and $(\phi, a_{\phi})$ on the right.
  Conversely, send the pair $(\phi, a_{\phi})$ with $\phi \in \Hom_{\F}(X, Y)$ to the image of $a_{\phi}$ under $D(\phi) \to \Hom_{\F_D}(X, Y)$.
  This establishes a bijection between the two objects.

  This correspondence extends to the morphisms in the obvious way.
  For example, consider the gamma-morphisms:
  \begin{equation}
    \begin{tikzcd}[row sep = tiny]
      {a \otimes b} && {(\phi, a_{\phi}) \otimes (\psi, b_{\psi})} \\
      & {\text{corresponds to}} \\
      {a \circ b} && {(\phi \circ \psi, (a \circ b)_{\phi \circ \psi})}
      \arrow["\gamma"', from=1-1, to=3-1]
      \arrow["\gamma", from=1-3, to=3-3]
    \end{tikzcd}
  \end{equation}
  Moreover, it is clear that their monoidal structures agree under this bijection.
  Hence the two monoidal categories are equivalent.
\end{proof}

\begin{rmk}
  In the bimodule language, this result is analogous to Theorem~\ref{thm:pluseltbasic} with $\el(\chi_D)$ as the analog of the left hand side is and $\el(D)$ the analog of the right hand side.
\end{rmk}

\bibliography{Plethysm}
\bibliographystyle{halpha}

\end{document}